\documentclass{amsart}
\usepackage{amssymb, amsmath, amsthm, graphics, comment, xspace, enumerate}
\usepackage{hyperref}
\usepackage{color}
\newtheorem{thm}{Theorem}[section]
\newtheorem{cor}[thm]{Corollary}
\newtheorem{lem}[thm]{Lemma}

\newtheorem{prop}[thm]{Proposition}
\theoremstyle{definition}
\newtheorem{defn}[thm]{Definition}
\newtheorem{example}[thm]{Example}
\theoremstyle{remark}
\newtheorem{rem}[thm]{Remark}
\numberwithin{equation}{section}

\begin{document}
\title[Multidimensional vector-valued Laplace transform...]{Multidimensional vector-valued Laplace transform and applications}

\author{Marko Kosti\' c}
\address{Faculty of Technical Sciences,
University of Novi Sad,
Trg D. Obradovi\' ca 6, 21125 Novi Sad, Serbia}
\email{marco.s@verat.net}

{\renewcommand{\thefootnote}{} \footnote{2020 {\it Mathematics
Subject Classification.} 44A10, 44A30, 47D99.
\\ \text{  }  \ \    {\it Key words and phrases.} Multidimensional vector-valued Laplace transform, abstract Volterra integro-differential inclusions, sequentially complete locally convex spaces, holomorphic functions of several variables.}}

\begin{abstract}
In this paper, we introduce and analyze multidimensional vector-valued Laplace transform of functions with values in sequentially complete locally convex spaces. A great number of our results seem to be new even for the functions with values in Banach spaces. We provide several new applications of multidimensional vector-valued Laplace transform to the abstract Volterra integro-differential inclusions with multiple variables, as well.
\end{abstract}
\maketitle

\section{Introduction and preliminaries}

In order to prove a special case of the central limit theorem (1785), Pierre-Simon, Marquis de Laplace, introduced a new type of integral transform for applications to probability (\cite{ps laplace}). He used sometimes what is called the $z$-transform nowadays and he was the predecessor for something that is called Mellin's transform today. The Laplace transform was systematically developed by Oliver Heaviside and now it is commonly used by electrical engineers for solutions of various electronic circuit problems. We all agree
that P.-S. Laplace was not the first who invented his transform, in any way. The generating functions, used by Laplace, have been already known from the the work of A. de Moivre on combinatorics/probability (1718), who can be called the father of modern  harmonic analysis. It is also well-known that many historians of mathematics emphasize that the modern form of 
Laplace transform was probably used for the first time by L. Euler in 1744. Concerning the Laplace transform and its applications, we can warmly recommend reading the research monographs \cite{a43} by W. Arendt et al., \cite{cohen} by 
A. M. Cohen, \cite{dech} by G. Doetsch, \cite{debnathb} by
L. Debnath and D. Bhatta, \cite{ditkin} by V. A.
Ditkin and A. P. Prudnikov, \cite{holbruk} by
J. G. Holbrook, \cite{knjigaho}-\cite{FKP}
by M. Kosti\' c, \cite{prus} by J. Pr\"uss,\cite{wider}-\cite{wider1} by
D. V. Widder and \cite{x263}
by
T.-J. Xiao and J. Liang.

The investigation of twodimensional scalar-valued Laplace transform starts probably with the works of D. L. Bernstein \cite{bern0}-\cite{bern}, J. C. Jaeger \cite{jager} (1939--1941) and L. Amerio \cite{amerio} (1940); for more details about the multidimensional scalar-valued Laplace transform and its aplications to (fractional) partial integro-differential equations, we refer the reader to  research articles \cite{iranac1,bern1,debnath,debnath1,elta,elta1,hant,lemn,rog,saa,sed}, doctoral dissertations \cite{ahmud,thesis0,thesis1} and references quoted therein. 

As mentioned in the abstract, the main aim of this paper is to introduce and thoroughly analyze the multidimensional vector-valued Laplace transform of functions with values in sequentially complete locallly convex spaces. We provide several new structural characterizations and  properties of the multidimensional vector-valued Laplace transform and furnish several new applications of multidimensional vector-valued Laplace transform to the abstract Volterra integro-differential inclusions in locally convex spaces.

The structure of paper can be briefly described as follows. We first collect some preliminaries and fix the notation used throughout the paper; after that, we recall the basic facts about the integration of functions with values in locally convex spaces (Subsection \ref{zoki}) and vector-valued holomorphic functions of several complex variables (Subsection \ref{zver}). Our main structural results are clarified in Section \ref{sdf}: The regions of convergence of multidimensional vector-valued Laplace transform are examined in Subsection \ref{regions} (the main results of this section are Theorem \ref{dl}, Theorem \ref{djura} and Theorem \ref{sas}), the  operational and analytical properties of multidimensional vector-valued Laplace transform are examined in Subsection \ref{regions1} (the main results of this subsection are Theorem \ref{opera}, Theorem \ref{cit}, Proposition \ref{analyticala} and Proposition \ref{lapkonv}), the multidimensional vector-valued Laplace transform of holomorphic functions of several variables is examined in Subsection \ref{neubrander} (the main results of this subsection are Proposition \ref{tauber}, Proposition \ref{holo} and Theorem \ref{analytic}; these results are new even for scalar-valued functions) and the  uniqueness theorem for multidimensional vector-valued Laplace transform is examined in Subsection \ref{napolje} (the main result of this subsection is Theorem \ref{sasu}). Section \ref{dzubre}, which is divided into three subsections, is reserved for the applications of multidimensional vector-valued Laplace transform to the abstract Volterra integro-differential inclusions with multiple variables (with the exception of our recent research article \cite{axioms}, this topic seems to be completely uninvestigated in the existing literature, which strongly influenced us to write this paper): Subsection \ref{nulice} examines the holomorphic solutions of abstract fractional partial differential inclusions with Riemann-Liouville and Caputo derivatives, Subsection \ref{sajkote} examines the abstract degenerate Volterra integro-differential inclusions with multiple variables and Subsection \ref{skockao} examines the abstract higher-order differential equations with multiple variables. The final section of paper is reserved for some conclusions and final observations about the considered topics. With the exception of Theorem \ref{cit}, all structural results of ours seem to be new even for the functions with values in complex Banach spaces.\vspace{0.1cm}

\noindent {\bf Notation and preliminaries.} If $Y$ is a Hausdorff sequentially complete locally convex space\index{sequentially complete locally convex space!Hausdorff} over the field of complex numbers, then we simply write that $Y$ is an SCLCS.
If $X$ and $Y$ are Hausdorff locally convex spaces over the field ${\mathbb C}$, 
then the abbreviation $\circledast$ ($\circledast_{Y}$) stands for the fundamental system of seminorms\index{system of seminorms} which defines the topology of $X$ ($Y$), $L(X,Y)$ denotes the space consisting of all continuous linear mappings\index{continuous linear mapping} from $X$ into $Y$ and $L(X)\equiv L(X,X)$; if $\mathcal B$ is the family of bounded subsets\index{bounded subset} of $X$, then this space carries the Hausdorff locally convex topology induced
by the calibration $(p_{B})_{B\in {\mathcal B}}$ of seminorms on $L(X,Y)$, where $p_B(T):=\sup_{x\in B}p(Tx)$, $p\in\circledast$, $B\in\mathcal B$, $T\in L(X,Y)$.
It is well known that the space $L(X,Y)$ is sequentially complete provided that $X$ is barreled\index{barreled space}; by $X^{\ast}$ we denote the dual space of $X$ equipped with the strong topology (cf. \cite{jarc}, \cite{meise} and \cite{trev} for more details about topological vector spaces and locally convex spaces). For further information concerning the multivalued linear operators in sequentially complete locally convex spaces (MLOs) and solution operator families subgenerated by them, we refer the reader to \cite{FKP}. 

The Gamma function\index{function!Gamma} will be denoted by $\Gamma(\cdot)$ and the principal branch will be always used to take the powers; define $g_{\zeta}(t):=t^{\zeta-1}/\Gamma(\zeta)$ and
$0^{\zeta}:=0$ ($\zeta>0,$ $t>0$). If $0<\alpha \leq \pi,$ then we define $\Sigma_{\alpha}:=\{z\in {\mathbb C} \setminus \{0\} : |\arg(z)|<\alpha\}.$
Given the numbers $s\in\mathbb R$ and $m\in {\mathbb N},$ we set $\lceil s\rceil:=\inf\{l\in\mathbb Z:s\leq l\}$, ${\mathbb N}_{m}:=\{1,...,m\} $ and ${\mathbb N}_{m}^{0}:=\{0,1,...,m\} .$ Unless stated otherwise, we will always assume henceforth that $X$ is an SCLCS and $n\in {\mathbb N}.$

If $u\in L_{loc}^{1}([0,\infty)^{n}:X)$, $j\in {\mathbb N}_{n}$, $\alpha_{j}>0$ and $a\in L_{loc}^{1}([0,\infty)^{n})$, then we define 
\begin{align*}
J_{t_{j}}^{\alpha_{j}} u\bigl(x_{1},...,x_{j-1},x_{j},x_{j+1},...x_{n}\bigr)&:=\int^{x_{j}}_{0}g_{\alpha_{j}}\bigl( x_{j}-s\bigr)u\bigl(x_{1},...,x_{j-1},s,x_{j+1},...x_{n}\bigr)\, ds
\end{align*}
and
\begin{align*}                    
\bigl(a\ast_{0}u\bigr)(x) := \int^{x_{1}}_{0}\cdot ....\cdot \int^{x_{n}}_{0}a\bigl( x_{1}-s_{1},...,x_{n}-s_{n}\bigr)u\bigl(s_{1},...,s_{n}\bigr)\, ds_{1} ...\, ds_{n},
\end{align*}
for any $ x=\bigl(x_{1},...,x_{n}\bigr)\in [0,\infty)^{n}.$ The convolution product $\ast_{0}$ introduced in the last formula is sometimes called the convolution product of Faltung (see, e.g., \cite[p. 232]{debnath}). In general SCLCSs, the convolution product $\ast_{0}$ is well-defined if $a\in C([0,\infty)^{n})$ and $f\in L_{loc}^1([0,\infty)^{n}:X)$, or  
$a\in L_{loc}^1([0,\infty)^{n})$ and $f\in C([0,\infty)^{n}:X)$, when we have $a\ast_{0}f\in  L_{loc}^1([0,\infty)^{n}:X);$ cf. Lemma \ref{bnm1} below. In Fr\' echet spaces, the convolution product $\ast_{0}$ is well-defined if $a\in  L_{loc}^1([0,\infty)^{n})$ and $f\in L_{loc}^1([0,\infty)^{n}:X)$, when we have $a\ast_{0}f\in  L_{loc}^1([0,\infty)^{n}:X);$ cf. Lemma \ref{ft} and  Proposition \ref{lapkonv} below. 

\subsection{Integration\index{integration in locally convex spaces} of functions with values in locally convex spaces.}\label{zoki}

Let $X$ be an SCLCS,
let $\Omega$ be a locally compact\index{space!locally compact} and separable metric space,\index{space!separable metric} and let $\mu$ be a locally finite Borel measure\index{measure!locally finite Borel} defined on $\Omega$. We refer the reader to \cite[Definition 1.1.1]{FKP} for the notion of a simple [(strongly) $\mu$-measurable] function $f: \Omega\to X.$ The following concept of integration has been proposed by C. Martinez and M. Sanz (see \cite[pp.\,99--102]{martinez}):

\begin{defn}\label{Kint}
\begin{itemize}
\item[(i)] Let $K\subseteq\Omega$ be a compact set, and let a function $f\colon K\to X$ be strongly measurable.
Then we say that $f(\cdot)$ is $\mu$-integrable if there exists a sequence $(f_k)_{k\in\mathbb N}$ of simple functions such that $\lim_{k\to\infty}f_k(t)=f(t)$ a.e. $t\in K$ and for all $\epsilon>0$ and each $p\in\circledast$ there exists an integer  $k_0=k_0(\epsilon,p)$ such that
\begin{equation*}
\int_Kp\bigl(f_k-f_m\bigr)d\mu\leq\epsilon\quad \bigl(m,k\geq k_0\bigr).
\end{equation*}
In this case, we define
\[
\int_Kf\,d\mu:=\lim_{k\to\infty}\int_Kf_k\, d\mu.
\]
\item[(ii)] A function $f\colon\Omega\to X$ is said to be locally $\mu$-integrable if, for every compact set $K\subseteq\Omega$, the restriction $f_{|K}\colon K\to X$ is $\mu$-integrable.
\item[(iii)] A function $f\colon\Omega\to X$ is said to be $\mu$-integrable\index{function!$\mu$-integrable} if $f(\cdot)$ is locally $\mu$-integrable and
\begin{equation}\label{4.43}
\int_{\Omega}p(f)\, d\mu<\infty,\quad p\in\circledast.
\end{equation}
In this case, we define
\[
\int_{\Omega}f\,d\mu:=\lim_{k\to\infty}\int_{K_k}f\,d\mu,
\]
where $(K_k)_{k\in\mathbb N}$ is an expansive sequence of compact subsets of $\Omega$ with the property that $\bigcup_{k\in\mathbb N}K_k=\Omega$.
\end{itemize}
\end{defn}

The definition in (ii) is meaningful and does not depend on the choice of a sequence $(K_k)_{k\in\mathbb N}$.
Furthermore, for each seminorm $p\in\circledast$ the function $p(f(\cdot))$ is $\mu$-integrable and we have 
\begin{equation}\label{blue-note-0}
p\bigg(\int_{\Omega}f\,d\mu\bigg)\leq\int_{\Omega}p(f)\, d\mu.
\end{equation}
Any continuous function $f\colon\Omega\to X$ satisfying \eqref{4.43} is $\mu$-integrable and it is not difficult to prove that
the $\mu$-integrability of a function $f\colon K\to X$, resp. $f\colon\Omega\to X $, implies that for each $x^*\in X^*$, the function $\langle x^*,f(\cdot)\rangle$ is $\mu$-integrable and
\begin{equation}\label{sex-pistols}
\bigg\langle x^*,\int_Kf\,d\mu\bigg\rangle=\int_K \bigl \langle x^*,f\bigr \rangle \, d\mu,\text{ resp. }\bigg\langle x^*,\int_{\Omega}f\,d\mu\bigg\rangle=\int_{\Omega}\bigl \langle x^*,f\bigr \rangle \, d\mu;
\end{equation}
furthermore, the dominated convergence theorem holds in our framework as well as the following result:

\begin{lem}\label{bnm}
Let $Y$ be an \emph{SCLCS}, and let $T\colon D(T)\subseteq X\to Y$ be a closed linear operator.
If $f\colon\Omega\to D(T)$ is $\mu$-integrable\index{function!$\mu$-integrable} and $Tf\colon\Omega\to Y$ is likewise $\mu$-integrable, then $\int_{\Omega}f\,d\mu\in D(T)$ and 
\begin{equation*}
T\int_{\Omega}f\,d\mu=\int_{\Omega}Tf\,d\mu.
\end{equation*} 
\end{lem}

Unless specified otherwise, we will always assume henceforth that $\mu=m$ is the Lebesgue measure on ${\mathbb R}^{n}.$
Then, in place of terms ``$\mu$-measurable'', ``locally $\mu$-integrable'' and ``$\mu$-integrable'', we use the terms ``(Lebesgue) measurable'', ``locally (Lebesgue) integrable'' and ``(Lebesgue) integrable'', respectively.

The following lemma can be proved as in the one-dimensional case (see \cite[Theorem 1.1.4(i)-(ii)]{FKP}; here and hereafter, by $L_{loc}^1(\Omega:X)$ we denote the vector space of all locally integrable functions from $\Omega$ into $X$):

\begin{lem}\label{bnm1}
\begin{itemize}
\item[(i)] Suppose that $g\in C(\Omega)$ and $f\in L_{loc}^1(\Omega:X)$.
Then $g\cdot f\in L_{loc}^1([0,\infty):X)$.
\item[(ii)] If $g\in L_{loc}^1(\Omega)$ and $f\in C(\Omega:X)$, then $g\cdot f\in L_{loc}^1(\Omega:X)$.
\end{itemize}
\end{lem}

Unfortunately, the Fubini--Tonelli theorem does not hold in general SCLCS; if $X$ is a Fr\'echet space, then the concept of integration introduced in Definition \ref{Kint} is equivalent with the concept of Bochner integration and we have the following results (cf. also \cite[Theorem 1.1.9]{a43} for the Banach space setting):

\begin{lem}\label{ft}
Suppose that $X$ is a Fr\'echet space. Then the following holds:
\begin{itemize}
\item[(i)] If $f: \Omega \rightarrow X$ is a Lebesgue measurable function and the function $p(f(\cdot))$ is Lebesgue integrable for each seminorm $p\in \circledast$, then $f(\cdot)$ is Lebesgue integrable.
\item[(ii)] $($The  Fubini--Tonelli theorem$)$ Let $I=I_{1}\times I_{2}$ be a rectangle in ${\mathbb R}^{2}$,  let $F: I \rightarrow X$ be a Lebesgue measurable function, and let 
$$
\int_{I}p\bigl(f(t,s)\bigr)\, dt\, ds<+\infty,\quad p\in \circledast.
$$
Then the repeated integrals
$$
\int_{I_{2}}\int_{I_{1}}f(t,s)\, dt\, ds\ \ \mbox{ and }\ \ \int_{I_{1}}\int_{I_{2}}f(t,s)\, ds\, dt
$$
exist, they are equal and coincide with the double integral $\int_{I}f(t,s)\, dt\, ds.$
\end{itemize}
\end{lem}

In general SCLCSs, we can only prove that the existence of integral $
\int_{I}f(t,s)\, dt\, ds$ and the repeated integrals
$
\int_{I_{2}}\int_{I_{1}}f(t,s)\, dt\, ds$ \mbox{ and }$ \int_{I_{1}}\int_{I_{2}}f(t,s)\, ds\, dt
$ implies their equality; see, e.g., the proof of Theorem \ref{analytic} below.

\subsection{Vector-valued analytic functions of several complex variables}\label{zver}

For the basic source of information on scalar-valued analytic functions of several variables, we refer the reader to the research monographs
\cite{alexander,adachi,gun,jar,hermander,kaup,range}. Albeit contain  some results on analytical properties of functions of several variables, the research monographs \cite{dina1,dina2,mujica} are primarily devoted to the study of analytic functions depending on infinitely number of variables. For more details about analytic vector-valued functions depending on one complex variable (several complex variables), we refer the reader to the list of references quoted in the recent research article \cite{kruse} by K. Kruse. In that paper, the author has shown several important results concerning the analytic functions of several variables with values in locally complete locally convex spaces (let us recall that any sequentially complete locally convex space is locally complete). 

A function $f\colon\Omega\to X$, where $\Omega$ is an open subset of ${\mathbb C}^{n}$, is said to be analytic if for each $\lambda=(\lambda_{1},...,\lambda_{n}) \in \Omega$ there exists $\epsilon>0$ such that $L(\lambda,\epsilon):=\{(z_{1},..,z_{n})\in {\mathbb C}^{n} : \max\{ |z_{i}-\lambda_{i}| : 1\leq i\leq n\}<\epsilon \}\subseteq \Omega$  
and
$$
f(z)=\sum_{v_{1}=0}^{+\infty}...\sum_{v_{n}=0}^{\infty}a_{v_{1},...,v_{n}}\bigl( z_{1}-\lambda_{1} \bigr)^{v_{1}}\cdot .... \cdot \bigl( z_{n}-\lambda_{n} \bigr)^{v_{n}},\quad z\in L(\lambda,\epsilon),
$$
for some elements $a_{v_{1},...,v_{n}}\in X$ satisfying that for each seminorm $p\in \circledast$ and $r\in [0,\epsilon)$ we have
$$
\sum_{(v_{1},...,v_{n})\in {\mathbb N}_{0}^{n}}p\bigl(a_{v_{1},...,v_{n}} \bigr) r^{v_{1}+...+v_{n}}<+\infty .
$$ 
As in the one-dimensional setting, the mapping $\lambda\mapsto f(\lambda)$, $\lambda\in\Omega$ is analytic if and only if the mapping $\lambda\mapsto\langle x^*,f(\lambda)\rangle$, $\lambda\in\Omega$ is analytic for every $x^*\in X^*$.
If the mapping $f\colon\Omega\to X$ is analytic, then the Cauchy integral formula in polydiscs holds (see, e.g., \cite[Theorem 5.1, Theorem 5.7]{kruse}), the mapping $\lambda\mapsto f(\lambda)$, $\lambda\in\Omega$ is infinitely differentiable and we have
$$
a_{v_{1},...,v_{n}}=\frac{\partial^{(v_{1},...,v_{n})}f(\lambda)}{v_{1}!\cdot ... \cdot v_{n}!},\quad (v_{1},...,v_{n})\in {\mathbb N}_{0}^{n}.
$$  
The Osgood lemma and Hartogs theorem hold for the vector-valued analytic functions of several variables; theWeierstrass theorem also holds in our framework:   
Let $\emptyset\neq\Omega\subseteq\mathbb C^{n}$ be open and connected, and let $f_k\colon\Omega\to X$ be an analytic function $(k\in\mathbb N$). If there exists a function $f\colon\Omega\to X$ such that the sequence $(f_{k}(\cdot))$ converges uniformly to $f(\cdot)$ on compacts of $\Omega$, then $f(\cdot)$ is analytic and 
\begin{align}\label{osgud}
\partial^{(v_{1},...,v_{n})}f(\lambda)=\lim_{k\rightarrow +\infty}\partial^{(v_{1},...,v_{n})}f_{k}(\lambda),\quad \lambda \in \Omega,\ (v_{1},...,v_{n})\in {\mathbb N}_{0}^{n}.
\end{align}

\section{Basic properties of multidimensional vector-valued Laplace transform}\label{sdf}

We start this section by introducing the following notion, which is well known in the one-dimensional setting (\cite{a43}, \cite{FKP}):

\begin{defn}\label{svi}
Suppose that $X$ is an SCLCS, $f: [0,+\infty)^{n}\rightarrow X$ is a locally integrable function and $(\lambda_{1},...,\lambda_{n})\in {\mathbb C}^{n}$. If 
\begin{align*} 
F\bigl(\lambda_{1},...,\lambda_{n}\bigr)&:=\lim_{t_{1}\rightarrow +\infty;...;t_{n}\rightarrow +\infty}\int^{t_{1}}_{0}...\int^{t_{n}}_{0}e^{-\lambda_{1}s_{1}-...-\lambda_{n}s_{n}}f\bigl(s_{1},...,s_{n}\bigr)\, ds_{1}\, ...\, ds_{n}
\\& :=\int^{+\infty}_{0}...\int^{+\infty}_{0}e^{-\lambda_{1}t_{1}-...-\lambda_{n}t_{n}}f\bigl(t_{1},...,t_{n}\bigr)\, dt_{1}\, ...\, dt_{n},
\end{align*}
exists for the topology of $X,$ i.e., if for each $\epsilon>0$ and $p\in \circledast$ there exists a real number $M>0$ such that the assumptions $t_{j}>M$ for all $j\in {\mathbb N}_{n}$ imply 
$$
p\Biggl(\int^{t_{1}}_{0}...\int^{t_{n}}_{0}e^{-\lambda_{1}s_{1}-...-\lambda_{n}s_{n}}f\bigl(s_{1},...,s_{n}\bigr)\, ds_{1}\, ...\, ds_{n}- F\bigl(\lambda_{1},...,\lambda_{n}\bigr)\Biggr) <\epsilon ,
$$
then we say that the Laplace integral $({\mathcal Lf})(\lambda_{1},...,\lambda_{n}):=F(\lambda_{1},...,\lambda_{n})$ exists.
We define the region of convergence of Laplace integral $
\Omega(f)$ by
$$
\Omega(f):=\Bigl\{ \bigl(\lambda_{1},...,\lambda_{n}\bigr)\in {\mathbb C}^{n} : F(\lambda_{1},...,\lambda_{n})\mbox{ exists}\Bigr\}.
$$
Furthermore, if
for each seminorm $p\in \circledast$ we have
\begin{align*}
\int^{+\infty}_{0}...\int^{+\infty}_{0}p\Bigl(e^{-\lambda_{1}t_{1}-...-\lambda_{n}t_{n}}f\bigl(t_{1},...,t_{n}\bigr)\Bigr)\, dt_{1}\, ...\, dt_{n}<+\infty,
\end{align*}
then we say that the Laplace integral $F(\lambda_{1},...,\lambda_{n})$
converges absolutely. We define the region of absolute convergence of Laplace integral $
\Omega_{abs}(f)$ by
$$
\Omega_{abs}(f):=\Bigl\{ \bigl(\lambda_{1},...,\lambda_{n}\bigr)\in {\mathbb C}^{n} : F(\lambda_{1},...,\lambda_{n})\mbox{ converges absolutely}\Bigr\}.
$$
\end{defn}

Equivalently, if $f: [0,+\infty)^{n}\rightarrow X$ is a locally integrable function and $(\lambda_{1},...,\lambda_{n})\in {\mathbb C}^{n}$, then 
$(\lambda_{1},...,\lambda_{n}\bigr)\in \Omega_{abs}(f)$ if and only if the function 
$$
\bigl(t_{1},...t_{n}\bigr)\mapsto e^{-\lambda_{1}t_{1}-...-\lambda_{n}t_{n}}f\bigl(t_{1},...,t_{n}\bigr),\quad \bigl(t_{1},...t_{n}\bigr)\in [0,+\infty)^{n}
$$ 
is Lebesgue integrable in the sense of Definition \ref{Kint}(iii); cf. also Lemma \ref{bnm1}(i).
Using the sequential completeness of space $X$, we can prove that $\Omega_{abs}(f) \subseteq \Omega(f) $ and the assumption $(\lambda_{1},...,\lambda_{n})\in \Omega_{abs}(f)$ implies
\begin{align*} 
F\bigl(\lambda_{1},...,\lambda_{n}\bigr)=\lim_{t\rightarrow +\infty}\int^{t}_{0}...\int^{t}_{0}e^{-\lambda_{1}s_{1}-...-\lambda_{n}s_{n}}f\bigl(s_{1},...,s_{n}\bigr)\, ds_{1}\, ...\, ds_{n}.
\end{align*}
Let us notice that the assumptions $(\lambda_{1},...,\lambda_{n})\in \Omega_{abs}(f),$ $(z_{1},...,z_{n}) \in {\mathbb C}^{n}$ and $\Re \lambda_{j}=\Re z_{j}$ for $1\leq j \leq n$ imply $(z_{1},...,z_{n}) \in \Omega_{abs}(f);$ the above fact is not true if the set $\Omega_{abs}(f)$ is replaced by the set $\Omega(f);$ see, e.g., \cite[Example 1.4.2]{a43} for a simple counterexample given in the one-dimensional setting. 

In the higher-dimensional setting, there are many problems concerning the unconditional convergence of improper Riemann integrals, since it turns out that the ordinary convergence of improper Riemann integral is equivalent to its absolute convergence in some concepts. In order to avoid any confusion and misunderstanding, we would like to emphasize that this fact is not true for the improper integrals introduced in Definition \ref{svi}:

\begin{example}\label{majst}
Suppose that $\lambda_{1}=\lambda_{2}=0$ and $f(t_{1},t_{2}):=\sin(t_{1}^{2}) \cdot \sin (t_{2}^{2}),$ $t_{1}\geq 0,$ $t_{2}\geq 0.$ Using the Fubini theorem, it readily follows that
\begin{align*}
& \lim_{t_{1}\rightarrow +\infty;t_{2}\rightarrow +\infty} \int^{t_{1}}_{0}\int^{t_{2}}_{0}f\bigl(s_{1},s_{2}\bigr)\, ds_{1}\,\, ds_{2}\\&=\lim_{t_{1}\rightarrow +\infty;t_{2}\rightarrow +\infty} \int^{t_{1}}_{0}\int^{t_{2}}_{0}\sin \bigl(s_{1}^{2}\bigr) \cdot \sin \bigl(s_{2}^{2}\bigr)\, ds_{1}\,\, ds_{2}\\ & =\lim_{t_{1}\rightarrow +\infty;t_{2}\rightarrow +\infty}\Biggl[\int^{t_{1}}_{0}\sin \bigl(s_{1}^{2}\bigr) \, ds_{1} \cdot \int^{t_{2}}_{0}\sin \bigl(s_{2}^{2}\bigr) \, ds_{2} \Biggr]=\bigl(\sqrt{\pi}/4\bigr)^{2}=\pi/16.
\end{align*}
On the other hand, if $\int^{+\infty}_{0}\int^{\infty}_{0}|f\bigl(s_{1},s_{2}\bigr)|\, ds_{1}\,\, ds_{2}<+\infty,$ then the Fubini theorem would imply that the integral $\int^{+\infty}_{0}| \sin (s_{1}^{2})|\cdot  | \sin (s_{2}^{2})|\, ds_{2}$ is convergent for a.e. $s_{1}\geq 0$ and
\begin{align*}
\int^{+\infty}_{0}\int^{+\infty}_{0}& \Bigl| \sin \bigl(s_{1}^{2}\bigr) \cdot \sin \bigl(s_{2}^{2}\bigr) \Bigr| \, ds_{1}\,\, ds_{2}\\&=\int^{+\infty}_{0}\Biggl[ \int^{+\infty}_{0}\Bigl| \sin \bigl(s_{1}^{2}\bigr)\Bigr|\cdot  \Bigl| \sin \bigl(s_{2}^{2}\bigr)\Bigr|\, ds_{2}\Biggr] \, ds_{1}<+\infty.
\end{align*}
In particular, there exists $s_{1}\geq 0$ such that $s_{1}^{2}\notin {\mathbb Z}\cdot \pi$ and the integral $\int^{+\infty}_{0}| \sin (s_{1}^{2})|\cdot  | \sin (s_{2}^{2})|\, ds_{2}$ is convergent, so that the Fresnel integral $\int^{+\infty}_{0} | \sin (s_{2}^{2})|\, ds_{2}$ is also convergent, which is a contradiction.
Finally, we would like to note that $\Omega (f) =\{\lambda \in {\mathbb C} : \Re \lambda \geq 0\}^{2}$ and $\Omega_{abs} (f) =\{\lambda \in {\mathbb C} : \Re \lambda > 0\}^{2}.$ 
\end{example}

If $(\lambda_{1},...,\lambda_{n}\bigr)\in \Omega(f)$, resp. $(\lambda_{1},...,\lambda_{n}\bigr)\in \Omega_{abs}(f),$ and $x^{\ast}\in X^{\ast},$ then we have $(\lambda_{1},...,\lambda_{n}\bigr)\in \Omega(\langle x^{\ast}, f\rangle),$ resp. $(\lambda_{1},...,\lambda_{n}\bigr)\in \Omega_{abs}(\langle x^{\ast}, f\rangle).$
Further on, if $p\in \circledast,$ then $U_{p}:=\{x\in X : p(x)\leq 1\}$ is an absolutely convex zero neighborhood in $X$ and \cite[Proposition 22.14, p. 256]{meise} implies 
$
p(x)=\sup_{x^{\ast}\in U_{p}^{\circ}}|\langle x^{\ast},x\rangle |,
$ $x\in X.$ Using this equality, we can simply prove the following result:

\begin{prop}\label{sor}
If $f: [0,+\infty)^{n}\rightarrow X$ is a locally integrable function and\\ $(\lambda_{1},...,\lambda_{n})\in {\mathbb C}^{n}$, then
$(\lambda_{1},...,\lambda_{n}\bigr)\in \Omega_{abs}(f)$ if and only if for each seminorm $p\in \circledast$ we have 
$$
\int^{+\infty}_{0}...\int^{+\infty}_{0}e^{-(\Re \lambda_{1})t_{1}-...-(\Re \lambda_{n})t_{n}}  \sup_{x^{\ast}\in U_{p}^{\circ}}\Bigl| \Bigl \langle x^{\ast}, f\bigl(t_{1},...,t_{n}\bigr) \Bigr \rangle \Bigr| \, dt_{1}\, ...\, dt_{n}<+\infty.
$$
If this is the case, then 
\begin{align*} &
p\Bigl( F\bigl( \lambda_{1},...,\lambda_{n} \bigr)\Bigr)
\\& =\lim_{t\rightarrow +\infty}\sup_{x^{\ast}\in U_{p}^{\circ}}\Biggl|  \int^{t}_{0}...\int^{t}_{0}e^{-(\Re \lambda_{1})t_{1}-...-(\Re \lambda_{n})t_{n}}  \Bigl \langle x^{\ast}, f\bigl(t_{1},...,t_{n}\bigr) \Bigr \rangle  \, dt_{1}\, ...\, dt_{n}  \Biggr|.
\end{align*}
\end{prop}

We continue with the following example:

\begin{example}\label{pfk}
Suppose that $f: [0,+\infty)^{n} \rightarrow X$ is a Lebesgue measurable function and there exist real constants $\omega_{1}\in {\mathbb R},\ ...,\ \omega_{n}\in {\mathbb R},$ $\eta_{1}\in (-1,+\infty),\ ...,\ \eta_{n}\in (-1,+\infty)$ and $\zeta_{1}\in (-1,+\infty),\ ...,\ \zeta_{n}\in (-1,+\infty)$ such that, for every seminorm $p\in \circledast$, there exists a finite real constant $M_{p}>0$ such that
\begin{align}\label{zuca}
p\Bigl( f\bigl(t_{1},...,t_{n}\bigr)\Bigr) & \leq M_{p}\Bigl( t_{1}^{\eta_{1}} +t_{1}^{\zeta_{1}}\Bigr)\cdot ...\cdot \Bigl( t_{n}^{\eta_{n}} +t_{n}^{\zeta_{n}}\Bigr)\exp \bigl(\omega_{1}t_{1}+...+\omega_{n}t_{n}\bigr),
\end{align}
$\mbox{for a.e. }t_{1}\geq 0, ..., \ t_{n}\geq 0.$
Then we can apply the Fubini theorem in order to see that $F(\lambda_{1},...,\lambda_{n})$ absolutely converges for $\Re \lambda_{1}>\omega_{1},\ ...,\ \Re \lambda_{n}>\omega_{n};$ furthermore, the Lebesgue dominated convergence theorem implies that $F(\cdot)$ is analytic in this region.
\end{example}

\subsection{Regions of convergence of multidimensional vector-valued Laplace transform}\label{regions}

Unless specified otherwise, we will always assume henceforward that $n\geq 2.$ Set
$$
\text{abs}_{j}(f):=\inf \Bigl\{ \Re \lambda_{j} : \exists \bigl(\lambda_{1},...,\lambda_{j},...,\lambda_{n}\bigr)\in \Omega(f)\Bigr\},\quad 1\le j\leq n.
$$
Then it is clear that $(\lambda_{1},...,\lambda_{n})\notin \Omega(f)$ if there exists an index $j\in {\mathbb N}_{n}$ such that $\Re \lambda_{j}<\text{abs}_{j}(f);$ a similar statement holds for the set $\Omega_{abs}(f).$
Now we will state and prove the following result; concerning the one-dimensional setting, see \cite[Proposition 1.4.1]{a43} and \cite[Theorem 1.4.1(i)]{FKP}:

\begin{thm}\label{dl}
Suppose that $f: [0,+\infty)^{n}\rightarrow X$ is a locally integrable function, $(\lambda_{1}^{0},...,\lambda_{n}^{0})\in \Omega_{abs}(f),$  $(\lambda_{1},...,\lambda_{n})\in {\mathbb C}^{n}$ and $\Re \lambda_{j} >\Re \lambda_{j}^{0}$ for $1\leq j\leq n.$ Then $(\lambda_{1},...,\lambda_{n})\in \Omega_{abs}(f).$
\end{thm}

\begin{proof}
We will consider dimensions $n=2$ and $n=3,$ only; the proof in general case is quite similar.
Suppose first that $n=2 $ and $t>0.$ We will prove the formula{\small
\begin{align}\label{FDE}
\begin{split}
& \int^{t}_{0}\int^{t}_{0}e^{-\lambda_{1}t_{1}-\lambda_{2}t_{2}}f\bigl( t_{1},t_{2}\bigr)\, dt_{1}\, dt_{2}
\\ & =e^{-(\lambda_{1}-\lambda_{1}^{0})t-(\lambda_{2}-\lambda_{2}^{0})t}\int^{t}_{0}\int^{t}_{0}e^{-\lambda_{1}^{0}s_{1}-\lambda_{2}^{0}s_{2}}f\bigl( s_{1},s_{2}\bigr)\, ds_{1}\, ds_{2} 
\\& +\bigl( \lambda_{1}-\lambda_{1}^{0}\bigr)e^{-(\lambda_{2}-\lambda_{2}^{0})t}\int^{t}_{0}e^{-(\lambda_{1}-\lambda_{1}^{0})t_{1}}\Biggl[ \int^{t_{1}}_{0}\int^{t}_{0}e^{-\lambda_{1}^{0}s_{1}-\lambda_{2}^{0}s_{2}}f\bigl( s_{1},s_{2}\bigr)\, ds_{1} \, ds_{2}\Biggr]\, dt_{1} \\
& \\& +\bigl( \lambda_{2}-\lambda_{2}^{0}\bigr)e^{-(\lambda_{1}-\lambda_{1}^{0})t}\int^{t}_{0}e^{-(\lambda_{2}-\lambda_{2}^{0})t_{2}}\Biggl[ \int^{t}_{0}\int^{t_{2}}_{0}e^{-\lambda_{1}^{0}s_{1}-\lambda_{2}^{0}s_{2}}f\bigl( s_{1},s_{2}\bigr)\, ds_{1} \, ds_{2}\Biggr]\, dt_{2}
\\& +\bigl( \lambda_{1}-\lambda_{1}^{0}\bigr)\bigl( \lambda_{2}-\lambda_{2}^{0}\bigr)\int^{t}_{0}\int^{t}_{0}e^{-(\lambda_{1}-\lambda_{1}^{0})t_{1}-(\lambda_{2}-\lambda_{2}^{0})t_{2}}
\\& \times \Biggl[ \int^{t_{1}}_{0}\int^{t_{2}}_{0}e^{-\lambda_{1}^{0}s_{1}-\lambda_{2}^{0}s_{2}}f\bigl( s_{1},s_{2}\bigr)\, ds_{1}\, ds_{2}  \Biggr]\, dt_{1}\, dt_{2}.
\end{split}
\end{align}}
Keeping in mind Lemma \ref{bnm1}(i), it readily follows that the function $(t_{1},t_{2})\mapsto e^{-\lambda_{1}t_{1}-\lambda_{2}t_{2}}f( t_{1},t_{2})$ is Lebesgue integrable on the rectangle $[0,t] \times [0,t];$ similarly, the first double integral on the right hand side of \eqref{FDE} is well-defined as well as the double integrals appearing in the second addend and the third addend on the right hand side of \eqref{FDE}. Now it can be simply shown that the function in the second addend on the right hand side of \eqref{FDE} is continuous with respect to the variable $t_{1}$ as well as that the function in the third addend on the right hand side of \eqref{FDE} is continuous with respect to the variable $t_{2};$ furthermore, it is elementary to prove that the mapping appearing in the fourth addend of \eqref{FDE} is continuous with respect in plane with respect to the variable $(t_{1},t_{2}).$ Hence, any term in \eqref{FDE} is well-defined. 

Now, in order to deduce this equality, we may assume without loss of generality that the function $f(\cdot,\cdot)$ is scalar-valued; otherwise, we can consider the function $\langle x^{\ast}, f(\cdot,\cdot) \rangle ,$ where $x^{\ast}\in X^{\ast}$ is an arbitrary functional (see also the equation \eqref{sex-pistols}).  
Applying the Fubini theorem and partial integration, we get:
\begin{align*}
&\int^{t}_{0}\int^{t}_{0}e^{-\lambda_{1}t_{1}-\lambda_{2}t_{2}}f\bigl( t_{1},t_{2}\bigr)\, dt_{1}\, dt_{2}=\int^{t}_{0}e^{-\lambda_{1}t_{1}}\Biggl[ \int^{t}_{0}e^{-\lambda_{2}t_{2}}f\bigl( t_{1},t_{2}\bigr)\, dt_{2}\Biggr]\, dt_{1}
\\& =\int^{t}_{0}e^{-\lambda_{1}t_{1}}\Biggl[ \int^{t}_{0}e^{-(\lambda_{2}-\lambda_{2}^{0})t_{2}}e^{-\lambda_{2}^{0}t_{2}}f\bigl( t_{1},t_{2}\bigr)\, dt_{2}\Biggr]\, dt_{1}
\\& =\int^{t}_{0}e^{-(\lambda_{1}-\lambda_{1}^{0})t_{1}}e^{-\lambda_{1}^{0}t_{1}}\Biggl[ e^{-(\lambda_{2}-\lambda_{2}^{0})t}\int^{t}_{0}e^{-\lambda_{2}^{0}s_{2}}f\bigl( t_{1},s_{2}\bigr)\, ds_{2}\Biggr]\, dt_{1}
\\& +\bigl( \lambda_{2}-\lambda_{2}^{0}\bigr)\int^{t}_{0}e^{-(\lambda_{1}-\lambda_{1}^{0})t_{1}}e^{-\lambda_{1}^{0}t_{1}}\Biggl[ \int^{t}_{0}e^{-(\lambda_{2}-\lambda_{2}^{0})t_{2}}\int^{t_{2}}_{0}e^{-\lambda_{2}^{0}s_{2}}f\bigl( t_{1},s_{2}\bigr)\, ds_{2}\, dt_{2}\Biggr]\, dt_{1}
\\& =e^{-(\lambda_{1}-\lambda_{1}^{0})t_{1}-(\lambda_{2}-\lambda_{2}^{0})t_{2}}\int^{t}_{0}\int^{t}_{0}e^{-\lambda_{1}^{0}s_{1}-\lambda_{2}^{0}s_{2}}f\bigl( s_{1},s_{2}\bigr)\, ds_{1}\, ds_{2}
\\& +\bigl( \lambda_{1}-\lambda_{1}^{0}\bigr)e^{-(\lambda_{2}-\lambda_{2}^{0})t}\int^{t}_{0}e^{-(\lambda_{1}-\lambda_{1}^{0})t_{1}}\Biggl[\int^{t_{1}}_{0}e^{-\lambda_{1}^{0}s_{1}}\int^{t}_{0}e^{-\lambda_{2}^{0}s_{2}}f\bigl( s_{1},s_{2}\bigr)\, ds_{2}\, ds_{1}\Biggr] \, dt_{1}
\\& +\bigl( \lambda_{2}-\lambda_{2}^{0}\bigr)e^{-(\lambda_{1}-\lambda_{1}^{0})t}\int^{t}_{0}e^{-\lambda_{1}^{0}s_{1}}\int^{t}_{0}e^{-(\lambda_{2}-\lambda_{2}^{0})t_{2}}\int^{t_{2}}_{0}e^{-\lambda_{2}^{0}s_{2}}    f\bigl( s_{1},s_{2}\bigr)\, ds_{2}\, dt_{2} \, ds_{1}
\\& +\bigl( \lambda_{1}-\lambda_{1}^{0}\bigr)\bigl( \lambda_{2}-\lambda_{2}^{0}\bigr)\int^{t}_{0}e^{-(\lambda_{1}-\lambda_{1}^{0})t_{1}}\int^{t_{1}}_{0}e^{-\lambda_{1}^{0}s_{1}}
\\& \times \int^{t}_{0}e^{-(\lambda_{2}-\lambda_{2}^{0})t_{2}}\int^{t_{2}}_{0}e^{-\lambda_{2}^{0}s_{2}}f\bigl( s_{1},s_{2}\bigr)\, ds_{2}\, dt_{2} \, ds_{1}\, dt_{1},
\end{align*}
so that \eqref{FDE} follows by applying the Fubini theorem in the third addend and the fourth addend of the above computation. It is clear that the first addend in \eqref{FDE} tends to zero as $t\rightarrow +\infty.$ Moreover, the double integrals in the big brackets in the second addend and the third addend of \eqref{FDE} are bounded so that these terms also tend to zero as $t\rightarrow +\infty.$ If $p\in \circledast,$ then we can simply prove that{\small
$$
\int^{t}_{0}\int^{t}_{0}p\Bigl(e^{-\lambda_{1}t_{1}-\lambda_{2}t_{2}}f\bigl( t_{1},t_{2}\bigr)\Bigr)\, dt_{1}\, dt_{2}= \int^{t}_{0}\int^{t}_{0}e^{-(\Re \lambda_{1})t_{1}-(\Re \lambda_{2})t_{2}}p\Bigl(f\bigl( t_{1},t_{2}\bigr)\Bigr)\, dt_{1}\, dt_{2}<+\infty
$$}
by replacing the function $f(\cdot,\cdot)$ and the numbers $\lambda_{1},\ \lambda_{2},\ \lambda_{1}^{0},\ \lambda_{2}^{0}$ in \eqref{FDE} by the function $p(f(\cdot,\cdot))$ and the numbers $\Re\lambda_{1},\ \Re \lambda_{2},\ \Re \lambda_{1}^{0},\ \Re \lambda_{2}^{0}$, respectively; summa summarum, $F(\lambda_{1},\lambda_{2})$ absolutely converges and we have 
\begin{align*}
& F\bigl( \lambda_{1},\lambda_{2}\bigr)=\bigl( \lambda_{1}-\lambda_{1}^{0}\bigr)\bigl( \lambda_{2}-\lambda_{2}^{0}\bigr)
\int^{+\infty}_{0}\int^{+\infty}_{0}e^{-(\lambda_{1}-\lambda_{1}^{0})t_{1}-(\lambda_{2}-\lambda_{2}^{0})t_{2}}
\\& \times \Biggl[ \int^{t_{1}}_{0}\int^{t_{2}}_{0}e^{-\lambda_{1}^{0}s_{1}-\lambda_{2}^{0}s_{2}}f\bigl( s_{1},s_{2}\bigr) \, ds_{2}\, ds_{1} \Biggr]\, dt_{1}\, dt_{2}.
\end{align*}

Suppose now that $n=3 $ and $t>0.$ Then we have
\begin{align}
& \notag \int^{t}_{0}\int^{t}_{0}\int^{t}_{0}e^{-\lambda_{1}t_{1}-\lambda_{2}t_{2}-\lambda_{3}t_{3}}f\bigl( t_{1},t_{2},t_{3}\bigr)\, dt_{1}\, dt_{2}\, dt_{3}
\\\label{FDE0}& =\int^{t}_{0}e^{-(\lambda_{1}-\lambda_{1}^{0})t_{1}}e^{-\lambda_{1}t_{1}}\Biggl[ \int^{t}_{0}\int^{t}_{0}e^{-\lambda_{2}t_{2}-\lambda_{3}t_{3}}f\bigl( t_{1},t_{2},t_{3}\bigr)\, dt_{1}\, dt_{2} \Biggr]\, dt_{3}.
\end{align}
Applying the formula \eqref{FDE} for the double integral in \eqref{FDE0} and the partial integration, we can similarly prove that the following formula holds true: {\small
\begin{align}\label{FDE1}
\begin{split}
& \int^{t}_{0}\int^{t}_{0}\int^{t}_{0}e^{-\lambda_{1}t_{1}-\lambda_{2}t_{2}-\lambda_{3}t_{3}}f\bigl( t_{1},t_{2},t_{3}\bigr)\, dt_{1}\, dt_{2}\, dt_{3}
\\& =e^{-(\lambda_{1}-\lambda_{1}^{0})t-(\lambda_{2}-\lambda_{2}^{0})t-(\lambda_{3}-\lambda_{3}^{0})t}\int^{t}_{0}\int^{t}_{0}\int^{t}_{0}e^{-\lambda_{1}^{0}s_{1}-\lambda_{2}^{0}s_{2}-\lambda_{3}^{0}s_{3}}f\bigl( s_{1},s_{2},s_{3}\bigr)\, ds_{1}\, ds_{2} \, ds_{3}
\\& +\bigl( \lambda_{1}-\lambda_{1}^{0}\bigr)e^{-(\lambda_{2}-\lambda_{2}^{0})t-(\lambda_{3}-\lambda_{3}^{0})t}
\\& \times \int^{t}_{0}e^{-(\lambda_{1}-\lambda_{1}^{0})t_{1}}\Biggl[\int^{t_{1}}_{0}\int^{t}_{0}\int^{t}_{0}e^{-\lambda_{1}^{0}s_{1}-\lambda_{2}^{0}s_{2}-\lambda_{3}^{0}s_{3}}f\bigl( s_{1},s_{2},s_{3}\bigr)\, ds_{1}\, ds_{2} \, ds_{3} \Biggr]\, dt_{1}
\\& +\bigl( \lambda_{2}-\lambda_{2}^{0}\bigr)e^{-(\lambda_{1}-\lambda_{1}^{0})t-(\lambda_{3}-\lambda_{3}^{0})t}
\\& \times \int^{t}_{0}e^{-(\lambda_{2}-\lambda_{2}^{0})t_{2}}\Biggl[ \int^{t}_{0}\int^{t_{2}}_{0}\int^{t}_{0}e^{-\lambda_{1}^{0}s_{1}-\lambda_{2}^{0}s_{2}-\lambda_{3}^{0}s_{3}}f\bigl( s_{1},s_{2},s_{3}\bigr)\, ds_{2} \, ds_{3}\, ds_{1} \Biggr]\, dt_{2}
\\& +\bigl( \lambda_{11}-\lambda_{1}^{0}\bigr)\bigl( \lambda_{2}-\lambda_{2}^{0}\bigr)e^{-(\lambda_{3}-\lambda_{3}^{0})t}
\\& \times \int^{t}_{0} \int^{t}_{0}e^{-(\lambda_{1}-\lambda_{1}^{0})t_{1}-(\lambda_{2}-\lambda_{2}^{0})t_{2}}\Biggl[ \int^{t_{1}}_{0}\int^{t_{2}}_{0}\int^{t}_{0}e^{-\lambda_{1}^{0}s_{1}-\lambda_{2}^{0}s_{2}-\lambda_{3}^{0}s_{3}}f\bigl( s_{1},s_{2},s_{3}\bigr)\, ds_{2} \, ds_{3}\, ds_{1} \Biggr]\, dt_{1}\, dt_{2}
\\& +\bigl( \lambda_{3}-\lambda_{3}^{0}\bigr)e^{-(\lambda_{1}-\lambda_{1}^{0})t-(\lambda_{2}-\lambda_{2}^{0})t}
\\& \times \int^{t}_{0}e^{-(\lambda_{3}-\lambda_{3}^{0})t_{3}}\Biggl[ \int^{t}_{0}\int^{t}_{0}\int^{t_{3}}_{0}e^{-\lambda_{1}^{0}s_{1}-\lambda_{2}^{0}s_{2}-\lambda_{3}^{0}s_{3}}f\bigl( s_{1},s_{2},s_{3}\bigr)\, ds_{3} \, ds_{2}\, ds_{1} \Biggr]\, dt_{3}
\\& +\bigl( \lambda_{1}-\lambda_{1}^{0}\bigr)\bigl( \lambda_{3}-\lambda_{3}^{0}\bigr)e^{-(\lambda_{1}-\lambda_{1}^{0})t-(\lambda_{2}-\lambda_{2}^{0})t}
\\& \times \int^{t}_{0}\int^{t}_{0}e^{-(\lambda_{1}-\lambda_{1}^{0})t_{1}-(\lambda_{3}-\lambda_{3}^{0})t_{3}}\Biggl[ \int^{t_{1}}_{0}\int^{t}_{0}\int^{t_{3}}_{0}e^{-\lambda_{1}^{0}s_{1}-\lambda_{2}^{0}s_{2}-\lambda_{3}^{0}s_{3}}f\bigl( s_{1},s_{2},s_{3}\bigr)\, ds_{3} \, ds_{2}\, ds_{1} \Biggr]\, dt_{1}\, dt_{3}
\\& +\bigl( \lambda_{2}-\lambda_{2}^{0}\bigr)\bigl( \lambda_{3}-\lambda_{3}^{0}\bigr)e^{-(\lambda_{1}-\lambda_{1}^{0})t}
\\& \times \int^{t}_{0}\int^{t}_{0}e^{-(\lambda_{2}-\lambda_{2}^{0})t_{2}-(\lambda_{3}-\lambda_{3}^{0})t_{3}}\Biggl[ \int^{t}_{0}\int^{t_{2}}_{0}\int^{t_{3}}_{0}e^{-\lambda_{1}^{0}s_{1}-\lambda_{2}^{0}s_{2}-\lambda_{3}^{0}s_{3}}f\bigl( s_{1},s_{2},s_{3}\bigr)\, ds_{3} \, ds_{2}\, ds_{1} \Biggr]\, dt_{2}\, dt_{3}
\\& +\bigl( \lambda_{1}-\lambda_{1}^{0}\bigr)\bigl( \lambda_{2}-\lambda_{2}^{0}\bigr)\bigl( \lambda_{3}-\lambda_{3}^{0}\bigr)
\\& \times \int^{t}_{0}\int^{t}_{0}\int^{t}_{0}e^{-(\lambda_{1}-\lambda_{1}^{0})t_{1}-(\lambda_{2}-\lambda_{2}^{0})t_{2}-(\lambda_{3}-\lambda_{3}^{0})t_{3}}
\\& \times \Biggl[ \int^{t_{1}}_{0}\int^{t_{2}}_{0}\int^{t_{3}}_{0}e^{-\lambda_{1}^{0}s_{1}-\lambda_{2}^{0}s_{2}-\lambda_{3}^{0}s_{3}}f\bigl( s_{1},s_{2},s_{3}\bigr)\, ds_{3} \, ds_{2}\, ds_{1} \Biggr]\, dt_{1}\, dt_{2}\, dt_{3}.
\end{split}
\end{align}}
Arguing as in the two-dimensional setting, we can prove that any term in \eqref{FDE1} is well-defined as well as that the first seven addends tend to zero as $t\rightarrow +\infty,$ while the eight addend tends to
\begin{align*}
& F\bigl( \lambda_{1},\lambda_{2},\lambda_{3}\bigr)=\bigl( \lambda_{1}-\lambda_{1}^{0}\bigr)\bigl( \lambda_{2}-\lambda_{2}^{0}\bigr)\bigl( \lambda_{3}-\lambda_{3}^{0}\bigr)
\\& \times \int^{+\infty}_{0}\int^{+\infty}_{0}\int^{+\infty}_{0}e^{-(\lambda_{1}-\lambda_{1}^{0})t_{1}-(\lambda_{2}-\lambda_{2}^{0})t_{2}-(\lambda_{3}-\lambda_{3}^{0})t_{3}}
\\& \times \Biggl[ \int^{t_{1}}_{0}\int^{t_{2}}_{0}\int^{t_{3}}_{0}e^{-\lambda_{1}^{0}s_{1}-\lambda_{2}^{0}s_{2}-\lambda_{3}^{0}s_{3}}f\bigl( s_{1},s_{2},s_{3}\bigr)\, ds_{3} \, ds_{2}\, ds_{1} \Biggr]\, dt_{1}\, dt_{2}\, dt_{3},
\end{align*}
as $t\rightarrow +\infty .$
\end{proof}

An immediate corollary of Theorem \ref{dl} is the following statement:

\begin{cor}\label{dl1}
Suppose that $f: [0,+\infty)^{n}\rightarrow X$ is a locally integrable function. Then we have
\begin{align*}
\bigcup_{(\lambda_{1}^{0},...,\lambda_{n}^{0})\in \Omega_{abs}(f)}\Bigl[ \bigl\{ \lambda_{1}\in {\mathbb C} : \Re \lambda_{1}>\Re \lambda_{1}^{0}\bigr\} \times ... \times \bigl\{ \lambda_{n}\in {\mathbb C} : \Re \lambda_{n}>\Re \lambda_{n}^{0}\bigr\}\Bigr] \subseteq \Omega_{abs}(f).
\end{align*}
\end{cor}  

If $f: [0,+\infty)^{n}\rightarrow X$ is a locally integrable function, then we define $\Omega_{b}(f)$ as the set of all tuples $(\lambda_{1},...\lambda_{n})\in {\mathbb C}^{n}$ such that the set
\begin{align}\label{qaw}
\Biggl\{
\int^{t_{1}}_{0}\int^{t_{2}}_{0}...\int^{t_{n}}_{0}e^{-\lambda_{1}s_{1}-\lambda_{2}s_{2}-...-\lambda_{n}s_{n}}f\bigl( s_{1},s_{2},...,s_{n}\bigr)\, ds_{1} \, ds_{2}...\, ds_{n} :  t_{1}\geq 0,...,t_{n}\geq 0\Biggr\}
\end{align}
is bounded in $X.$ By
 \cite[Mackey's theorem 23.15]{meise}, the estimate \eqref{qaw} is equivalent with 
\begin{align*} \sup_{t_{1}\geq 0,...,t_{n}\geq 0}&
\int^{t_{1}}_{0}\int^{t_{2}}_{0}...\int^{t_{n}}_{0}e^{-(\Re \lambda_{1}^{0})s_{1}-(\Re \lambda_{2}^{0})s_{2}-...-(\Re \lambda_{n}^{0})s_{n}}
\\& \times \Bigl| \Bigl \langle x^{\ast},f\bigl( s_{1},s_{2},...,s_{n}\bigr)\Bigr \rangle \Bigr|\, ds_{1} \, ds_{2}...\, ds_{n}<+\infty ,\quad x^{\ast}\in X^{\ast};
\end{align*}
cf. also the first paragraph on \cite[p. 29]{a43} and \cite[Theorem 1.4.1(ii)]{FKP} for the one-dimensional setting. It is clear that $\Omega_{abs}(f)\subseteq \Omega_{b}(f);$ furthermore, we have the following analogue of Theorem \ref{dl}:

\begin{thm}\label{djura}
Suppose that $f: [0,+\infty)^{n}\rightarrow X$ is a locally integrable function, $(\lambda_{1}^{0},...,\lambda_{n}^{0})\in \Omega_{b}(f),$  $(\lambda_{1},...,\lambda_{n})\in {\mathbb C}^{n}$ and $\Re \lambda_{j} >\Re \lambda_{j}^{0}$ for $1\leq j\leq n.$ Then $(\lambda_{1},...,\lambda_{n})\in \Omega_{b}(f) \cap \Omega(f).$
\end{thm}

\begin{proof}
We will present all relevant details of proof in two-dimensional setting, only. Let $T\geq 0$, $t\geq 0$ and $(\lambda_{1},\lambda_{2})\in {\mathbb C}^{2}.$ Arguing as in the proof of Theorem \ref{dl}, we can prove that:
{\small
\begin{align}\label{FDET}
\begin{split}
& \int^{T}_{0}\int^{t}_{0}e^{-\lambda_{1}t_{1}-\lambda_{2}t_{2}}f\bigl( t_{1},t_{2}\bigr)\, dt_{1}\, dt_{2}
\\ & =e^{-(\lambda_{1}-\lambda_{1}^{0})T-(\lambda_{2}-\lambda_{2}^{0})t}\int^{T}_{0}\int^{t}_{0}e^{-\lambda_{1}^{0}s_{1}-\lambda_{2}^{0}s_{2}}f\bigl( s_{1},s_{2}\bigr)\, ds_{1}\, ds_{2} 
\\& +\bigl( \lambda_{1}-\lambda_{1}^{0}\bigr)e^{-(\lambda_{2}-\lambda_{2}^{0})t}\int^{T}_{0}e^{-(\lambda_{1}-\lambda_{1}^{0})t_{1}}\Biggl[ \int^{t_{1}}_{0}\int^{t}_{0}e^{-\lambda_{1}^{0}s_{1}-\lambda_{2}^{0}s_{2}}f\bigl( s_{1},s_{2}\bigr)\, ds_{1} \, ds_{2}\Biggr]\, dt_{1} \\
& \\& +\bigl( \lambda_{2}-\lambda_{2}^{0}\bigr)e^{-(\lambda_{1}-\lambda_{1}^{0})T}\int^{t}_{0}e^{-(\lambda_{2}-\lambda_{2}^{0})t_{2}}\Biggl[ \int^{T}_{0}\int^{t_{2}}_{0}e^{-\lambda_{1}^{0}s_{1}-\lambda_{2}^{0}s_{2}}f\bigl( s_{1},s_{2}\bigr)\, ds_{1} \, ds_{2}\Biggr]\, dt_{2}
\\& +\bigl( \lambda_{1}-\lambda_{1}^{0}\bigr)\bigl( \lambda_{2}-\lambda_{2}^{0}\bigr)\int^{T}_{0}\int^{t}_{0}e^{-(\lambda_{1}-\lambda_{1}^{0})t_{1}-(\lambda_{2}-\lambda_{2}^{0})t_{2}}
\\& \times \Biggl[ \int^{t_{1}}_{0}\int^{t_{2}}_{0}e^{-\lambda_{1}^{0}s_{1}-\lambda_{2}^{0}s_{2}}f\bigl( s_{1},s_{2}\bigr)\, ds_{1}\, ds_{2}  \Biggr]\, dt_{1}\, dt_{2}.
\end{split}
\end{align}}
This simply implies $(\lambda_{1},\lambda_{2})\in \Omega_{b}(f).$ Furthermore, it is clear that the first three addens in \eqref{FDET} converge to zero as $T\rightarrow +\infty$ and $t\rightarrow +\infty$, while the fourth addend in \eqref{FDET} converges to
\begin{align*}
\bigl( \lambda_{1}-\lambda_{1}^{0}\bigr)\bigl( \lambda_{2}-\lambda_{2}^{0}\bigr)&\int^{+\infty}_{0}\int^{+\infty}_{0}e^{-(\lambda_{1}-\lambda_{1}^{0})t_{1}-(\lambda_{2}-\lambda_{2}^{0})t_{2}}
\\& \times \Biggl[ \int^{t_{1}}_{0}\int^{t_{2}}_{0}e^{-\lambda_{1}^{0}s_{1}-\lambda_{2}^{0}s_{2}}f\bigl( s_{1},s_{2}\bigr)\, ds_{1}\, ds_{2}  \Biggr]\, dt_{1}\, dt_{2}.
\end{align*} 
In actual fact, the convergence can be proved with the help of Cauchy criterium since we have ($p\in \circledast;$ $C_{p}\leq T\leq T'<+\infty;$ $C_{p}\leq t\leq t'<+\infty$):
\begin{align*}&
p\Biggl( \int^{T}_{0}\int^{t}_{0}e^{-(\lambda_{1}-\lambda_{1}^{0})t_{1}-(\lambda_{2}-\lambda_{2}^{0})t_{2}}f\bigl( t_{1},t_{2}\bigr)\, dt_{1}\, dt_{2}\\&-\int^{T'}_{0}\int^{t'}_{0}e^{-(\lambda_{1}-\lambda_{1}^{0})t_{1}-(\lambda_{2}-\lambda_{2}^{0})t_{2}}f\bigl( t_{1},t_{2}\bigr)\, dt_{1}\, dt_{2}\Biggr)
\\ & \leq \int^{T}_{0}\int^{t'}_{t}e^{-(\Re \lambda_{1}-\Re \lambda_{1}^{0})t_{1}-(\Re \lambda_{2}-\Re \lambda_{2}^{0})t_{2}}p\Bigl(f\bigl( t_{1},t_{2}\bigr)\Bigr)\, dt_{1}\, dt_{2}\\&+\int^{T'}_{T}\int^{t}_{0}e^{-(\Re \lambda_{1}-\Re \lambda_{1}^{0})t_{1}-(\Re \lambda_{2}-\Re \lambda_{2}^{0})t_{2}}p\Bigl(f\bigl( t_{1},t_{2}\bigr)\Bigr)\, dt_{1}\, dt_{2}
\\ & \leq \frac{c_{p}}{\Re \lambda_{1}-\Re \lambda_{1}^{0}}\int^{+\infty}_{t}e^{-(\Re \lambda_{2}-\Re \lambda_{2}^{0})t_{2}}\, dt_{2}+\frac{c_{p}}{\Re \lambda_{2}-\Re \lambda_{2}^{0}}\int^{+\infty}_{T}e^{-(\Re \lambda_{1}-\Re \lambda_{1}^{0})t_{1}}\, dt_{1}\leq \epsilon,
\end{align*}
for certain positive real constants $c_{p}>0$ and $C_{p}>0.$
This implies $(\lambda_{1},\lambda_{2})\in \Omega(f)$ and completes the proof. 
\end{proof}

Now we can state the following corollary:

\begin{cor}\label{dl121}
Suppose that $f: [0,+\infty)^{n}\rightarrow X$ is a locally integrable function. Then we have{\small
\begin{align*}
\bigcup_{(\lambda_{1}^{0},...,\lambda_{n}^{0})\in \Omega_{b}(f)}\Bigl[ \bigl\{ \lambda_{1}\in {\mathbb C} : \Re \lambda_{1}>\Re \lambda_{1}^{0}\bigr\} \times ... \times \bigl\{ \lambda_{n}\in {\mathbb C} : \Re \lambda_{n}>\Re \lambda_{n}^{0}\bigr\}\Bigr] \subseteq \Omega(f) \cap \Omega_{b}(f).
\end{align*}}
\end{cor}  

\begin{rem}\label{coolcats}
\begin{itemize}
\item[(i)]
Let $(\lambda_{1}^{0},...,\lambda_{n}^{0})\in \Omega(f)$ and let all remaining assumptions from the formulation of Theorem \ref{djura} hold. Then it is not clear how to prove that $(\lambda_{1},...,\lambda_{n})\in \Omega(f);$ this result has been incorrectly stated in \cite[Theorem 2.1]{simplify}, where the authors have employed the Fubini theorem to achieve their aims (of course, this is always possible if $(\lambda_{1}^{0},...,\lambda_{n}^{0})\in \Omega_{abs}(f);$ cf. Theorem \ref{dl}). 
\item[(ii)]
Following  G. A. Coon and D. L. Bernstein \cite{bern1}, we say that the Laplace integral $F(\lambda_{1}^{0},...,\lambda_{n}^{0})$  converges boundedly if $(\lambda_{1}^{0},...,\lambda_{n}^{0})\in \Omega(f) \cap \Omega_{b}(f)$. It is worth noting that Theorem \ref{djura} provides a proper extension of \cite[Theorem 2.1(a)]{bern1} to the vector-valued functions as well as that it is not true that the region of bounded convergence $\Omega(f) \cap \Omega_{b}(f)$ of function $f(\cdot)$ has the form 
$\{ \lambda_{1}\in {\mathbb C} : \Re \lambda_{1}>a_{1} \} \times ... \times  \{ \lambda_{n}\in {\mathbb C} : \Re \lambda_{n}>a_{n}\}$ for some real numbers $a_{1},..., a_{n};$ see \cite[p. 145,  l. -13- l. -8]{bern1} for more details on the subject.
\end{itemize}
\end{rem}

In the one-dimensional setting, we have $\Omega(f)\subseteq \Omega_{b}(f)$; the converse statement does not hold since the function $f: [0,\infty) \rightarrow {\mathbb R}$, given by $f(t):=1,$ if $t\in [2k,2k+1)$ for some $k\in {\mathbb N}_{0}$ and $f(t):=-1,$ if $t\in [2k+1,2k+2)$ for some $k\in {\mathbb N}_{0},$ has the property that the set $\{ \int^{t}_{0}f(s)\, ds : t\geq 0\}$ is bounded but $\lim_{t\rightarrow \infty}\int^{t}_{0}f(s)\, ds$ does not exist. In the multi-dimensional setting, it is not generally true that $\Omega(f)\subseteq \Omega_{b}(f);$  for example, if we consider the function $f(\cdot,\cdot)$ given on \cite[p. 144, l. 16--l. 19]{bern1}, then we have $(0,0)\in \Omega(f)$ and $(0,0)\notin \Omega_{b}(f).$ 

Further on, in the one-dimensional setting, it is well-known that a locally integrable function $f: [0,\infty) \rightarrow X$ is Laplace transformable if and only if its first antiderivative $f^{[1]}(\cdot):=\int^{\cdot}_{0}f(t)\, dt$ is exponentially bounded; see \cite[Theorem 1.4.3]{a43} and \cite[Theorem 1.4.1(v)]{FKP}. Concerning the above-mentioned statements, we would like to state and prove the following results which provide the vector-valued extensions of \cite[Theorem 2.2, Theorem 2.3, Theorem 2.4]{bern1}; cf. also \cite[Theorem 2.7-Theorem 2.10 and Corollary 1-Corollary 2]{bern1}:

\begin{thm}\label{sas}
Suppose that $f: [0,+\infty)^{n}\rightarrow X$ is a locally integrable function. Then the following holds:
\begin{itemize}
\item[(i)] Let  $(\lambda_{1}^{0},...,\lambda_{n}^{0})\in \Omega_{abs}(f).$ Then for each seminorm $p\in \circledast$ there exists a finite real number $M_{p}>0$ such that
\begin{align}\label{antid}
\int^{t_{1}}_{0}\int^{t_{2}}_{0}...\int^{t_{n}}_{0}p\Bigl( f\bigl( s_{1},s_{2},...,s_{n}\bigr)\Bigr)\, ds_{1} \, ds_{2}...\, ds_{n}\leq M_{p}\Biggl[ e^{\max \bigl(0,\lambda_{1}^{0}\bigr)t_{1}+...+\max \bigl(0,\lambda_{n}^{0}\bigr)t_{n}}\Biggr],
\end{align}
for any $(t_{1},...,t_{n}) \in [0,+\infty)^{n}.$
\item[(ii)] Let  $(\lambda_{1}^{0},...,\lambda_{n}^{0})\in \Omega_{b}(f)$ and
\begin{align}\label{anatolia}
G\bigl( t_{1},...,t_{n} \bigr):=\int^{t_{1}}_{0}...\int^{t_{n}}_{0}f\bigl( s_{1},...,s_{n}\bigr)\, ds_{1}...\, ds_{n},\quad \bigl( t_{1},...,t_{n} \bigr)\in [0,+\infty)^{n}.
\end{align}
Then for each seminorm $p\in \circledast$ there exists $M_{p}>0$ such that, for every $t_{1}\geq 0,...,$ $t_{n}\geq 0,$ we have: \begin{align}\label{pr}  p\Bigl(G\bigl( t_{1},...,t_{n} \bigr)\Bigr)
 \leq M_{p}\sum \Bigl[a_{1}\bigl(t_{1}\bigr)\cdot ...\cdot a_{n}\bigl(t_{n}\bigr)\Bigr] ,
\end{align}
where $a_{j}(t_{j})$ can be $ e^{(\Re \lambda_{j}^{0})\cdot t_{j}}$ or $\int^{t_{j}}_{0}e^{(\Re \lambda_{j}^{0})\cdot s}\, ds$ for $1\leq j\leq n$ (the sum has exactly $2^{n}$ addens).
\item[(iii)] Suppose that the function $G(\cdot)$ is given by \eqref{anatolia}, $(\lambda_{1}^{0},...,\lambda_{n}^{0})\in {\mathbb C}^{n}$, \eqref{pr} holds, $(\lambda_{1},...,\lambda_{n})\in {\mathbb C}^{n}$ and $\Re \lambda_{j}>\max(\Re \lambda_{j}^{0},0)$ for all $j\in {\mathbb N}_{n}.$ Then $(\lambda_{1},...,\lambda_{n})\in \Omega_{abs}(G)$, $(\lambda_{1},...,\lambda_{n})\in \Omega (f) \cap \Omega_{b}(f)$ and
\begin{align}\label{ej}
\bigl( {\mathcal L} f\bigr) \bigl(\lambda_{1},...,\lambda_{n} \bigr)=\lambda_{1}\cdot ... \cdot \lambda_{n}\cdot \bigl( {\mathcal L} G\bigr) \bigl(\lambda_{1},...,\lambda_{n} \bigr).
\end{align}
\end{itemize}
\end{thm}

\begin{proof}
Suppose first that $(\lambda_{1}^{0},...,\lambda_{n}^{0})\in \Omega_{abs}(f)$ and $p\in \circledast.$ Then 
\begin{align*} &
\int^{t_{1}}_{0}\int^{t_{2}}_{0}...\int^{t_{n}}_{0}p\Bigl( f\bigl( s_{1},s_{2},...,s_{n}\bigr)\Bigr)\, ds_{1} \, ds_{2}...\, ds_{n}
\\& \leq \int^{t_{1}}_{0}\int^{t_{2}}_{0}...\int^{t_{n}}_{0}p\Bigl( e^{-\lambda_{1}^{0}s_{1}-...-\lambda_{n}^{0}s_{n}}f\bigl( s_{1},s_{2},...,s_{n}\bigr)\Bigr)e^{\Re (\lambda_{1}^{0})s_{1}+...+\Re(\lambda_{n}^{0})s_{n}}\, ds_{1} \, ds_{2}...\, ds_{n} \\&
\leq \Biggl[ e^{\max \bigl(0,\lambda_{1}^{0}\bigr)t_{1}+...+\max \bigl(0,\lambda_{n}^{0}\bigr)t_{n}}\Biggr]
\\& \times \int^{t_{1}}_{0}\int^{t_{2}}_{0}...\int^{t_{n}}_{0}p\Bigl( e^{-\lambda_{1}^{0}s_{1}-...-\lambda_{n}^{0}s_{n}}f\bigl( s_{1},s_{2},...,s_{n}\bigr)\Bigr)\, ds_{1} \, ds_{2}...\, ds_{n},
\end{align*}
which yields \eqref{antid} and completes the proof of (i). We will prove (ii) only in two-dimensional setting. Let $(\lambda_{1}^{0},\lambda_{2}^{0})\in \Omega_{b}(f);$ plugging $\lambda_{1}=\lambda_{2}=0$ in \eqref{FDET}, we obtain that for each seminorm $p\in \circledast$ there exists $M_{p}>0$ such that, for every $T\geq 0$ and $t\geq 0,$ the following holds:
\begin{align*}
  p\Biggl(&\int^{T}_{0}\int^{t}_{0}f\bigl( t_{1},t_{2}\bigr)\, dt_{1}\, dt_{2}\Biggr)
 \leq M_{p}\Biggl[ e^{(\Re \lambda_{1}^{0})\cdot T+(\Re \lambda_{2}^{0})\cdot t}+e^{(\Re \lambda_{2}^{0})\cdot t} \int^{T}_{0}e^{(\Re \lambda_{1}^{0})s}\, ds\\&+e^{(\Re \lambda_{1}^{0})\cdot T} \int^{t}_{0}e^{(\Re \lambda_{2}^{0})s}\, ds+\int^{T}_{0}e^{(\Re \lambda_{1}^{0})s}\, ds \cdot  \int^{t}_{0}e^{(\Re \lambda_{2}^{0})s}\, ds\Biggr].
\end{align*}
This simply yields the required conclusion. 

We will deduce the assertion (ii) only in two-dimensional setting. Using the estimate \eqref{blue-note-0}, it follows that the function $G: [0,\infty)^{2}\rightarrow X$ is continuous. Now we will prove the following equality:
{\small
\begin{align}\label{FDETint}
\begin{split}
& \int^{t_{2}}_{0}\int^{t_{1}}_{0}e^{-\lambda_{1}s_{1}-\lambda_{2}s_{2}}f\bigl( s_{1},s_{2}\bigr)\, ds_{1}\, ds_{2}
=e^{-\lambda_{1}t_{1}-\lambda_{2}t_{2}}\int^{t_{1}}_{0}\int^{t_{2}}_{0}f\bigl( v_{1},r\bigr)\, dr\, dv_{1} 
\\& + \lambda_{1}\int^{t_{1}}_{0}e^{-\lambda_{1}s_{1}-\lambda_{2}t_{2}}\Biggl[ \int^{s_{1}}_{0}\int^{t_{2}}_{0}f\bigl( v_{1},r\bigr)\, dr\, dv_{1} \Biggr] \, ds_{1}\\
& \\& +\lambda_{2}\int^{t_{2}}_{0}e^{-\lambda_{1}t_{1}-\lambda_{2}s_{2}}\Biggl[ \int^{t_{1}}_{0}\int^{s_{2}}_{0}f\bigl( v_{1},r\bigr)\, dr \, dv_{1}\Biggr]\, ds_{2}
\\& +\ \lambda_{1} \lambda_{2}\int^{t_{1}}_{0}\int^{t_{2}}_{0}e^{-\lambda_{1}s_{1}-\lambda_{2}s_{2}}
\Biggl[ \int^{s_{1}}_{0}\int^{s_{2}}_{0}f\bigl( v_{1},r\bigr)\, dr\, dv_{1}\Biggr]\, ds_{1}\, ds_{2}.
\end{split}
\end{align}}
In order to that, we may assume without loss of generality that the function $f(\cdot,\cdot)$ is scalar-valued (otherwise, we can use the corresponding functionals). Applying the partial integration three times, we get:
\begin{align*}
& \int^{t_{2}}_{0}\int^{t_{1}}_{0}e^{-\lambda_{1}s_{1}-\lambda_{2}s_{2}}f\bigl( s_{1},s_{2}\bigr)\, ds_{1}\, ds_{2}
\\& = \int^{t_{2}}_{0}e^{-\lambda_{1}s_{1}}\Biggl[ e^{-\lambda_{2}t_{2}}\int^{t_{2}}_{0}f\bigl( s_{1},r\bigr)\, dr  +\lambda_{2}\int^{t_{2}}_{0}e^{-\lambda_{2}s_{2}}\int^{s_{2}}_{0}f\bigl( v_{1},r\bigr)\, dr \, ds_{2}\Biggr] \, ds_{1}
\\& =e^{-\lambda_{1}t_{1}-\lambda_{2}t_{2}}\int^{t_{1}}_{0}\int^{t_{2}}_{0}f\bigl( v_{1},r\bigr)\, dr\, dv_{1}
\\& +\lambda_{1}\int^{t_{1}}_{0}e^{-\lambda_{1}s_{1}-\lambda_{2}s_{2}}\Biggl[ \int^{s_{1}}_{0}\int^{t_{2}}_{0}f\bigl( v_{1},r\bigr)\, dr\, dv_{1} \Biggr]\, ds_{1} 
\\& +\lambda_{2}e^{-\lambda_{1}t_{1}}\int^{t_{2}}_{0} e^{-\lambda_{2}s_{2}} \Biggl[ \int^{t_{1}}_{0}\int^{s_{2}}_{0}f\bigl( v_{1},r\bigr)\, dr\, dv_{1}  \Biggr]  \, ds_{2}
\\& +\lambda_{1}\lambda_{2}\int^{t_{1}}_{0}\int^{s_{1}}_{0}\int^{t_{2}}_{0}e^{-\lambda_{1}s_{1}-\lambda_{2}s_{2}}
\Biggl[ \int^{s_{2}}_{0}f\bigl( v_{1},r\bigr)\, dr\, ds_{2}\, dv_{1}\Biggr]\, ds_{1},
\end{align*} 
so that \eqref{FDETint} follows by applying the Fubini theorem in the third addend and the fourth addend in the above computation. 

Keeping in mind the estimate \eqref{pr}, we can simply prove that $(\lambda_{1},\lambda_{2})\in \Omega_{abs}(G)$ since for each seminorm $p\in \circledast$ there exists a finite real constant $M_{p}>0$ such that
\begin{align*} &
 \int^{+\infty}_{0}\int^{+\infty}_{0}e^{-\Re \lambda_{1}s_{1}-\Re \lambda_{2}s_{2}}p\Bigl(G\bigl( s_{1},s_{2}\bigr)\Bigr)\, ds_{1}\, ds_{2}
\\& \leq M_{p} \int^{+\infty}_{0}\int^{+\infty}_{0}e^{-\Re \lambda_{1}s_{1}-\Re \lambda_{2}s_{2}}\Biggl[ e^{(\Re \lambda_{1}^{0})\cdot s_{1}+(\Re \lambda_{2}^{0})\cdot s_{2}}+e^{(\Re \lambda_{2}^{0})\cdot s_{2}} \int^{s_{1}}_{0}e^{(\Re \lambda_{1}^{0})s}\, ds\\&+e^{(\Re \lambda_{1}^{0})\cdot s_{1}} \int^{s_{2}}_{0}e^{(\Re \lambda_{2}^{0})s}\, ds+\int^{s_{1}}_{0}e^{(\Re \lambda_{1}^{0})s}\, ds \cdot  \int^{s_{2}}_{0}e^{(\Re \lambda_{2}^{0})s}\, ds\Biggr]\, ds_{1}\, ds_{2}<+\infty.
\end{align*}
Further on, it is clear that the first addend in \eqref{FDETint} tends to zero as $t_{1}\rightarrow +\infty$ and $t_{2}\rightarrow +\infty;$ a simple computation involving the estimate \eqref{pr} also shows that the second addend in \eqref{FDETint} and the third addend addend in \eqref{FDETint} tend to zero as $t_{1}\rightarrow +\infty$ and $t_{2}\rightarrow +\infty.$ The inclusion $(\lambda_{1},\lambda_{2})\in \Omega(f)$ and equality \eqref{ej} with $n=2$ now follow from the fact that the fourth addend in \eqref{FDETint} tends to $\lambda_{1}\lambda_{2}\cdot ( {\mathcal L} G) (\lambda_{1},\lambda_{2})$ as $t_{1}\rightarrow +\infty$ and $t_{2}\rightarrow +\infty;$ we can similarly prove that $(\lambda_{1},\lambda_{2})\in \Omega_{b}(f)$.
\end{proof}  

We close this subsection with the following observation:

\begin{rem}\label{out}
We would like to notice that the requirements of Theorem \ref{sas}(iii) do not necessarily imply that $(\lambda_{1},...,\lambda_{n})\in \Omega_{abs}(f) ;$ see \cite[Example 1.4.4]{a43}. 
\end{rem}

\subsection{Operational and analytical properties of multidimensional vector-valued Laplace transform}\label{regions1}

The basic operational properties of multidimensional vector-valued Laplace transform are collected in the following theorem (cf. also \cite[Section 4]{bern1} for some other features of double Laplace transform which can be simply reformulated in the vector-valued setting):

\begin{thm}\label{opera}
Suppose that $f: [0,+\infty)^{n}\rightarrow X$ is a locally Lebesgue integrable function. Then the following holds:
\begin{itemize}
\item[(i)] Suppose $h_{j}\geq 0$ for $1\leq j\leq n$, 
$$
f_{h}\bigl(t_{1},...,t_{n}\bigr):=f\bigl(t_{1}+h_{1},...,t_{n}+h_{n}\bigr),\quad t_{1}\geq 0, ..., \ t_{n}\geq 0
$$
and $H:=[h_{1},+\infty) \times ... \times [h_{n},+\infty).$ If
$(\lambda_{1},...,\lambda_{n})\in \Omega_{abs}(f),$ then\\ $(\lambda_{1},...,\lambda_{n})\in \Omega_{abs}(f_{h})$  and 
\begin{align}\notag &\bigl({\mathcal Lf_{h}}\bigr)\bigl(\lambda_{1},...,\lambda_{n}\bigr)=e^{\lambda_{1}h_{1}+...+\lambda_{n}h_{n}}\bigl({\mathcal Lf}\bigr)\bigl(\lambda_{1},...,\lambda_{n}\bigr)
\\& \label{katas}-e^{\lambda_{1}h_{1}+...+\lambda_{n}h_{n}}\int_{[0,+\infty)^{n}\setminus H}e^{-\lambda_{1}t_{1}-...-\lambda_{n}t_{n}}f\bigl(t_{1},...,t_{n}\bigr)\, dt_{1}\, ...\, dt_{n}.
\end{align}
\item[(ii)] Suppose $h_{j}\geq 0$ for $1\leq j\leq n$, 
$$
f_{h-}\bigl(t_{1},...,t_{n}\bigr):=f\bigl(t_{1}-h_{1},...,t_{n}-h_{n}\bigr),\quad t_{1}\geq h_{1},  ..., \ t_{n}\geq h_{n},
$$ 
and 
$$
f_{h-}\bigl(t_{1},...,t_{n}\bigr):=0\ \ \mbox{ if there exists }j\in {\mathbb N}_{n} \ \mbox{such that }t_{j}<h_{j}.
$$ 
Then $(\lambda_{1},...,\lambda_{n})\in \Omega_{abs}(f),$ resp. $(\lambda_{1},...,\lambda_{n})\in \Omega(f),$ if and only if\\ $(\lambda_{1},...,\lambda_{n})\in \Omega_{abs}(f_{h-})$, resp. $(\lambda_{1},...,\lambda_{n})\in \Omega(f_{h-})$; if this is the case, then we have  
$$
\bigl({\mathcal Lf_{h-}}\bigr)\bigl(\lambda_{1},...,\lambda_{n}\bigr)=e^{-\lambda_{1}h_{1}-...-\lambda_{n}h_{n}}\bigl({\mathcal Lf}\bigr)\bigl(\lambda_{1},...,\lambda_{n}\bigr).
$$
\item[(iii)] If $T\in L(X,Y)$ and $(\lambda_{1},...,\lambda_{n})\in \Omega_{abs}(f),$ resp. $(\lambda_{1},...,\lambda_{n})\in \Omega(f),$ then $(\lambda_{1},...,\lambda_{n})\in \Omega_{abs}(Tf)$, resp. $(\lambda_{1},...,\lambda_{n})\in \Omega(Tf)$, and $$
T\bigl({\mathcal Lf}\bigr) \bigl(\lambda_{1},...,\lambda_{n}\bigr)=\bigl({\mathcal L(Tf)}\bigr)\bigl(\lambda_{1},...,\lambda_{n}\bigr).
$$
\item[(iv)] If $z_{j}\in {\mathbb C}$ for $1\leq j\leq n$, 
$$
f_{z}\bigl(t_{1},...,t_{n}\bigr):=e^{-z_{1}t_{1}-...-z_{n}t_{n}}f\bigl(t_{1},...,t_{n}\bigr),\quad t_{1}\geq 0, ..., \ t_{n}\geq 0,
$$ 
then $(\lambda_{1}+z_{1},...,\lambda_{n}+z_{n})\in \Omega_{abs}(f)$, resp. $(\lambda_{1}+z_{1},...,\lambda_{n}+z_{n})\in \Omega(f)$, if and only if $(\lambda_{1},...,\lambda_{n})\in \Omega_{abs}(f_{z}),$ resp. $(\lambda_{1},...,\lambda_{n})\in \Omega(f_{z});$ if this is the case, then we have 
$$
\bigl({\mathcal Lf_{z}}\bigr) \bigl(\lambda_{1},...,\lambda_{n}\bigr)=\bigl({\mathcal Lf}\bigr)\bigl(\lambda_{1}+z_{1},...,\lambda_{n}+z_{n}\bigr).
$$
\item[(v)] Suppose that ${\mathcal A}$ is a closed \emph{MLO} between $X$ and $Y$, $g: [0,+\infty)^{n}\rightarrow Y$ is locally integrable and $g(t_{1},...,t_{n})\in {\mathcal A}f(t_{1},...,t_{n})$ for a.e. $(t_{1},...,t_{n})\in [0,+\infty)^{n}.$ If $(\lambda_{1},...,\lambda_{n})\in \Omega(f) \cap \Omega(g),$ then 
$$
\bigl({\mathcal Lg}\bigr)\bigl(\lambda_{1},...,\lambda_{n}\bigr)\in {\mathcal A}\bigl({\mathcal Lf}\bigr) \bigl(\lambda_{1},...,\lambda_{n}\bigr).
$$
\end{itemize}
\end{thm}

\begin{proof}
Using Lemma \ref{bnm1}(i) and a simple argumentation, we have that the functions $f_{h} : [0,+\infty)^{n}\rightarrow X$ and $e^{-\lambda_{1}\cdot_{1}-...-\lambda_{n}\cdot_{n}}f_{h}(\cdot_{1},...,\cdot_{n})$ are locally integrable. Furthermore, we have:
\begin{align*} &
\int^{+\infty}_{0}...\int^{+\infty}_{0}e^{-\lambda_{1}t_{1}-...-\lambda_{n}t_{n}}f_{h}\bigl(t_{1},...,t_{n}\bigr)\, dt_{1}\, ...\, dt_{n}
\\&=\int^{+\infty}_{h_{1}}...\int^{+\infty}_{h_{n}}e^{-\lambda_{1}t_{1}-...-\lambda_{n}t_{n}}f\bigl(t_{1},...,t_{n}\bigr)\, dt_{1}\, ...\, dt_{n}
\\& =e^{\lambda_{1}h_{1}+...+\lambda_{n}h_{n}}\bigl({\mathcal Lf}\bigr)\bigl(\lambda_{1},...,\lambda_{n}\bigr)\\&-e^{\lambda_{1}h_{1}+...+\lambda_{n}h_{n}}\int_{[0,+\infty)^{n}\setminus H}e^{-\lambda_{1}t_{1}-...-\lambda_{n}t_{n}}f\bigl(t_{1},...,t_{n}\bigr)\, dt_{1}\, ...\, dt_{n},
\end{align*}
which proves \eqref{katas}. The proof of (ii) is simple and therefore omitted. Suppose now that $T\in L(X,Y)$ and $(\lambda_{1},...,\lambda_{n})\in \Omega_{abs}(f)$. Then for each seminorm $p\in \circledast_{Y}$ there exist $c>0$ and $q\in \circledast$ such that $p(Tx)\leq cq(x)$ for all $x\in X;$ if this is the case, then we have
\begin{align*} &
\int^{+\infty}_{0}...\int^{+\infty}_{0}p\Bigl(e^{-\lambda_{1}t_{1}-...-\lambda_{n}t_{n}}Tf\bigl(t_{1},...,t_{n}\bigr)\Bigr)\, dt_{1}\, ...\, dt_{n}
\\& =\int^{+\infty}_{0}...\int^{+\infty}_{0}e^{-(\Re \lambda_{1})t_{1}-...-(\Re\lambda_{n})t_{n}}p\Bigl( Tf\bigl(t_{1},...,t_{n}\bigr)\Bigr)\, dt_{1}\, ...\, dt_{n}
\\& \leq c\int^{+\infty}_{0}...\int^{+\infty}_{0}e^{-(\Re \lambda_{1})t_{1}-...-(\Re\lambda_{n})t_{n}}q\Bigl( f\bigl(t_{1},...,t_{n}\bigr)\Bigr)\, dt_{1}\, ...\, dt_{n}
\\& =c \int^{+\infty}_{0}...\int^{+\infty}_{0}q\Bigl(e^{-\lambda_{1}t_{1}-...-\lambda_{n}t_{n}}f\bigl(t_{1},...,t_{n}\bigr)\Bigr)\, dt_{1}\, ...\, dt_{n}<+\infty,
\end{align*}
so that the equality $T(({\mathcal Lf})(\lambda_{1},...,\lambda_{n}))=({\mathcal L(Tf)})(\lambda_{1},...,\lambda_{n})$ follows from an application of Lemma \ref{bnm}; the proof of (iii) in case $(\lambda_{1},...,\lambda_{n})\in \Omega(f)$ is similar and therefore omitted. We will omit the proofs of (iv) and (v), as well, since they are very easy (cf. also \cite[Theorem 1.2.3, Theorem 1.4.2(iv)]{FKP}).
\end{proof}
 
The complex inversion theorem for the multidimensional Laplace transform of functions with values in complex Banach spaces has recently been clarified in \cite[Theorem 1]{axioms}. We can extend this result in the following way (cf. also \cite[Theorem 1.4.8]{FKP}):

\begin{thm}\label{cit}
Suppose that $\omega_{1}\geq 0,\ ...,\ \omega_{n}\geq 0,$  $\epsilon_{1}>0 ,\ ...,\ \epsilon_{n}>0$, $F: \{\lambda \in {\mathbb C} : \Re \lambda>\omega_{1}\} \times ... \times \{\lambda \in {\mathbb C} : \Re \lambda>\omega_{n}\}\rightarrow X$ is an analytic function and for each seminorm $p\in \circledast$ there exists a finite real constant $M_{p}>0$ such that
\begin{align*} 
p\Bigl( F\bigl( \lambda_{1},...,\lambda_{n}\bigr) \Bigr) \leq M_{p}|\lambda_{1}|^{-1-\epsilon_{1}}\cdot ...\cdot |\lambda_{n}|^{-1-\epsilon_{n}},\quad \Re \lambda_{j}>\omega_{j}\ \ (1\leq j\leq n).
\end{align*}
Then there exists a continuous function
$f:[0,+\infty)^{n}\rightarrow X$ such that for each seminorm $p\in \circledast$ there exists a finite real constant $M_{p}'>0$ such that
\begin{align*}
p\bigl( f(t_{1},...,t_{n})\bigr)\leq M_{p}'\Bigl[t_{1}^{\epsilon_{1}}e^{\omega_{1}t_{1}}\cdot ...\cdot t_{n}^{\epsilon_{n}}e^{\omega_{n}t_{n}}\Bigr] \mbox{ for all }t_{1}\geq 0,\ ...,\ t_{n}\geq 0
\end{align*}
and $F(\lambda_{1},...,\lambda_{n})=({\mathcal L}f)(\lambda_{1},...,\lambda_{n})$ converges absolutely for $\Re \lambda_{j}>\omega_{j}$ ($1\leq j\leq n$).
\end{thm}

We continue with the following illustrative example:

\begin{example}\label{pol}
\begin{itemize}
\item[(i)] If $\alpha>0$ and $\beta \in {\mathbb R},$ then the
Mittag-Leffler function $E_{\alpha,\beta}(z)$ is defined by
$$
E_{\alpha,\beta}(z):=\sum \limits_{k=0}^{\infty}\frac{z^{k}}{\Gamma(\alpha
k+\beta)},\quad z\in {\mathbb C};
$$
here, we assume that
$1/\Gamma(\alpha k+\beta)=0$ if $\alpha k+\beta \in -{{\mathbb
N}_{0}}.$ 
Since
\begin{align*}
\int \limits^{+\infty}_{0}e^{-\lambda t}t^{\beta-1}E_{\alpha,\beta}
\bigl( \omega t^{\alpha}\bigr) \, dt=\frac{\lambda^{\alpha-\beta}}{\lambda^{\alpha}-\omega},
\quad \Re \lambda >\omega^{1/\alpha},\ \omega>0,
\end{align*}
the Fubini theorem and the asymptotic expansion for the Mittag-Leffler functions \cite[Theorem 1.5.1]{FKP} imply that, for every $\alpha_{1}>0,$ ..., $\alpha_{n}>0,$ $\omega_{1}>0,$ ..., $\omega_{n}>0$ and $\beta_{1} \in {\mathbb R},$ ..., $\beta_{n}\in {\mathbb R},$ we have
\begin{align}\notag
\int \limits^{+\infty}_{0} \int \limits^{+\infty}_{0}...\int \limits^{+\infty}_{0} & e^{-\lambda_{1} t_{1}-...-\lambda_{n}t_{n}}\prod_{j=1}^{n}t_{j}^{\beta_{j}-1}E_{\alpha_{j},\beta_{j}}
\bigl( \omega_{j} t_{j}^{\alpha}\bigr) \, dt_{1}...\, dt_{n}
\\\label{l}&=\prod_{j=1}^{n}\frac{\lambda_{j}^{\alpha_{j}-\beta_{j}}}{\lambda_{j}^{\alpha_{j}}-\omega_{j}},
\quad \Re \lambda_{j} >\omega_{j}^{1/\alpha_{j}}\  \ (1\leq j\leq n).
\end{align}
\item[(ii)] If $\gamma \in (0,1),$ then the Wright function\index{function!Wright}
$\Phi_{\gamma}(\cdot)$ is defined by
$$
\Phi_{\gamma}(z):=\sum \limits_{k=0}^{\infty}
\frac{(-z)^{k}}{k! \Gamma (1-\gamma -\gamma k)},\quad z\in {\mathbb C}.
$$
Since
$$
\int^{+\infty}_{0}e^{-\lambda t}\gamma s t^{-1-\gamma}
\Phi_{\gamma}(t^{-\gamma}s)\, dt=e^{-\lambda^{\gamma}s},\quad \lambda \in {\mathbb C}_{+},\ s>0,
$$ 
the Fubini theorem and the asymptotic expansion for the Wright functions \cite[p. 67, l. -1]{FKP} imply that, for every $s_{1}>0,$ ..., $s_{n}>0,$ $\gamma_{1} \in (0,1),$ ..., $\gamma_{n}\in (0,1),$ we have
\begin{align}\notag
\int \limits^{+\infty}_{0} &\int \limits^{+\infty}_{0}...\int \limits^{+\infty}_{0}  e^{-\lambda_{1} t_{1}-...-\lambda_{n}t_{n}}\\\label{l1} & \times \prod_{j=1}^{n}\gamma_{j} s_{j} t_{j}^{-1-\gamma_{j}}
\Phi_{\gamma_{j}}\bigl(t_{j}^{-\gamma_{j}}s_{j}\bigr)\, dt_{1}...\, dt_{n}=\prod_{j=1}^{n}e^{-\lambda_{j}^{\gamma_{j}}s_{j}},\quad \lambda_{j} \in {\mathbb C}_{+} \  \ (1\leq j\leq n).
\end{align}
\end{itemize}
Moreover, the convergence of Laplace integrals in \eqref{l} and \eqref{l1} is absolute.
\end{example}

The part (i) of subsequent result extends the statement of \cite[Theorem 1.5.1]{a43} and the formula \cite[(34), p. 58]{FKP}, while the part (ii) extends the statement of \cite[Proposition 8.2.1(i)]{funkcionalne}:

\begin{prop}\label{analyticala}
Suppose that $f: [0,+\infty)^{n}\rightarrow X$ is locally integrable and $v_{j}\in {\mathbb N}_{0}$ for all $j\in {\mathbb N}_{n}.$ Then the following holds:
\begin{itemize}
\item[(i)] Let $\emptyset \neq \Omega \subseteq \Omega(f) \cap \Omega_{b}(f)$ be open and let $(\lambda_{1},...,\lambda_{n})\in \Omega.$ Then the mapping $F: \Omega \rightarrow X$ is analytic,
\begin{align}\label{kec}
\bigl(\lambda_{1},...,\lambda_{n}\bigr)\in \Omega \Biggl(\frac{\cdot_{1}^{v_{1}}}{v_{1}!}\cdot ...\cdot\frac{\cdot_{n}^{v_{n}}}{v_{n}!}f\bigl( \cdot_{1},...,\cdot_{n}\bigr)\Biggr)
\end{align} 
and we have
\begin{align}\notag &
F^{(v_{1},...,v_{n})}(\lambda_{1},...,\lambda_{n})=(-1)^{v_{1}+...+v_{n}} \lim_{k_{1}\rightarrow +\infty;...;k_{n}\rightarrow +\infty; \ (k_{1},...,k_{n})\in {\mathbb N}_{0}^{n}}\\\label{izvodi}& \Biggl[ \int^{k_{1}}_{0}...\int^{k_{n}}_{0}e^{-\lambda_{1}t_{1}-...-\lambda_{n}t_{n}}\frac{t_{1}^{v_{1}}}{v_{1}!}\cdot ... \cdot\frac{t_{n}^{v_{n}}}{v_{n}!} f\bigl(t_{1},...,t_{n}\bigr)\, dt_{1}\, ...\, dt_{n} \Biggr].
\end{align}
\item[(ii)] Let $\emptyset \neq \Omega \subseteq \Omega_{abs}(f)$ be open and let $(\lambda_{1},...,\lambda_{n})\in \Omega.$ Then the mapping $F: \Omega \rightarrow X$ is analytic, $$\bigl(\lambda_{1},...,\lambda_{n}\bigr)\in \Omega_{abs}\Biggl(\frac{\cdot_{1}^{v_{1}}}{v_{1}!}\cdot ...\cdot\frac{\cdot_{n}^{v_{n}}}{v_{n}!}f\bigl( \cdot_{1},...,\cdot_{n}\bigr)\Biggr)$$ and we have
\begin{align*}
F^{(v_{1},...,v_{n})}(\lambda_{1},...,\lambda_{n})=(-1)^{v_{1}+...+v_{n}}\Biggl({\mathcal L}\Biggl[\frac{\cdot_{1}^{v_{1}}}{v_{1}!}\cdot ...\cdot\frac{\cdot_{n}^{v_{n}}}{v_{n}!}f\bigl( \cdot_{1},...,\cdot_{n}\bigr)\Biggr]\Biggr)(\lambda_{1},...,\lambda_{n}).
\end{align*}
\end{itemize}
\end{prop}

\begin{proof} 
In order to prove (i), we will slightly modify the proof of \cite[Theorem 1.5.1]{a43} given in the one-dimensional setting.
First of all, let us emphasize that the multiple integral in \eqref{izvodi} is well-defined due to Lemma \ref{bnm1}(i). If $k_{j}\in {\mathbb N}$ for all $j\in {\mathbb N}_{n}$, then we define the mapping $q_{k_{1},...,k_{n}} : \Omega \rightarrow X$ by{\small
$$
q_{k_{1},...,k_{n}}\bigl(z_{1},...,z_{n}\bigr):=\int^{k_{1}}_{0}...\int^{k_{n}}_{0}e^{-z_{1}t_{1}-...-z_{n}t_{n}} f\bigl(t_{1},...,t_{n}\bigr)\, dt_{1}\, ...\, dt_{n},\ \bigl(z_{1},...,z_{n}\bigr)\in \Omega.
$$}
It is clear that the mapping $q_{k_{1},...,k_{n}} : \Omega \rightarrow X$ is analytic since the mapping
{\tiny
$$
q_{k_{1},...,k_{n};N}\bigl(z_{1},...,z_{n}\bigr):=\sum_{j=0}^{N}\int^{k_{1}}_{0}...\int^{k_{n}}_{0}\frac{(-z_{1}t_{1}-...-z_{n}t_{n}\bigr)^{j}}{j!} f\bigl(t_{1},...,t_{n}\bigr)\, dt_{1}\, ...\, dt_{n},\ \bigl(z_{1},...,z_{n}\bigr)\in \Omega
$$}
is analytic for every $N\in {\mathbb N}$, the Weierstrass theorem holds in our framework and 

$$
q_{k_{1},...,k_{n}}\bigl(z_{1},...,z_{n}\bigr):=\lim_{N\rightarrow +\infty}q_{k_{1},...,k_{n};N}\bigl(z_{1},...,z_{n}\bigr),\quad \bigl(z_{1},...,z_{n}\bigr)\in \Omega,
$$
uniformly on compact subsets of $\Omega.$ Furthermore, a relatively simple argumentation involving the polynomial formula and formula \eqref{osgud} shows that
\begin{align} \notag &
q_{k_{1},...,k_{n}}
^{(v_{1},...,v_{n})}\bigl(z_{1},...,z_{n}\bigr):=\lim_{N\rightarrow +\infty}q_{k_{1},...,k_{n};N}^{(v_{1},...,v_{n})}\bigl(z_{1},...,z_{n}\bigr)
\\\label{izvodi0} &  = \int^{k_{1}}_{0}...\int^{k_{n}}_{0}e^{-z_{1}t_{1}-...-z_{n}t_{n}} \frac{t_{1}^{v_{1}}}{v_{1}!}\cdot ...\cdot\frac{t_{n}^{v_{n}}}{v_{n}!} f\bigl(t_{1},...,t_{n}\bigr)\, dt_{1}\, ...\, dt_{n} ,\quad \bigl(z_{1},...,z_{n}\bigr)\in \Omega.
\end{align}
In order to prove that the mapping $F: \Omega \rightarrow X$ is analytic and \eqref{izvodi} holds, we can use the formula \eqref{izvodi0}, the formula \eqref{osgud} and the fact that 
$$
\lim_{k_{1}\rightarrow +\infty;...;k_{n}\rightarrow +\infty; \ (k_{1},...,k_{n})\in {\mathbb N}_{0}^{n}}q_{k_{1},...,k_{n}}\bigl(\lambda_{1},...,\lambda_{n}\bigr)=F\bigl(\lambda_{1},...,\lambda_{n}\bigr),\ \bigl(\lambda_{1},...,\lambda_{n}\bigr)\in \Omega,
$$
uniformly on compact subsets of $\Omega.$ We will prove the last assertion only in two-dimensional setting. Let $K\subseteq \Omega$ be a compact set, let $p\in  \circledast $ and $\epsilon>0$ be given, and let $(\lambda_{1}^{0},\lambda_{2}^{0})\in \Omega$ satisfy $\Re \lambda_{j}>\Re \lambda_{j}^{0}>\Re \lambda_{j}-\delta(\epsilon)$ and $\Im \lambda_{j}^{0}=\Im \lambda_{j}$ for $j=1,2,$ where the number $\delta(\epsilon)>0$ will be determined a little bit later. If $T>k_{1}$ and $t>k_{2},$ then we have:
\begin{align*} &
p\Bigl( q_{k_{1},...,k_{n}}\bigl(\lambda_{1},\lambda_{2}\bigr)-F\bigl(\lambda_{1},\lambda_{2}\bigr) \Bigr)
\\& =p\Biggl(  \int^{k_{1}}_{0}\int^{t}_{0}e^{-\lambda_{1}t_{1}-\lambda_{2}t_{2}}f\bigl( t_{1},t_{2}\bigr)\, dt_{1}\, dt_{2}- \int^{k_{1}}_{0}\int^{k_{2}}_{0}e^{-\lambda_{1}t_{1}-\lambda_{2}t_{2}}f\bigl( t_{1},t_{2}\bigr)\, dt_{1}\, dt_{2}
\\& +  \int^{T}_{0}\int^{t}_{0}e^{-\lambda_{1}t_{1}-\lambda_{2}t_{2}}f\bigl( t_{1},t_{2}\bigr)\, dt_{1}\, dt_{2}- \int^{k_{1}}_{0}\int^{t}_{0}e^{-\lambda_{1}t_{1}-\lambda_{2}t_{2}}f\bigl( t_{1},t_{2}\bigr)\, dt_{1}\, dt_{2}\Biggr) .
\end{align*}
Using the equality \eqref{FDET}, we get that there exists a sufficiently small number $\delta_{1}(\epsilon)>0,$ independent of the choice of point $(\lambda_{1},\lambda_{2})\in K,$ such that the sum of  $p$-values of the first three addends in \eqref{FDET} for any integral of function $e^{-\lambda_{1}\cdot_{1}-\lambda_{2}\cdot_{2}}f( \cdot_{1},\cdot_{2})$ with bounds $\int^{k_{1}}_{0}\int^{t}_{0}...$, $\int^{k_{1}}_{0}\int^{k_{2}}_{0},$ $ \int^{T}_{0}\int^{t}_{0}$ and $\int^{k_{1}}_{0}\int^{t}_{0}...$ and the meaning clear, is less or equal than $\epsilon/2.$  Since $|\lambda_{j}-\lambda_{j}^{0}|=\Re \lambda_{j}-\Re \lambda_{j}^{0}$ for $j=1,2,$ it follows that there exists a sufficiently small real number $\delta_{2}(\epsilon)>0,$ independent of the choice of point $(\lambda_{1},\lambda_{2})\in K,$ such that the sum of  $p$-values of the fourth addends in \eqref{FDET} for any integral of function $e^{-\lambda_{1}\cdot_{1}-\lambda_{2}\cdot_{2}}f( \cdot_{1},\cdot_{2})$ with bounds $\int^{k_{1}}_{0}\int^{t}_{k_{2}}...$ (we estimate here $\int^{k_{1}}_{0}$ with $\int^{+\infty}_{0}$ for the first variable) and $\int_{k_{1}}^{T}\int^{t}_{0}$ (we estimate here $\int^{t}_{0}$ with $\int^{+\infty}_{0}$ for the second variable), is less or equal than $\epsilon/2.$  Now the required assertion follows by plugging $\delta(\epsilon):=\min(\delta_{1}(\epsilon),\delta_{2}(\epsilon)).$ Using the Cauchy criterium of convergence and a similar argumentation, we can prove that \eqref{kec} holds.
 The proof of (ii) is simple and therefore omitted.
\end{proof}

\begin{rem}\label{neide}
Let us note that it is not clear how one can prove that the requirements in (i) imply that
\begin{align} \label{kecb}
\bigl(\lambda_{1},...,\lambda_{n}\bigr)\in \Omega_{b}\Biggl(\frac{\cdot_{1}^{v_{1}}}{v_{1}!}\cdot ...\cdot\frac{\cdot_{n}^{v_{n}}}{v_{n}!}f\bigl( \cdot_{1},...,\cdot_{n}\bigr)\Biggr);
\end{align} 
cf. also \cite[Theorem 2.1(d)-(e)]{bern1} and \cite{amerio}. To the best knowledge of the author, we can only prove that the additional assumption
\begin{align*}
\bigl(\lambda_{1},...,\lambda_{n}\bigr)\in \Omega_{b} \Biggl(\frac{\cdot_{1}^{v_{1}}}{v_{1}!}\cdot ...\cdot\frac{\cdot_{n}^{v_{n}}}{v_{n}!}f\bigl( \cdot_{1},...,\cdot_{n}\bigr)\Biggr)
\end{align*} 
and all other requirements in (i) imply \eqref{kecb}.
The proof folows from an application of formula \eqref{FDET} with the function $f(\cdot_{1},...,\cdot_{n})$ replaced therein with the function $[(\cdot_{1}^{v_{1}})/(v_{1}!)\cdot ...\cdot (\cdot_{n}^{v_{n}})/(v_{n}!)]\cdot f( \cdot_{1},...,\cdot_{n}).$ 
\end{rem}

Now we will state and prove the following result concerning the Laplace transform of convolution product of Faltung $\ast_{0}:$

\begin{prop}\label{lapkonv}
Suppose that $a\in L_{loc}^1([0,\infty)^{n})$, $f\in L_{loc}^1([0,\infty)^{n}:X)$ and $(\lambda_{1},....,\lambda_{n})\in \Omega_{abs}(a) \cap \Omega_{abs}(f).$
\begin{itemize}
\item[(i)] Let $X$ be a Fr\' echet space. Then $(a\ast_{0}f)(\cdot)\in L_{loc}^{1}([0,+\infty)^{n})$ and  
\begin{align}\label{konvolucije}
\Bigl({\mathcal L}\bigl(a \ast_{0} f\bigr)\Bigr)\bigl(\lambda_{1},...,\lambda_{n}\bigr)=\bigl({\mathcal L}a\bigr)\bigl(\lambda_{1},...,\lambda_{n}\bigr) \cdot \bigl({\mathcal L}f\bigr) \bigl(\lambda_{1},...,\lambda_{n}\bigr).
\end{align}
\item[(ii)] Suppose, in addition, that
$a\in C([0,\infty)^{n})$ or $f\in C([0,\infty)^{n}:X)$. Then we have $a\ast_{0}f\in  L_{loc}^1([0,\infty)^{n}:X)$ and \eqref{konvolucije}.
\end{itemize}
\end{prop}

\begin{proof}
We will consider the two-dimensional setting, only. If $X$ is a Fr\' echet space, then we have already explained that $(a\ast_{0}f)(\cdot)\in L_{loc}^{1}([0,+\infty)^{n}).$ Keeping in mind the Fubini theorem clarified in Lemma \ref{ft}(ii) as well as Lemma \ref{opera}(ii) and the proof of \cite[Theorem 4.1]{debnath}, the assertion simply follows if we prove that the integral
$$
I=\int_{0}^{+\infty}\int^{+\infty}_{0}e^{-\lambda_{1}x-\lambda_{2}y}\Biggl[ \int^{x}_{0}\int^{y}_{0}a(x-u,y-v)f(u,v)\, du\, dv\Biggr] \, dx\, dy
$$
is absolutely convergent. Towards this end, it suffices to see that for each seminorm $p\in \circledast$ we have:
\begin{align*}
& p(I)\leq \int_{0}^{+\infty}\int^{+\infty}_{0}e^{-\Re \lambda_{1}x-\Re \lambda_{2}y}\Biggl[ \int^{x}_{0}\int^{y}_{0}|a|(x-u,y-v)p\bigl(f(u,v)\bigr)\, du\, dv\Biggr] \, dx\, dy
\\& =\int_{0}^{+\infty}\int^{+\infty}_{0}e^{-\Re \lambda_{1}(x-u)-\Re \lambda_{2}(y-v)}\\& \times \Biggl[ \int^{x}_{0}\int^{y}_{0}|a|(x-u,y-v)e^{-\Re \lambda_{1}u-\Re \lambda_{2}v}p\bigl(f(u,v)\bigr)\, du\, dv\Biggr] \, dx\, dy
\\& =\int_{0}^{+\infty}\int^{+\infty}_{0}\Biggl[ \int^{+\infty}_{u}\int^{+\infty}_{v}e^{-\Re \lambda_{1}(x-u)-\Re \lambda_{2}(y-v)}|a|(x-u,y-v)\, dx\, dy\Biggr] 
\\& \times e^{-\Re \lambda_{1}u-\Re \lambda_{2}v}p\bigl(f(u,v)\bigr)\, du\, dv
\\& =\Biggl[ \int^{+\infty}_{0}\int^{+\infty}_{0}e^{-\Re \lambda_{1}(x)-\Re \lambda_{2}(y)}|a|(x,y)\, dx\, dy\Biggr] 
\\& \cdot  \Biggl[ \int_{0}^{+\infty}\int^{+\infty}_{0}e^{-\Re \lambda_{1}u-\Re \lambda_{2}v}p\bigl(f(u,v)\bigr)\, du\, dv\Biggr] <+\infty.
\end{align*}
The proof of (ii) is similar and can be deduced by using the corresponding functionals.
\end{proof}

\begin{rem}\label{poginuosam}
It seems very plausible that the assumption $(\lambda_{1},....,\lambda_{n})\in \Omega_{abs}(f)$ in the formulation of Proposition \ref{lapkonv} can be weakened by assuming that $(\lambda_{1},....,\lambda_{n})\in \Omega(f) \cap \Omega_{b}(f)$; cf. \cite[Proposition 1.6.4]{a43} and \cite[Theorem 1.4.2(vi)]{FKP} for the one-dimensional setting, as well as \cite[Theorem 5.2]{bern1} for the two-dimensional setting.
\end{rem}

\subsection{Multidimensional vector-valued Laplace transform of holomorphic functions of several variables}\label{neubrander}

The study of Laplace transform of analytic vector-valued functions starts probably with the paper \cite{sova}
by M. Sova; cf. also the important research article \cite{frank} by F. Neubrander. We will first state
the following simple result of Tauberian type, which is independent of holomorphy (cf. the proof of \cite[Theorem 2.6.4, p. 89]{a43}); cf. also the research articles \cite{tauber01,tauber021} by U. Stadtm\"uller:

\begin{prop}\label{tauber}
Suppose that $x\in X$, $f: [0,+\infty)^{n}\rightarrow X$ is a locally Lebesgue integrable function and there exist real numbers $\omega_{j} \in {\mathbb R}$ ($1\leq j\leq n$) such that the set $\{ e^{-\omega_{1}t_{1}-...-\omega_{n}t_{n}}f(t_{1},...,t_{n}) : t_{1}\geq 0,...,t_{n}\geq 0 \}$ is bounded in $X.$ Then the following holds:
\begin{itemize}
\item[(i)] If $\lim_{(t_{1},...,t_{n})\rightarrow (0,...,0)}f(t_{1},...,t_{n})=x$, then\\ $\lim_{\lambda_{1}\rightarrow +\infty;...;\lambda_{n}\rightarrow +\infty}[\lambda_{1}\cdot ... \cdot \lambda_{n}\cdot F(\lambda_{1},...,\lambda_{n})]=x.$
\item[(ii)] Suppose that $\omega_{1}=...=\omega_{n}=0.$ If $\lim_{t_{1}\rightarrow +\infty;...;t_{n}\rightarrow +\infty}f(t_{1},...,t_{n})=x$, then $\lim_{(\lambda_{1},...,\lambda_{n}) \rightarrow (0,...,0)}[\lambda_{1}\cdot ... \cdot \lambda_{n}\cdot  F(\lambda_{1},...,\lambda_{n})]=x.$
\end{itemize}
\end{prop}

Now we will extend the statements of \cite[Proposition 2.6.3]{a43} and \cite[Theorem 1.4.10(ii)]{FKP} to the holomorphic vector-valued functions of several variables:

\begin{prop}\label{holo}
Suppose that $0<\alpha_{j}\leq \pi$ for all $j\in {\mathbb N}_{n}$, the function $f : \Sigma_{\alpha_{1}} \times ... \times  \Sigma_{\alpha_{n}} \rightarrow X$ is holomorphic and the assumption $\beta_{j}\in (0,\alpha_{j})$ for all $ j\in {\mathbb N}_{n}$ implies that the set $\{f(z_{1},...,z_{n}) : z_{j}\in \Sigma_{\beta_{j}}\mbox{ for all }j\in {\mathbb N}_{n}\}$ is bounded in $X.$
\begin{itemize}
\item[(i)] Let $\beta_{j}\in (0,\alpha_{j})$, $ j\in {\mathbb N}_{n}$. Then we have $\lim_{t_{1}\rightarrow +\infty;...;t_{n}\rightarrow +\infty}f(t_{1},...,t_{n})=x$ if and only if $\lim_{|z_{1}|\rightarrow +\infty;...;|z_{n}|\rightarrow +\infty; (z_{1},...,z_{n})\in \Sigma_{\beta_{1}}\times .... \times \Sigma_{\beta_{n}}}f(z_{1},...,z_{n})=x.$
\item[(ii)]  Let $\beta_{j}\in (0,\alpha_{j})$, $ j\in {\mathbb N}_{n}$. Then we have $\lim_{(t_{1},...,t_{n})\rightarrow (0,...,0)}f(t_{1},...,t_{n})=x$ if and only if $\lim_{(z_{1},...,z_{n})\rightarrow (0,...,0); (z_{1},...,z_{n})\in \Sigma_{\beta_{1}}\times .... \times \Sigma_{\beta_{n}}}f(z_{1},...,z_{n})=x.$
\end{itemize}
\end{prop}

\begin{proof}
It is clear that (i) follows if we prove that $\lim_{t_{1}\rightarrow +\infty;...;t_{n}\rightarrow +\infty}f(t_{1},...,t_{n})=x$ implies $\lim_{|z_{1}|\rightarrow +\infty;...;|z_{n}|\rightarrow +\infty; (z_{1},...,z_{n})\in \Sigma_{\beta_{1}}\times .... \times \Sigma_{\beta_{n}}}f(t_{1},...,t_{n})=x.$ Towards this end, observe that the set $I={\mathbb N}_{0}^{n}$, equipped with the order relation $(k_{1},...,k_{n}) \leq (l_{1},...,l_{n})$ if and only if $k_{j}\leq l_{j}$ for all $j\in {\mathbb N}_{ n}$ ($(k_{1},...,k_{n}) \in {\mathbb N}_{0}^{n};\  (l_{1},...,l_{n})\in {\mathbb N}_{0}^{n}$), is a directed set; see  \cite[p. 9]{meise} for the notion. Set $f_{k_{1},...,k_{n}}(z_{1},...,z_{n}):=f(k_{1}z_{1},...,k_{n}z_{n}),$ $(z_{1},...,z_{n}) \in \Sigma_{\alpha_{1}} \times ... \times  \Sigma_{\alpha_{n}},$ $(k_{1},...,k_{n}) \in {\mathbb N}_{0}^{n}.$  Then it is clear that for each $(z_{1}',....,z_{n}')\in \Sigma_{\beta_{1}}\times .... \times \Sigma_{\beta_{n}}$ there exists a sufficiently small real number $r>0$ such that the set $\{ f_{k_{1},...,k_{n}}(z_{1},...,z_{n}) : z_{j}\in \Sigma_{\beta_{j}}\mbox{ and }|z_{j}-z_{j}'|\leq r \mbox{ for all }j\in {\mathbb N}_{n}\}$ is bounded in $X.$ Furthermore, $\lim_{i\in I}f_{i}(t_{1},...,t_{n})=x$ for all $(t_{1},...,t_{n})\in (0,+\infty)^{n}.$ Applying the Vitali type theorem \cite[Theorem 3]{jorda}, we get that $\lim_{i\in I}f_{i}(z_{1},...,z_{n})=x$ locally uniformly in $\Sigma_{\alpha_{1}} \times ... \times  \Sigma_{\alpha_{n}}.$  In particular, if $\epsilon>0$ and $p\in  \circledast$ are given, then there exists $k_{0}\in {\mathbb N}$ such that the assumptions $(z_{1},...,z_{n})\in \Sigma_{\beta_{1}}\times .... \times \Sigma_{\beta_{n}}$, $k_{j}\geq k_{0}$  and $1\leq |z_{j}|\leq 2$ for all $j\in {\mathbb N}_{n}$ together imply $p(f_{k_{1},...,k_{n}}(z_{1},...,z_{n})-x)\leq \epsilon,$ so that $p(f(z_{1},...,z_{n})-x)\leq \epsilon$ for all $(z_{1},...,z_{n})\in \Sigma_{\beta_{1}}\times .... \times \Sigma_{\beta_{n}}$ such that $k_{j}\leq |z_{j}|\leq 2k_{j}$  for all $j\in {\mathbb N}_{n}$. This simply completes the proof of (i).
The proof of (ii) follows immediately from (i) by considering the function $f(1/z_{1},...,1/z_{n}).$
\end{proof}

The multidimensional vector-valued Laplace transform of holomorphic functions of several variables is considered in the following extension of \cite[Theorem 2.6.1]{a43} and \cite[Theorem 1.4.10(i)]{FKP}:

\begin{thm}\label{analytic}
Suppose that $\omega_{j}\in {\mathbb R}$ and $\alpha_{j}\in (0,\pi/2]$ for all $j\in {\mathbb N}_{n}$ as well as that $F: (\omega_{1},+\infty) \times ... \times (\omega_{n},+\infty)\rightarrow X$ is a given function. Then the following assertions are equivalent:
\begin{itemize}
\item[(i)] There exists a holomorphic function $f : \Sigma_{\alpha_{1}}\times ... \times \Sigma_{\alpha_{n}}\rightarrow X$ such that $({\mathcal L}f) (\lambda_{1},...,\lambda_{n}) =F(\lambda_{1},...,\lambda_{n})  $ for $\lambda_{1}>\omega_{1}, ...,$ $\lambda_{n}>\omega_{n}$ and, for every $\gamma_{1}\in (0,\alpha_{1}), ...,$ $ \gamma_{n}\in (0,\alpha_{n}),$ the set $\{ e^{-\omega_{1}z_{1}-....-z_{n}\omega_{n}}f(z_{1},...,z_{n}) : (z_{1},...,z_{n})\in \Sigma_{\gamma_{1}} \times ... \times \Sigma_{\gamma_{n}} \}$ is bounded in $X.$
\item[(ii)] There exists a holomorphic function $\tilde{F} : (\omega_{1}+\Sigma_{(\pi/2)+\alpha_{1}}) \times ... \times (\omega_{n}+\Sigma_{(\pi/2)+\alpha_{n}})\rightarrow X$ such that $\tilde{F}(\lambda_{1},...,\lambda_{n})  =F(\lambda_{1},...,\lambda_{n})  $ for $\lambda_{1}>\omega_{1}, ...,$ $\lambda_{n}>\omega_{n}$ and, for every $\gamma_{1}\in (0,\alpha_{1}), ...,$ $ \gamma_{n}\in (0,\alpha_{n}),$ the set $\{(\lambda_{1}-\omega_{1})\cdot ... \cdot (\lambda_{n}-\omega_{n})\tilde{F}(\lambda_{1},...,\lambda_{n})  : (\lambda_{1},...,\lambda_{n})\in (\omega_{1}+\Sigma_{(\pi/2)+\gamma_{1}}) \times ... \times (\omega_{n}+\Sigma_{(\pi/2)+\gamma_{n}}) \}$ is bounded in $X.$
\end{itemize}
If this is the case, then for every $(k_{1},...,k_{n})\in {\mathbb N}_{0}^{n}$, $\gamma_{1}\in (0,\alpha_{1}), ...,$ and $ \gamma_{n}\in (0,\alpha_{n}),$ the set $D:=\{\lambda_{1}^{k_{1}}\cdot ... \cdot \lambda_{n}^{k_{n}}e^{-\omega_{1}\lambda_{1}-....-\lambda_{n}\omega_{n}}f^{(k_{1},...,k_{n})} (\lambda_{1},...,\lambda_{n}) : (\lambda_{1},...,\lambda_{n})\in \Sigma_{\gamma_{1}} \times ... \times \Sigma_{\gamma_{n}}\}$  is bounded in $X$ and we have the following:
\begin{itemize}
\item[(a)] $\lim_{(t_{1},...,t_{n})\rightarrow (0+,...,0+)}f(t_{1},...,t_{n})=x$ if and only if $\lim_{\lambda_{1}\rightarrow +\infty;...;\lambda_{n}\rightarrow +\infty}[\lambda_{1}\cdot ... \cdot \lambda_{n}\cdot F(\lambda_{1},...,\lambda_{n})]=x.$
\item[(b)] Asssume that $\omega_{1}=....=\omega_{n}=0.$ Then $\lim_{t_{1}\rightarrow +\infty;...;t_{n}\rightarrow +\infty}f(t_{1},...,t_{n})=x$ if and only if $\lim_{(\lambda_{1},...,\lambda_{n})\rightarrow (0+,...,0+)}[\lambda_{1}\cdot ... \cdot \lambda_{n} \cdot F(\lambda_{1},...,\lambda_{n})]=x.$ 
\end{itemize}
\end{thm}

\begin{proof}
The proof is a technical modification of the proof given in the one-dimensional setting and we will only outline the main details. If  (i) holds, let $0<\beta_{n}<\alpha_{n}$ and $0<\epsilon<(\pi/2)-\beta_{n}$. Then we define the holomorphic extension of the function $F : \{\lambda_{1}\in {\mathbb C} : \Re \lambda>\omega_{1}\} \times ... \times \{\lambda_{n}\in {\mathbb C} : \Re \lambda>\omega_{n}\} \rightarrow X$ to the set $\{\lambda_{1}\in {\mathbb C} : \Re \lambda>\omega_{1}\} \times ... \times \{\lambda_{n-1}\in {\mathbb C} : \Re \lambda>\omega_{n-1}\} \times (\omega_{n}+\Sigma_{(\pi/2)+\beta_{n}-\epsilon})$ as in the proof of \cite[Theorem 2.6.1]{a43}:
\begin{align*} 
& \tilde{F}\bigl(\lambda_{1},...,\lambda_{n}\bigr)=\int^{+\infty}_{0}...\int^{+\infty}_{0}e^{-\lambda_{1}t_{1}-...-\lambda_{n}t_{n}}f\bigl(t_{1},...,t_{n}\bigr)\, dt_{1}\, ...\, dt_{n}
\\& = \int^{+\infty}_{0}...\int^{+\infty}_{0}e^{-\lambda_{1}t_{1}-...-\lambda_{n-1}t_{n-1}}
\\& \times \Biggl[e^{\pm i \beta_{n}}\int^{+\infty}_{0}e^{-\lambda_{n}s\exp(\pm i\beta_{n})}f\bigl(t_{1},...,t_{n-1},s\exp(\pm i\beta_{n})\bigr)\, ds\Biggr]\, dt_{1}\, ...\, dt_{n-1}.
\end{align*}
After that, for every $\lambda_{n}\in \omega_{n}+\Sigma_{(\pi/2)+\beta_{n}-\epsilon}$, we define the function $f_{\lambda_{n}} : \Sigma_{\alpha_{1}}\times ... \times \Sigma_{\alpha_{n-1}}\rightarrow X$ by
$$
f_{\lambda_{n}}\bigl(z_{1},...,z_{n-1}\bigr):=\Biggl[e^{\pm i \beta_{n}}\int^{+\infty}_{0}e^{-\lambda_{n}s\exp(\pm i\beta_{n})}f\bigl(z_{1},...,z_{n-1},s\exp(\pm i\beta_{n})\bigr)\, ds\Biggr],
$$
for $z_{j}\in \Sigma_{\alpha_{j}}$, where $1\leq j\leq n.$ 
Then $f_{\lambda_{n}}(\cdot)$ is holomorphic and the set\\
 $\{ e^{-\omega_{1}z_{1}-....-\omega_{n-1}z_{n-1}}(\lambda_{n}-\omega_{n})f_{\lambda_{n}}(z_{1},...,z_{n-1}) : (z_{1},...,z_{n-1})\in \Sigma_{\gamma_{1}} \times ... \times \Sigma_{\gamma_{n-1}},\ \lambda_{n}\in \omega_{n}+\Sigma_{(\pi/2)+\beta_{n}-\epsilon} \}$ is bounded in $X.$ Repeating this procedure, we obtain the holomorphic extension 
 $\tilde{F} : (\omega_{1}+\Sigma_{(\pi/2)+\alpha_{1}}) \times ... \times (\omega_{n}+\Sigma_{(\pi/2)+\alpha_{n}})\rightarrow X$ of function $F(\cdot)$ with the desired property.  

Assume now that (ii) holds. Let $\gamma_{1}\in (0,\alpha_{1}), ...,$ $ \gamma_{n}\in (0,\alpha_{n}),$ and let $0<\epsilon<\min(\gamma_{1},....,\gamma_{n-1}).$ Then we define
$$
f\bigl(z_{1},...,z_{n}\bigr):=\frac{1}{(2\pi i)^{n}}\int_{\Gamma_{1}}...\int_{\Gamma_{n}}e^{\lambda_{1} z_{1}+...+\lambda_{n}z_{n}}\tilde{F}\bigl(\lambda_{1},...,\lambda_{n}\bigr)\, d\lambda_{1} ... \, d\lambda_{n},
$$
for $z_{j}\in \Sigma_{\gamma_{j}},$ where $1\leq j\leq n$ and the contour $\Gamma_{j}$ being the same as in the proof of \cite[Theorem 2.6.1]{a43}, with the numbers $\gamma$ and $\omega$ replaced therein with the numbers $\gamma_{j}$ and $\omega_{j}$, respectively ($1\leq j\leq n$). Define
$
\tilde{F}_{n} : (\omega_{1}+\Sigma_{(\pi/2)+\alpha_{1}}) \times ... \times (\omega_{n-1}+\Sigma_{(\pi/2)+\alpha_{n-1}})\rightarrow X
$ through
$$
\tilde{F}_{n}\bigl(\lambda_{1},...,\lambda_{n-1}\bigr)  :=\frac{1}{2\pi i }\int_{\Gamma_{n}}e^{\lambda_{n} z_{n}}\tilde{F}\bigl(\lambda_{1},...,\lambda_{n}\bigr)\,d\lambda_{n},
$$
for $\lambda_{j}\in \omega_{j}+\Sigma_{(\pi/2)+\alpha_{j}},$ where $1\leq j\leq n-1.$
Using the estimates \cite[(2.12)-(2.15)]{a43} in the locally convex space setting, it follows that for each seminorm $p\in \circledast$ there exists a finite real number $M_{p}>0$ such that
\begin{align*}
p\Bigl( \tilde{F}_{n}\bigl(\lambda_{1},...,\lambda_{n-1}\bigr)\Bigr)
\leq M_{p}\bigl| \lambda_{1}-\omega_{1} \bigr|^{-1}\cdot ... \cdot \bigl| \lambda_{n-1}-\omega_{n-1} \bigr|^{-1}e^{\omega_{n}\Re z_{n}},
\end{align*}
for all $\lambda_{j}\in \omega_{j}+\Sigma_{(\pi/2)+\alpha_{j}}$, where $1\leq j\leq n-1.$
Since $\tilde{F}_{n}(\cdot)$ is holomorphic and
\begin{align*}
&\frac{1}{(2\pi i)^{n}}\int_{\Gamma_{1}}....\int_{\Gamma_{n}}e^{\lambda_{1} z_{1}+...+\lambda_{n}z_{n}}\tilde{F}\bigl(\lambda_{1},...,\lambda_{n}\bigr)\, d\lambda_{1} ... \, d\lambda_{n}
\\&= \frac{1}{(2\pi i)^{n-1}}\int_{\Gamma_{1}}....\int_{\Gamma_{n-1}}e^{\lambda_{1} z_{1}+...+\lambda_{n-1}z_{n-1}}\tilde{F}_{n}\bigl(\lambda_{1},...,\lambda_{n-1}\bigr)\, d\lambda_{1} ... \, d\lambda_{n-1},
\end{align*}
we can repeat the above procedure to deduce part (i).

If $(k_{1},...,k_{n})\in {\mathbb N}_{0}^{n}$, $\gamma_{1}\in (0,\alpha_{1}), ...,$ and $ \gamma_{n}\in (0,\alpha_{n}),$ then we can apply the Cauchy integral formula in polydiscs to show that the set $D$  is bounded in $X.$ In order to prove (a), we may assume without loss of generality that $\omega_{1}=....=\omega_{n}=0$ and $x=0.$ Assume that $0<\gamma_{j}<\beta_{j}<\alpha_{j}
$ for all $j\in {\mathbb N}_{n}.$ Due to Proposition \ref{holo}(i), we have 
$\lim_{|\lambda_{1}|\rightarrow +\infty;...;|\lambda_{n}|\rightarrow +\infty; (\lambda_{1},...,\lambda_{n})\in \Sigma_{\beta_{1}}\times .... \times \Sigma_{\beta_{n}}}[\lambda_{1}\cdot...\cdot \lambda_{n}\cdot \tilde{F}(\lambda_{1},...,\lambda_{n})]=0.$ Let $\epsilon>0$ and $p\in \circledast$ be given. Then there exists $M>0$ such that the assumptions $|\lambda_{j}|\geq M$ and $\lambda_{j}\in \Sigma_{\gamma_{j}}$ for all $j\in {\mathbb N}_{n}$ imply $p(\lambda_{1}\cdot...\cdot \lambda_{n}\cdot \tilde{F}(\lambda_{1},...,\lambda_{n}))\leq \epsilon.$ The remainder of proof can be given using the estimates given in the proof of \cite[Theorem 2.6.4, p. 88]{a43}; for simplicity, we will present all relevant details in the two-dimensional setting. By the proof of Theorem \ref{analytic}, we have
$$
f\bigl(t_{1},t_{2}\bigr):=\frac{1}{(2\pi i)^{2}}\int_{\Gamma_{1}} \int_{\Gamma_{2}}e^{\lambda_{1} t_{1}+ \lambda_{2}t_{2}}\tilde{F}\bigl(\lambda_{1}, \lambda_{2}\bigr)\, d\lambda_{1}  \, d\lambda_{2},
$$
for $t_{1}>0 $ and $t_{2}>0.$ Let $1/t_{1}\geq M,$ let $1/t_{2}\geq M$ and let the contours $\Gamma_{1}^{0}$ and $\Gamma_{2}^{0}$ be the portion of the circles with radius $1/t_{1}$ and $1/t_{2}$, respectively. By the proof of \cite[Theorem 2.6.4, p. 88]{a43} (we will basically use the same notation), it follows that:
\begin{align*}
& p\Bigl(f\bigl(t_{1},t_{2}\bigr) \Bigr)\leq p\Biggl(\frac{1}{2\pi i}\int_{\Gamma_{1}^{0}}e^{\lambda_{1}t_{1}}\frac{1}{2\pi i}\int_{-\gamma_{2}-(\pi/2)}^{\gamma_{2}+(\pi/2)}e^{e^{i\theta_{2}}}\tilde{F}\Bigl(\lambda_{1},\frac{e^{i\theta_{2}}}{t_{2}} \Bigr)\frac{e^{i\theta_{2}}}{t_{2}}\, d\theta_{2}\, d\lambda_{1}
\\& +\frac{1}{2\pi i}\int_{\Gamma_{1}^{\pm}}e^{\lambda_{1}t_{1}}\frac{1}{2\pi i}\int_{-\gamma_{2}-(\pi/2)}^{\gamma_{2}+(\pi/2)}e^{e^{i\theta_{2}}}\tilde{F}\Bigl(\lambda_{1},\frac{e^{i\theta_{2}}}{t_{2}} \Bigr)\frac{e^{i\theta_{2}}}{t_{2}}\, d\theta_{2}\, d\lambda_{1}
\\& +\frac{1}{2\pi i}\int_{\Gamma_{1}^{0}}e^{\lambda_{1}t_{1}}\frac{1}{2\pi i}\int_{1}^{+\infty}e^{s\exp(\pm i (\gamma_{2}+(\pi/2)))}\\& \times \tilde{F}\Biggl( \lambda_{1},\frac{s}{t_{2}}\exp(\pm i (\gamma_{2}+(\pi/2)))\Biggr)\frac{s}{t_{2}}\exp(\pm i (\gamma_{2}+(\pi/2)))\, \frac{ds}{s}\, d\lambda_{1}
\\& +\frac{1}{2\pi i}\int_{\Gamma_{1}^{\pm}}e^{\lambda_{1}t_{1}}\frac{1}{2\pi i}\int_{1}^{+\infty}e^{s\exp(\pm i (\gamma_{2}+(\pi/2)))}\\& \times \tilde{F}\Biggl( \lambda_{1},\frac{s}{t_{2}}\exp(\pm i (\gamma_{2}+(\pi/2)))\Biggr)\frac{s}{t_{2}}\exp(\pm i (\gamma_{2}+(\pi/2)))\, \frac{ds}{s}\, d\lambda_{1}\Biggr).
\end{align*} 
The first addend can be estimated as follows:
\begin{align*}
&p\Biggl(\frac{1}{2\pi i}\int_{\Gamma_{1}^{0}}e^{\lambda_{1}t_{1}}\frac{1}{2\pi i}\int_{-\gamma_{2}-(\pi/2)}^{\gamma_{2}+(\pi/2)}e^{e^{i\theta_{2}}}\tilde{F}\Bigl(\lambda_{1},\frac{e^{i\theta_{2}}}{t_{2}} \Bigr)\frac{e^{i\theta_{2}}}{t_{2}}\, d\theta_{2}\, d\lambda_{1}\Biggr)
\\& \leq \frac{1}{(2\pi)^{2}}\int_{-\gamma_{1}-(\pi/2)}^{\gamma_{1}+(\pi/2)}e^{\cos \theta_{1}}\int_{-\gamma_{2}-(\pi/2)}^{\gamma_{2}+(\pi/2)}e^{\cos \theta_{2}}p\Biggl(\tilde{F}\Bigl( \frac{e^{i\theta_{1}}}{t_{1}},\frac{e^{i\theta_{2}}}{t_{2}}\Bigr)\frac{ie^{i\theta_{2}}}{t_{2}}\frac{ie^{i\theta_{1}}}{t_{1}} \Biggr)\, d\theta_{2}\, d\theta_{1}
\\& \leq c_{p}\epsilon \Biggl[ \int_{-\gamma_{1}-(\pi/2)}^{\gamma_{1}+(\pi/2)}e^{\cos \theta_{1}} \, d\theta_{1}\Biggr] \cdot \Biggl[ \int_{-\gamma_{2}-(\pi/2)}^{\gamma_{2}+(\pi/2)}e^{\cos \theta_{2}} \, d\theta_{2} \Biggr],
\end{align*} 
for a certain positive real constant $c_{p}>0.$  The second addend can be estimated as follows:
\begin{align*}
&p\Biggl( \frac{1}{2\pi i}\int_{\Gamma_{1}^{\pm}}e^{\lambda_{1}t_{1}}\frac{1}{2\pi i}\int_{-\gamma_{2}-(\pi/2)}^{\gamma_{2}+(\pi/2)}e^{e^{i\theta_{2}}}\tilde{F}\Bigl(\lambda_{1},\frac{e^{i\theta_{2}}}{t_{2}} \Bigr)\frac{e^{i\theta_{2}}}{t_{2}}\, d\theta_{2}\, d\lambda_{1}\Biggr)
\\& = p\Biggl( \frac{1}{2\pi i}\int_{1}^{+\infty}e^{v\exp(\pm i (\gamma_{1}+(\pi/2)))}\frac{1}{2\pi i}\\& \int_{-\gamma_{2}-(\pi/2)}^{\gamma_{2}+(\pi/2)}e^{e^{i\theta_{2}}}\tilde{F}\Bigl( \frac{v\exp(\pm i (\gamma_{1}+(\pi/2)))}{t_{1}},\frac{ie^{i\theta_{2}}}{t_{2}} \Bigr)\frac{e^{i\theta_{2}}}{t_{2}}\frac{v\exp(\pm i (\gamma_{1}+(\pi/2)))}{t_{1}}\, d\theta_{2}\, \frac{dv}{v}\Biggr)
\\& \leq c_{p}'\epsilon \Biggl[ \int_{1}^{+\infty}e^{v\cos (\gamma_{1}+(\pi/2))}\,\frac{dv}{v} \Biggr] \cdot \Biggl[ \int_{-\gamma_{2}-(\pi/2)}^{\gamma_{2}+(\pi/2)}e^{\cos \theta_{2}} \, d\theta_{2} \Biggr],
\end{align*} 
for a certain positive real constant $c_{p}'>0.$
We can similarly estimate the third addend and the fourth addend, finishing the proof of (a). 
The issue (b) can be deduced analogously.
\end{proof}

\subsection{The uniqueness theorem for multidimensional vector-valued Laplace transform}\label{napolje}

In this subsection, we will analyze the uniqueness type results for the multi-dimensional vector-valued Laplace transform. 

We need the following notion:

\begin{defn}\label{leb}
Suppose that $f : [0,+\infty)^{n}\rightarrow X$ is a locally integrable function and $(t_{1},...,t_{n})\in [0,+\infty)^{n}.$ Then we say that $(t_{1},...,t_{n})$ is a Lebesgue point of $f(\cdot)$ if for each seminorm $p\in \circledast$ we have
$$
\lim_{h\rightarrow 0+}\frac{1}{h^{n}}\int^{h}_{0}...\int^{h}_{0}p\Bigl( f\bigl( t_{1}+h,....,t_{n}+h\bigr) - f\bigl( t_{1},....,t_{n}\bigr) \Bigr)\, dt_{1} ... \, dt_{n}=0.
$$
\end{defn}

It is clear that any point of continuity of $f(\cdot)$ is its Lebesgue point. 
Furthermore, by the proof of Lebesgue differentiation theorem (see, e.g., \cite[p. 157]{rudin}), one can show the following:
\begin{itemize}
\item[(LP)] For each seminorm $p\in \circledast,$ there exists a Lebesgue measurable set $N_{p}\subseteq [0,+\infty)^{n}$ such that $m(N_{p})=0$ and
$$
\lim_{h\rightarrow 0+}p\Biggl(\frac{1}{h^{n}}\int^{h}_{0}...\int^{h}_{0}f\bigl( t_{1}+h,....,t_{n}+h\bigr) \, dt_{1} ...  \, dt_{n}- f\bigl( t_{1},....,t_{n}\bigr)\Biggr)=0.
$$
\end{itemize}
If $X$ is a Fr\' echet space, this immediately implies that almost every point $(t_{1},...,t_{n})\in [0,+\infty)^{n}$ is a Lebesgue point of $f(\cdot)$.

Now we will state and prove the following uniqueness type result for the multi-dimensional vector-valued Laplace transform, which provides a proper extension of \cite[Theorem 1.4.5]{FKP}:

\begin{thm}\label{sasu}
Suppose that $f: [0,+\infty)^{n}\rightarrow X$ is a locally integrable function,  $(\lambda_{1}^{0},...,\lambda_{n}^{0})\in \Omega_{b}(f)$ and there exist real numbers $a_{1}>\Re \lambda_{1}^{0},..., a_{n}>\Re \lambda_{n}^{0}$ such that $( {\mathcal L} f) (\lambda_{1},...,\lambda_{n} )=0$ for $\lambda_{1}>a_{1},...,\lambda_{n}>a_{n}.$ Then, for every seminorm $p\in \circledast,$ there exists a Lebesgue measurable set $N_{p}\subseteq [0,+\infty)^{n}$ such that $m(N_{p})=0$ and $p(f(t_{1},...,t_{n}))=0$ for all $(t_{1},...,t_{n})\in  [0,+\infty)^{n} \setminus N_{p}.$  In particular, if $X$ is a Fr\'echet space, then $f(t_{1},...,t_{n})=0$ for a.e. $(t_{1},...,t_{n})\in [0,+\infty)^{n}. $
\end{thm}

\begin{proof}
We will consider the two-dimensional setting, only.
By Theorem \ref{sas}, the function $G(\cdot),$ given by \eqref{anatolia}, satisfies that for each seminorm $p\in \circledast$ there exists $M_{p}>0$ such that, for every $t_{1}\geq 0 $ and $t_{2}\geq 0,$  \eqref{pr} holds with $n=2.$ Then Theorem \ref{sas} implies that for each $(\lambda_{1}, \lambda_{2})\in {\mathbb C}^{2}$ with $\Re \lambda_{1}>\max(\Re \lambda_{1}^{0},0)$ and $\Re \lambda_{2}>\max(\Re \lambda_{2}^{0},0)$ we have $(\lambda_{1},\lambda_{2})\in \Omega_{abs}(G)$, $(\lambda_{1},\lambda_{2})\in \Omega (f) \cap \Omega_{b}(f)$ and
\begin{align*}
\bigl( {\mathcal L} f\bigr) \bigl(\lambda_{1},\lambda_{2} \bigr)=\lambda_{1} \lambda_{2} \bigl( {\mathcal L} G\bigr) \bigl(\lambda_{1},\lambda_{2} \bigr).
\end{align*}
Therefore, there exist sufficiently large positive real numbers $b_{1}>a_{1}$ and $b_{2}>a_{2}$ such that 
\begin{align*} 
( {\mathcal L} G) (\lambda_{1},\lambda_{2} )&=\int^{+\infty}_{0}\int^{+\infty}_{0}e^{-\lambda_{1}t_{1}-\lambda_{2}t_{2}}G\bigl(t_{1},t_{2} \bigr)\, dt_{1}\, dt_{2}
\\& =\int^{+\infty}_{0}e^{-\lambda_{1}t_{1}}\Biggl[\int^{+\infty}_{0}e^{-\lambda_{2}t_{2}}G\bigl(t_{1},t_{2} \bigr)\, dt_{2}\Biggr]\, dt_{1} =0,\ \lambda_{1}>a_{1},\ \lambda_{2}>a_{2}.
\end{align*} 
Since the mapping $t_{1}\mapsto \int^{+\infty}_{0}e^{-\lambda_{2}t_{2}}G(t_{1},t_{2} )\, dt_{2},$ $t_{1}\geq 0$ is continuous, the uniqueness theorem for the one-dimensional Laplace transform \cite[Theorem 1.4.5]{FKP} yields that 
$$
\int^{+\infty}_{0}e^{-\lambda_{2}t_{2}}G\bigl(t_{1},t_{2} \bigr)\, dt_{2}=0,\quad t_{1}\geq 0,\ \lambda_{2}>a_{2}.
$$
Since for each fixed number $t_{1}\geq 0$ the mapping $t_{2}\mapsto G(t_{1},t_{2} ),$ $t_{2}\geq 0$ is continuous, we can apply \cite[Theorem 1.4.5]{FKP} once more in order to see that $G(t_{1},t_{2} )=0$ for all $t_{1}\geq 0$ and $t_{2}\geq 0$. The final conclusions now follow from our previous analysis.
\end{proof}

We continue with the following result:

\begin{prop}\label{intr}
Suppose that  
$f: [0,+\infty)^{n}\rightarrow X$ is a locally integrable function, $(\lambda_{1}^{0},...,\lambda_{n}^{0})\in \Omega_{b}(f)$ and the function $G(\cdot)$ is given by \eqref{anatolia}. Then the following Post-Widder inversion formula holds:
\begin{align}\notag
G\bigl(t_{1},...,t_{n}\bigr)&=\lim_{k_{1}\rightarrow +\infty}...\lim_{k_{n}\rightarrow +\infty}\frac{(-1)^{k_{1}+...+k_{n}}}{k_{1}! \cdot ... \cdot k_{n}!}\Bigl( \frac{k_{1}}{t_{1}}\Bigr)^{k_{1}+1}\cdot ... \cdot \Bigl( \frac{k_{n}}{t_{n}}\Bigr)^{k_{n}+1}\\\label{pwg}& \times \frac{\partial^{k_{1}+...+k_{n}}({\mathcal L G})}{\partial x_{1}^{k_{1}}\cdot ... \cdot \partial x_{n}^{k_{n}}}\Bigl( \frac{k_{1}}{t_{1}},...,\frac{k_{n}}{t_{n}}\Bigr),\quad t_{1}>0,\ ...,\ t_{n}>0.
\end{align}
\end{prop}

\begin{proof}
We will consider the two-dimensional setting, only. If $(\lambda_{1},...,\lambda_{n})\in {\mathbb C}^{n}$ and $\Re \lambda_{j}>\Re \lambda_{j}^{0}$ for all $j\in {\mathbb N}_{n},$ then Theorem \ref{sas}(ii)-(iii) imply that $(\lambda_{1},...,\lambda_{n})\in \Omega_{abs}(G).$ 
Since
$$
\int^{+\infty}_{0}\int^{+\infty}_{0}e^{-\lambda_{1}t_{1}-\lambda_{2}t_{2}}G\bigl( t_{1},t_{2}\bigr)\, dt_{1}\, dt_{2}=\int^{+\infty}_{0}e^{-\lambda_{1}t_{1}}\Biggl[\int^{+\infty}_{0} e^{-\lambda_{2}t_{2}}G\bigl( t_{1},t_{2}\bigr)\, dt_{2}\Biggr] \, dt_{1}
$$
and the maping $t_{1}\mapsto \int^{+\infty}_{0} e^{-\lambda_{2}t_{2}}G( t_{1},t_{2})\, dt_{2},$ $t_{1}\geq 0$ is continuous for every fixed numbers $t_{1}>0$ and $\lambda_{2}>a_{2},$ where $a_{2}>0$ is a sufficiently large real number, the Post-Widder inversion formula \cite[Theorem 1.4.3]{FKP} yields
\begin{align*}
\int^{+\infty}_{0} e^{-\lambda_{2}t_{2}}G\bigl( t_{1},t_{2}\bigr)\, dt_{2}=\lim_{k\rightarrow +\infty}\frac{(-1)^{k}}{k! }\Bigl( \frac{k}{t_{1}}\Bigr)^{k+1} \frac{\partial^{k}({\mathcal L G})}{\partial x_{1}^{k}}\Bigl( t_{1},\frac{k}{t_{2}}\Bigr),
\end{align*}
for any $ t_{1}>0,\ t_{2}>0$ and $\lambda_{2}>a_{2}.$
Since the mapping $t_{2}\mapsto G( t_{1},t_{2}),$ $t_{2}\geq 0$ is continuous for every fixed number $t_{1}>0,$ we can apply the Post-Widder inversion formula \cite[Theorem 1.4.3]{FKP} once more to deduce the validity of \eqref{pwg}.
\end{proof}

\begin{rem}
Suppose that  
$f: [0,+\infty)^{n}\rightarrow X$ is a locally integrable function, $f_{| (0,+\infty)^{n}}(\cdot)$ is continuous and \eqref{zuca} holds. 
Then we can similarly prove the following multidimensional Post-Widder inversion formula:
\begin{align}\notag
f\bigl(t_{1},...,t_{n}\bigr)&=\lim_{k_{1}\rightarrow +\infty}...\lim_{k_{n}\rightarrow +\infty}\frac{(-1)^{k_{1}+...+k_{n}}}{k_{1}! \cdot ... \cdot k_{n}!}\Bigl( \frac{k_{1}}{t_{1}}\Bigr)^{k_{1}+1}\cdot ... \cdot \Bigl( \frac{k_{n}}{t_{n}}\Bigr)^{k_{n}+1}\\\label{pw}& \times \frac{\partial^{k_{1}+...+k_{n}}F}{\partial x_{1}^{k_{1}}\cdot ... \cdot \partial x_{n}^{k_{n}}}\Bigl( \frac{k_{1}}{t_{1}},...,\frac{k_{n}}{t_{n}}\Bigr),\quad t_{1}>0,\ ...,\ t_{n}>0;
\end{align}
cf. also \cite{iranac,cohen,ditkin,frolov} and references cited therein. It could be very tempting to prove the fiormula \eqref{pw} for a general Lebesgue point $(t_{1},...,t_{n})$ of a locaaly integrable function $f: [0,+\infty)^{n}\rightarrow X$.
\end{rem}

Now we are able to prove the following extension of \cite[Proposition 1.4.7]{FKP}:

\begin{prop}\label{piano-MLO}
Suppose that ${\mathcal A} : X\rightarrow P(Y)$ is an \emph{MLO}, ${\mathcal A}$ is $X_{\mathcal A} \times Y_{\mathcal A}$-closed,  
$f_{1}: [0,+\infty)^{n}\rightarrow X$ and $f_{2}: [0,+\infty)^{n}\rightarrow X$ are locally integrable functions, $(\lambda_{1}^{0},...,\lambda_{n}^{0})\in \Omega_{b}(f_{1}) \cap \Omega_{b}(f_{2})$ and there exist real numbers $a_{1}>\Re \lambda_{1}^{0},..., a_{n}>\Re \lambda_{n}^{0}$ such that $(( {\mathcal L} f_{1}) (\lambda_{1},...,\lambda_{n} ),( {\mathcal L} f_{2}) (\lambda_{1},...,\lambda_{n} ))\in {\mathcal A}$ for $\lambda_{1}>a_{1},...,\lambda_{n}>a_{n}.$
Then $(f_{1}(t_{1},...,t_{n}),f_{2}(t_{1},...,t_{n}))\in {\mathcal A}$ for any $(t_{1},...t_{n}) \in [0,+\infty)^{n}$ which is a Lebesgue point of both functions $f_{1}(\cdot)$ and $f_{2}(\cdot).$ 
\end{prop}

\begin{proof}
We will deduce the result for closed MLOs; the proof in general case is almost the same. By Theorem \ref{sas}(iii), we have $(\lambda_{1},...,\lambda_{n})\in \Omega_{abs}(G_{1}) \cap \Omega_{abs}(G_{2})$, $(\lambda_{1},...,\lambda_{n})\in \Omega (f_{1}) \cap \Omega_{b}(f_{1}) \cap \Omega (f_{2}) \cap \Omega_{b}(f_{2}) $ and \eqref{ej} holds with the functions $f(\cdot)$ and $G(\cdot)$ replaced therein with the functions $f_{1}(\cdot)$ and $G_{1}(\cdot)$ [$f_{2}(\cdot)$ and $G_{2}(\cdot)$]; here, the function $G_{1}(\cdot)$ [$G_{2}(\cdot)$] is given by replacing the function $f(\cdot)$ in \eqref{anatolia} with the function $f_{1}(\cdot)$ [$f_{2}(\cdot)$]. Since ${\mathcal A}$ is closed and the function $G_{1}(\cdot)$ [$G_{2}(\cdot)$] is continuous, we can apply Proposition \ref{intr} to show that $(G_{1}(t_{1},....,t_{n}), G_{2}(t_{1},....,t_{n}))\in {\mathcal A}$ for all $(t_{1},....,t_{n}) \in [0,+\infty)^{n}.$ After that, we can differentiate this inclusion (see, e.g., Theorem 12 in math.mcgill.ca /Teaching notes/Lebesgue diff) to see that $(f_{1}(t_{1},...,t_{n}),f_{2}(t_{1},...,t_{n}))\in {\mathcal A}$ for any $(t_{1},...t_{n}) \in [0,+\infty)^{n}$ which is a Lebesgue point of both functions $f_{1}(\cdot)$ and $f_{2}(\cdot).$ 
\end{proof}

The numerical inversion of multidimensional vector-valued Laplace transform has been considered in many research articles published by now (cf. \cite{abate,chakir,choud,moorthy,moorthy1,singal,wangzew} and references quoted therein).  We will analyze the numerical inversion of multidimensional vector-valued Laplace transform, the multidimensional Phragm\'en–Doetsch inversion of vector-valued Laplace transform and the inversion of  multidimensional vector-valued Laplace transform in $L_{p}(X)$-spaces somewhere else (\cite{weis}). 

\section{Applications to the abstract Volterra integro-differential inclusions with multiple variables}\label{dzubre}

In this section, we furnish some applications of our results to the abstract fractional partial differential inclusions and the abstract Volterra integro-differential inclusions with multiple variables. We divide our exposition into three individual subsections. 

\subsection{Holomorphic solutions to abstract fractional partial differential inclusions with Riemann-Liouville and Caputo derivatives}\label{nulice}

In this subsection, we continue our analysis from \cite[Subsection 6.1]{axioms}; for the notion of partial Riemann-Liouville fractional derivative
$D_{R}^{\alpha_{1}}  D_{R}^{\alpha_{2}}u$ and partial Caputo fractional derivative ${\mathbf D}_{C}^{\alpha_{1}}{\mathbf D}_{C}^{\alpha_{2}}u$ used below, we refer the reader to \cite[Section 2]{axioms}  (concerning the fractional calculus and fractional differential equations, we can recommend the research monographs and articles quoted in \cite{knjigaho}-\cite{funkcionalne}).

Suppose that $\alpha_{1}\in [0,2),$ $\alpha_{2}\in [0,2)$, $m_{1}=\lceil \alpha_{1} \rceil,$ $m_{2}=\lceil \alpha_{2} \rceil$ and ${\mathcal A}$ is a closed MLO in $X.$  
Let us consider the following abstract two-dimensional Cauchy inclusion 
\begin{align}\label{prvarl}
D_{R}^{\alpha_{1}}& D_{R}^{\alpha_{2}}u\bigl( x_{1},x_{2} \bigr)\in {\mathcal A}u\bigl( x_{1},x_{2} \bigr)+f\bigl( x_{1},x_{2} \bigr),\quad x_{1}\geq 0,\ x_{2}\geq 0,
\end{align}
subjected with the initial conditions of the form
\begin{align}\label{prvarl1}
&\frac{\partial ^{k}}{\partial x_{2}^{k}}\Biggl[J_{t_{2}}^{m_{2}-\alpha_{2}}\ast_{0} u\Biggr]\bigl( x_{1},0 \bigr)=f_{k}\bigl( x_{1} \bigr),\quad 0\leq k\leq m_{2}-1;
\\\label{prvarl2}& \int_{0}^{x_{2}}g_{{\alpha_{2}}}\bigl(x_{2}-s\bigr)\Biggl[ \frac{\partial ^{k}}{\partial x_{1}^{k}}
\Biggl[J_{t_{1}}^{m_{1}-\alpha_{1}}\ast_{0}
D_{R}^{\alpha_{2}}u\Biggr]\bigl( x_{1},x_{2}\bigr)\Biggr]_{x_{1}=0,x_{2}=s}\, ds =h_{k}\bigl( x_{2} \bigr),
\end{align}
for $0\leq k\leq m_{1}-1,$
and the following abstract two-dimensional Cauchy inclusion
\begin{align}\label{prvac}
{\mathbf D}_{C}^{\alpha_{1}}&{\mathbf D}_{C}^{\alpha_{2}}u\bigl( x_{1},x_{2} \bigr)\in {\mathcal A}u\bigl( x_{1},x_{2} \bigr)+f\bigl( x_{1},x_{2} \bigr),\quad x_{1}\geq 0,\ x_{2}\geq 0,
\end{align}
subjected with the initial conditions of the form
\begin{align}\label{prvac1}
&\frac{\partial ^{k}}{\partial x_{2}^{k}}u\bigl( x_{1},0 \bigr)=f_{k}\bigl( x_{1} \bigr),\quad 0\leq k\leq m_{2}-1;
\\\label{prvac2}&\int_{0}^{x_{2}}g_{{\alpha_{2}}}\bigl(x_{2}-s\bigr)\Biggl[ \frac{\partial ^{k}}{\partial x_{1}^{k}}{\mathbf D}_{C}^{\alpha_{2}}u\bigl( x_{1},x_{2}\bigr)\Biggr]_{x_{1}=0,x_{2}=s}\, ds=h_{k}\bigl( x_{2} \bigr),\quad 0\leq k\leq m_{1}-1.
\end{align}
The notion of a (strong) solution to the inclusions  [\eqref{prvarl}; \eqref{prvarl1}-\eqref{prvarl2}] and [\eqref{prvac}; \eqref{prvac1}-\eqref{prvac2}]. has been introduced in \cite[Definition 4]{axioms}; in \cite[Theorem 3]{axioms}, we have clarified some sufficient conditions for the existence and uniqueness of (strong) solutions to these inclusions, where we have essentially applied the complex characterization theorem for double Laplace transform.

In this part, we want to present an illustrative application of Theorem \ref{analytic} in the analysis of existence and uniqueness of holomorphic solutions to the inclusions [\eqref{prvarl}; \eqref{prvarl1}-\eqref{prvarl2}] and [\eqref{prvac}; \eqref{prvac1}-\eqref{prvac2}]. Suppose that the following conditions hold (by $D_{1}$ we denote the set of all integers $k\in {\mathbb N}_{m_{2}-1}^{0}$ such that $f_{k}(\cdot)$ is not identically equal to the zero function and by $D_{2}$ we denote the set of all integers $k\in {\mathbb N}_{m_{1}-1}^{0}$ such that $h_{k}(\cdot)$ is not identically equal to the zero function):
\begin{itemize}
\item[(i)] There exist real numbers $\theta_{1} \in (0,\pi/2]$, $\theta_{2} \in (0,\pi/2]$, $\theta \in (0,\pi]$ and $\beta \in (0,1]$ such that $\theta=((\pi/2)+\theta_{1})\alpha_{1}+((\pi/2)+\theta_{2})\alpha_{2}$ and, for every number $\theta'\in (0,\theta),$ the family  
\begin{align}\label{rate}
\Bigl\{ \bigl(1+|z|\bigr)^{\beta}\bigl( z -{\mathcal A}\bigr)^{-1}C : z\in \Sigma_{\theta'}\Bigr\}\subseteq L(X)\mbox{ is equicontinuous.}   
\end{align}
\item[(ii)] For every $k\in D_{1},$ $\theta_{1}' \in (0,\theta_{1})$ and $\theta_{2}' \in (0,\theta_{2}),$ the operator family{\small
\begin{align*}
\Bigl\{ |\lambda_{1}|^{\alpha_{1}+1}|\lambda_{2}|^{m_{2}-k}\bigl( \lambda_{1}^{\alpha_{1}} \lambda_{2}^{\alpha_{2}} -{\mathcal A}\bigr)^{-1}C \bigl({\mathcal L}f_{k}\bigr)(\lambda_{1}): \lambda_{1}\in \Sigma_{(\pi/2)+\theta_{1}'},\ \lambda_{2}\in \Sigma_{(\pi/2)+\theta_{2}'}\Bigr\}\subseteq L(X)
\end{align*}}
is equicontinuous,
resp.
the operator family{\small
\begin{align*}
\Bigl\{ |\lambda_{1}|^{\alpha_{1}+1}|\lambda_{2}|^{\alpha_{2}-k}\bigl( \lambda_{1}^{\alpha_{1}} \lambda_{2}^{\alpha_{2}} -{\mathcal A}\bigr)^{-1}C \bigl({\mathcal L}f_{k}\bigr)(\lambda_{1}): \lambda_{1}\in \Sigma_{(\pi/2)+\theta_{1}'},\ \lambda_{2}\in \Sigma_{(\pi/2)+\theta_{2}'}\Bigr\}\subseteq L(X)
\end{align*}}
is equicontinuous.
\item[(iii)] For every $k\in D_{2},$ $\theta_{1}' \in (0,\theta_{1})$ and $\theta_{2}' \in (0,\theta_{2}),$ the operator family{\small
\begin{align*}
\Bigl\{ |\lambda_{1}|^{m_{1}-k}|\lambda_{2}|^{\alpha_{2}+1}\bigl( \lambda_{1}^{\alpha_{1}} \lambda_{2}^{\alpha_{2}} -{\mathcal A}\bigr)^{-1}C \bigl({\mathcal L}h_{k}\bigr)(\lambda_{2}): \lambda_{1}\in \Sigma_{(\pi/2)+\theta_{1}'},\ \lambda_{2}\in \Sigma_{(\pi/2)+\theta_{2}'}\Bigr\}\subseteq L(X)
\end{align*}}
is equicontinuous,
resp.
the operator family{\small
\begin{align*}
\Bigl\{ |\lambda_{1}|^{\alpha_{1}-k}|\lambda_{2}|^{\alpha_{2}+1}\bigl( \lambda_{1}^{\alpha_{1}} \lambda_{2}^{\alpha_{2}} -{\mathcal A}\bigr)^{-1}C \bigl({\mathcal L}h_{k}\bigr)(\lambda_{2}): \lambda_{1}\in \Sigma_{(\pi/2)+\theta_{1}'},\ \lambda_{2}\in \Sigma_{(\pi/2)+\theta_{2}'}\Bigr\}\subseteq L(X)
\end{align*}}
is equicontinuous.
\item[(iv)] $({\mathcal L}f)(\lambda_{1},\lambda_{2})$ exists for $  \lambda_{1} \in \Sigma_{(\pi/2)+\theta_{1}}$, $  \lambda_{2}\in \Sigma_{(\pi/2)+\theta_{2}}$ and, for every numbers $\theta_{1}' \in (0,\theta_{1})$ and $\theta_{2}' \in (0,\theta_{2}),$ the operator family
\begin{align*}
\Bigl\{ \lambda_{1}\lambda_{2}\bigl( \lambda_{1}^{\alpha_{1}} \lambda_{2}^{\alpha_{2}} -{\mathcal A}\bigr)^{-1}C\bigl( {\mathcal L}f\bigr)(\lambda_{1},\lambda_{2}) : \lambda_{1}\in \Sigma_{(\pi/2)+\theta_{1}'},\ \lambda_{2}\in \Sigma_{(\pi/2)+\theta_{2}'}\Bigr\}\subseteq L(X)
\end{align*}
is equicontinuous.
\end{itemize}
Then Theorem \ref{analytic} yields that there exists a unique solution $u(x_{1},x_{2})$ of problem [\eqref{prvarl}; \eqref{prvarl1}-\eqref{prvarl2}], resp. [\eqref{prvac}; \eqref{prvac1}-\eqref{prvac2}]; furthermore, the solution $u(\cdot;\cdot)$ can be holomorphically extended to the sector $\Sigma_{\theta_{1}} \times \Sigma_{\theta_{2}}$, for every numbers $\theta_{1}' \in (0,\theta_{1})$ and $\theta_{2}' \in (0,\theta_{2}),$ the set $\{u(z_{1},z_{2}) : z_{1}\in \Sigma_{\theta_{1}'},\ z_{2}\in \Sigma_{\theta_{2}'} \}$ is bounded and the assumption
\begin{align*}
\lim_{\lambda_{1}\rightarrow +\infty; \lambda_{2}\rightarrow +\infty}&\Biggl[ \lambda_{1}\lambda_{2}\bigl( \lambda_{1}^{\alpha_{1}} \lambda_{2}^{\alpha_{2}} -{\mathcal A}\bigr)^{-1}C\Bigl\{ \lambda_{1}^{\alpha_{1}}\lambda_{2}^{m_{2}-k-1}\bigl({\mathcal L}f_{k}\bigr)(\lambda_{1})\\& +\lambda_{1}^{m_{1}-1-k}\lambda_{2}^{\alpha_{2}}\bigl({\mathcal L}h_{k}\bigr)(\lambda_{2})+\bigl( {\mathcal L}f\bigr)(\lambda_{1},\lambda_{2})\Bigr\}\Biggr]=x,
\end{align*}
resp.
\begin{align*}
\lim_{\lambda_{1}\rightarrow +\infty; \lambda_{2}\rightarrow +\infty}&\Biggl[ \lambda_{1}\lambda_{2}\bigl( \lambda_{1}^{\alpha_{1}} \lambda_{2}^{\alpha_{2}} -{\mathcal A}\bigr)^{-1}C\Bigl\{ \lambda_{1}^{\alpha_{1}}\lambda_{2}^{\alpha_{2}-k-1}\bigl({\mathcal L}f_{k}\bigr)(\lambda_{1})\\& +\lambda_{1}^{\alpha_{1}-1-k}\lambda_{2}^{\alpha_{2}}\bigl({\mathcal L}h_{k}\bigr)(\lambda_{2})+\bigl( {\mathcal L}f\bigr)(\lambda_{1},\lambda_{2})\Bigr\}\Biggr]=x,
\end{align*}
for some $x\in X$ 
implies that, for every numbers $\theta_{1}' \in (0,\theta_{1})$ and $\theta_{2}' \in (0,\theta_{2}),$ we have $\lim_{(z_{1},z_{2})\rightarrow (0,0); z_{1}\in \Sigma_{\theta_{1}'}, z_{2}\in \Sigma_{\theta_{2}'}}u(z_{1},z_{2})=x$ (cf. also Proposition \ref{holo}).

We can similarly analyze applications of Theorem \ref{analytic} in the study of existence and uniqueness of holomorphic solutions to the following abstract multi-term fractional partial differential equations with Riemann-Liouville and Caputo derivatives:
\begin{align*} 
\sum_{k=1}^{n}A_{k}D_{R}^{(0,...,\alpha_{k},...,0)}u\bigl(x_{1},...,x_{k},...,x_{n}\bigr)=f\bigl(x_{1},...,x_{n}\bigr),\quad x_{1}\geq 0,\ ...,\ x_{n}\geq 0,
\end{align*}
subjected with the initial conditions 
\begin{align*}
\Biggl[ \frac{\partial^{j}}{\partial x_{k}^{j} }J_{t_{k}}^{m_{k}-\alpha_{k}}u\bigl(x_{1},...,x_{n}\bigr)\Biggr]_{x_{k}=0}=f_{k,j}\bigl(x_{1},...,x_{k-1},x_{k+1},...,x_{n}\bigr),
\end{align*}
for $1\leq k\leq n,\ 0\leq j\leq m_{k}-1,$
and
\begin{align*}
\sum_{k=1}^{n}A_{k}D_{C}^{(0,...,\alpha_{k},...,0)}u\bigl(x_{1},...,x_{k},...,x_{n}\bigr)=f\bigl(x_{1},...,x_{n}\bigr),\quad x_{1}\geq 0,\ ...,\ x_{n}\geq 0,
\end{align*}
subjected with the initial conditions 
\begin{align*}
\Biggl[ \frac{\partial^{j}}{\partial x_{k}^{j} }u\bigl(x_{1},...,x_{n}\bigr) \Biggr]_{x_{k}=0}=f_{k,j}\bigl(x_{1},...,x_{k-1},x_{k+1},...,x_{n}\bigr),
\end{align*}
for $1\leq k\leq n,\ 0\leq j\leq m_{k}-1,$
where $A_{k}$ is a closed linear operator and $\alpha_{k}\geq 0$ for $1\leq k\leq n.$ For more details, we refer the reader to \cite[Subsection 6.2]{axioms}.

\subsection{Abstract degenerate Volterra integro-differential inclusions with multiple variables}\label{sajkote}

In \cite[Chapter 3]{FKP}, we have systematically analyzed the well-posedness of the following abstract degenerate Volterra integro-differential inclusion:
$$
Bu(t)\in {\mathcal A}\int^{t}_{0}a(t-s)u(s)\, ds+Cf(t),\quad t\geq 0,
$$
where $a\in L_{loc}^{1}([0,+\infty)),$ ${\mathcal A}$ is a closed MLO in $X$, $B$ is a closed linear operator on $X$ and the operator $C\in L(X)$ is injective. In this subsection, we will briefly explain how one can use the multidimensional vector-valued Laplace transform in the study of well-posedness to the following abstract degenerate Volterra integro-differential inclusions with multiple variables:
\begin{align}\notag
Bu\bigl(t_{1},....,t_{n}\bigr)&\in {\mathcal A}\int^{t_{1}}_{0}...\int^{t_{n}}_{0}a\bigl(t_{1}-s_{1},...,t_{n}-s_{n}\bigr)u\bigl(s_{1},...,s_{n}\bigr)\, ds_{1} ...\, ds_{n}
\\& \label{sol}+Cf\bigl(t_{1},...,t_{n}\bigr),\quad t_{1}\geq 0,\ ...,\ t_{n}\geq 0,
\end{align}
where $a\in L_{loc}^{1}([0,+\infty)^{n}).$ It is clear that a fairly complete study of problem \eqref{sol} is without scope of this paper since it requires the introduction and further analysis of solution operator families depending on several variables (cf. \cite[Definition 3.2.1]{FKP} and \cite[Definition 7.1.1]{funkcionalne} for some concepts known in the one-dimensional setting).

We will use the following notion:

\begin{defn}\label{labela}
\begin{itemize}
\item[(i)] By a mild solution of \eqref{sol} we mean any continuous function $u : [0,+\infty)^{n} \rightarrow X$ such that 
$$
\int^{t_{1}}_{0}...\int^{t_{n}}_{0}a\bigl(t_{1}-s_{1},...,t_{n}-s_{n}\bigr)u\bigl(s_{1},...,s_{n}\bigr)\, ds_{1} ...\, ds_{n}\in D({\mathcal A}),\quad t_{1}\geq 0,\ ...,\ t_{n}\geq 0,
$$ 
and \eqref{sol} holds identically on $[0,+\infty)^{n}.$
\item[(ii)] By a strong solution of \eqref{sol} we mean any continuous function $u : [0,+\infty)^{n} \rightarrow X$ which satisfies that
there exists a continuous function $u_{{\mathcal A}} : [0,+\infty)^{n} \rightarrow X$ such that $(u(t_{1},....,t_{n}),u_{{\mathcal A}}(t_{1},...,t_{n}))\in {\mathcal A}$ for all $t_{1}\geq 0,\ ...,\ t_{n}\geq 0,$ and
\begin{align*} 
Bu\bigl(t_{1},....,t_{n}\bigr)&=\int^{t_{1}}_{0}...\int^{t_{n}}_{0}a\bigl(t_{1}-s_{1},...,t_{n}-s_{n}\bigr)u_{{\mathcal A}}\bigl(s_{1},...,s_{n}\bigr)\, ds_{1} ...\, ds_{n}
\\& +Cf\bigl(t_{1},...,t_{n}\bigr),\quad t_{1}\geq 0,\ ...,\ t_{n}\geq 0.
\end{align*}
\end{itemize}
\end{defn}

It is clear that a strong solution of \eqref{sol} is also a mild solution of \eqref{sol} as well as that the converse statement is not true, in general. Now we will clarify the following result concerning the well-posedness of problem \eqref{sol}:

\begin{prop}\label{proja}
Suppose that there exist real numbers $\omega_{j}\geq 0$, $\epsilon_{j}>0$ and $\epsilon_{j}'>0$ ($j\in {\mathbb N}_{n}$) such that $\{\lambda_{1} \in {\mathbb C} : \Re \lambda_{1}>\omega_{1}\} \times ... \times \{\lambda_{n} \in {\mathbb C} : \Re \lambda_{n}>\omega_{n}\}\subseteq \Omega (f) \cap \Omega_{b}(f) \cap \Omega_{abs}(a)$, the function $f : [0,+\infty)^{n} \rightarrow X$ is continuous, the mappings{\small
\begin{align}\label{res12}\bigl( \lambda_{1},...,\lambda_{n}\bigr) \mapsto \Bigl( B-\bigl({\mathcal L}a\bigr)\bigl( \lambda_{1},...,\lambda_{n}\bigr){\mathcal A} \Bigr)^{-1}C\cdot \bigl({\mathcal L}f\bigr)\bigl( \lambda_{1},...,\lambda_{n}\bigr),\ \Re \lambda_{j}>\omega_{j}\ \ (1\leq j\leq n) \end{align}}
and{\small 
\begin{align}\label{res123}\bigl( \lambda_{1},...,\lambda_{n}\bigr) \mapsto B\Bigl( B-\bigl({\mathcal L}a\bigr)\bigl( \lambda_{1},...,\lambda_{n}\bigr){\mathcal A} \Bigr)^{-1}C\cdot \bigl({\mathcal L}f\bigr)\bigl( \lambda_{1},...,\lambda_{n}\bigr),\ \Re \lambda_{j}>\omega_{j}\ \ (1\leq j\leq n) \end{align}}
are holomorphic and for each seminorm $p\in \circledast$ there exists a finite real constant $M_{p}>0$ such that
\begin{align} \notag
p\Biggl( & \Bigl( B-\bigl({\mathcal L}a\bigr)\bigl( \lambda_{1},...,\lambda_{n}\bigr){\mathcal A} \Bigr)^{-1}C\cdot \bigl({\mathcal L}f\bigr)\bigl( \lambda_{1},...,\lambda_{n}\bigr)\Biggr) 
\\\label{p1}& \leq M_{p}|\lambda_{1}|^{-1-\epsilon_{1}}\cdot ...\cdot |\lambda_{n}|^{-1-\epsilon_{n}},\quad \Re \lambda_{j}>\omega_{j}\ \ (1\leq j\leq n) 
\end{align}
and
\begin{align} \notag
p\Biggl( & B\Bigl( B-\bigl({\mathcal L}a\bigr)\bigl( \lambda_{1},...,\lambda_{n}\bigr){\mathcal A} \Bigr)^{-1}C\cdot \bigl({\mathcal L}f\bigr)\bigl( \lambda_{1},...,\lambda_{n}\bigr)\Biggr) 
\\\label{p2}& \leq M_{p}|\lambda_{1}|^{-1-\epsilon_{1}'}\cdot ...\cdot |\lambda_{n}|^{-1-\epsilon_{n}'},\quad \Re \lambda_{j}>\omega_{j}\ \ (1\leq j\leq n) .
\end{align} 
Then there exists a mild solution $u(\cdot)$ of problem \eqref{sol} satisfying that for each seminorm $p\in \circledast$ there exists a finite real constant $M_{p}'>0$ such that
\begin{align}\label{proc1}
p\bigl( u(t_{1},...,t_{n}) \bigr) \leq M_{p}'\Bigl[t_{1}^{\epsilon_{1}}e^{\omega_{1}t_{1}}\cdot ...\cdot t_{n}^{\epsilon_{n}}e^{\omega_{n}t_{n}}\Bigr],\quad t_{1}\geq 0,\ ...,\ t_{n}\geq 0 
\end{align}
and
\begin{align}\label{proc2}
p\bigl( Bu(t_{1},...,t_{n})\bigr) \leq M_{p}'\Bigl[t_{1}^{\epsilon_{1}'}e^{\omega_{1}t_{1}}\cdot ...\cdot t_{n}^{\epsilon_{n}'}e^{\omega_{n}t_{n}}\Bigr],\quad t_{1}\geq 0,\ ...,\ t_{n}\geq 0;
\end{align}
furthermore, if there exist analytic function $F: \{\lambda_{1} \in {\mathbb C} : \Re \lambda_{1}>\omega_{1}\} \times ... \times \{\lambda_{n} \in {\mathbb C} : \Re \lambda_{n}>\omega_{n}\} \rightarrow X$ and numbers $\epsilon_{j}''>0$ ($j\in {\mathbb N}_{n}$) such that{\small
$$
F( \lambda_{1},...,\lambda_{n}) \in {\mathcal A}\Bigl( B-\bigl({\mathcal L}a\bigr)\bigl( \lambda_{1},...,\lambda_{n}\bigr){\mathcal A} \Bigr)^{-1}C\cdot \bigl({\mathcal L}f\bigr)\bigl( \lambda_{1},...,\lambda_{n}\bigr),\quad \Re \lambda_{j}>\omega_{j}\ \ (1\leq j\leq n) 
$$}
and for each seminorm $p\in \circledast$ there exists a finite real constant $M_{p}''>0$ such that
\begin{align*} 
p\Bigl(  F\bigl( \lambda_{1},...,\lambda_{n}\bigr)\Bigr) 
\leq M_{p}''|\lambda_{1}|^{-1-\epsilon_{1}''}\cdot ...\cdot |\lambda_{n}|^{-1-\epsilon_{n}''},\quad \Re \lambda_{j}>\omega_{j}\ \ (1\leq j\leq n) ,
\end{align*} 
then $u(\cdot)$ is a strong solution of \eqref{sol}. 
\end{prop}

\begin{proof}
Since we have assumed that \eqref{p1} holds,
we can apply Theorem \ref{cit} in order to see that the function 
\begin{align}\label{sc}
u\bigl(t_{1},...,t_{n}\bigr):=\Biggl({\mathcal L}^{-1}\Biggl[ \Bigl( B-({\mathcal L}a)( \lambda_{1},...,\lambda_{n}){\mathcal A} \Bigr)^{-1}C\cdot ({\mathcal L}f)\bigl( \lambda_{1},...,\lambda_{n})\Biggr]\Biggr)\bigl(t_{1},...,t_{n}\bigr),
\end{align}
is continuous for $t_{1}\geq 0,\ ...,\ t_{n}\geq 0$ and the estimate \eqref{proc1} holds true; here, of course, ${\mathcal L}^{-1}$ denotes the inverse Laplace transform. Keeping in mind \eqref{p2}, we similarly obtain that the function $(t_{1},...,t_{n}) \mapsto
Bu(t_{1},...,t_{n}), $ $t_{1}\geq 0,\ ...,\ t_{n}\geq 0$ is well-defined, continuous and the estimate \eqref{proc2} holds true; furthermore, 
Lemma \ref{bnm} yields that $({\mathcal L}(Bu))(\lambda_{1},...,\lambda_{n})=B({\mathcal L}u)(\lambda_{1},...,\lambda_{n})$ for $\Re \lambda_{j}>\omega_{j}\ \ (1\leq j\leq n).$ Since $\{\lambda_{1} \in {\mathbb C} : \Re \lambda_{1}>\omega_{1}\} \times ... \times \{\lambda_{n} \in {\mathbb C} : \Re \lambda_{n}>\omega_{n}\}\subseteq \Omega_{abs}(u) \cap \Omega_{abs}(a,)$ we can apply Proposition \ref{lapkonv} to show that $B({\mathcal L}u)(\lambda_{1},...,\lambda_{n})-C({\mathcal L}f)(\lambda_{1},...,\lambda_{n})\in {\mathcal A}[({\mathcal L}a)(\lambda_{1},...,\lambda_{n})\cdot ({\mathcal L}u)(\lambda_{1},...,\lambda_{n})]$ for $\Re \lambda_{j}>\omega_{j}\ \ (1\leq j\leq n).$ Applying Proposition \ref{piano-MLO}, we get that $u(\cdot)$ is a mild solution of problem \eqref{sol}. In the case that there exist analytic function $F: \{\lambda_{1} \in {\mathbb C} : \Re \lambda_{1}>\omega_{1}\} \times ... \times \{\lambda_{n} \in {\mathbb C} : \Re \lambda_{n}>\omega_{n}\} \rightarrow X$ and numbers $\epsilon_{j}''>0$ ($j\in {\mathbb N}_{n}$) with prescribed properties, then there exists a continuous function $u_{{\mathcal A}} : [0,+\infty)^{n}$ such that $({\mathcal L}u_{{\mathcal A}})(\lambda_{1},...,\lambda_{n}) \in {\mathcal A}({\mathcal L}u)(\lambda_{1},...,\lambda_{n})$ for $\Re \lambda_{j}>\omega_{j}\ \ (1\leq j\leq n).$ Applying Proposition \ref{piano-MLO} once more, we get that
$u_{{\mathcal A}}(t_{1},...,t_{n})\in {\mathcal A}u(t_{1},...,t_{n})$ for $t_{1}\geq 0,\ ...,\ t_{n}\geq 0$, so that $u(\cdot)$ is a strong solution to \eqref{sol}.
\end{proof}

\begin{rem}\label{tomici}
Due to Proposition \ref{analyticala}, the mappings $(\lambda_{1},...,\lambda_{n}) \rightarrow ({\mathcal La})(\lambda_{1},...,\lambda_{n}) $ and $(\lambda_{1},...,\lambda_{n}) \rightarrow ({\mathcal Lu})(\lambda_{1},...,\lambda_{n}) $ are holomorphic for $\Re \lambda_{j}>\omega_{j}\ \ (1\leq j\leq n) $. 
Then we can apply the Hartogs theorem to show that the mappings in \eqref{res12} and \eqref{res123} are holomorphic if they are separately holomorphic. This occurs, for example, if ${\mathcal A}=A$ is a closed linear operator, $X$ is a Fr\' echet space and the mapping $\lambda \mapsto A(\lambda A-B)^{-1}Cx,$ $\lambda \in \Omega$ is locally bounded for every fixed element $x\in X$; here $\emptyset \neq \Omega \subseteq \rho_{C}^{A}(B):=\{\lambda\in {\mathbb C} : (\lambda A-B)^{-1}C\in L(X)\}$ (cf. \cite[Proposition 2.1.3(i)]{FKP}). In this case, the mapping{\small
\begin{align*}\bigl( \lambda_{1},...,\lambda_{n}\bigr) \mapsto A\Bigl( B-\bigl({\mathcal L}a\bigr)\bigl( \lambda_{1},...,\lambda_{n}\bigr)A \Bigr)^{-1}C\cdot \bigl({\mathcal L}f\bigr)\bigl( \lambda_{1},...,\lambda_{n}\bigr),\ \Re \lambda_{j}>\omega_{j}\ \ (1\leq j\leq n) \end{align*}} 
is holomorphic, as well.
\end{rem}

It is clear that the function $u(\cdot),$ given by \eqref{sc}, is a unique mild solution of problem \eqref{sol} which has the property that $\{\lambda_{1} \in {\mathbb C} : \Re \lambda_{1}>\omega_{1}\} \times ... \times \{\lambda_{n} \in {\mathbb C} : \Re \lambda_{n}>\omega_{n}\}\subseteq \Omega_{abs}(u) \cap \Omega_{abs}(Bu).$
We can also apply Theorem \ref{analytic} in place of Theorem \ref{cit} here; in this case, the constructed solution will converge to an element $x\in B^{-1}Cf(0,...,0)$ as $(t_{1},...,t_{n})\rightarrow (0+,....,0+)$ if and only if
$$
\lim_{\lambda_{1}\rightarrow +\infty;....;\lambda_{n}\rightarrow +\infty}\Biggl[ \lambda_{1}\cdot\lambda_{2}\cdot ... \cdot \lambda_{n}\cdot\Bigl( B-\bigl({\mathcal L}a\bigr)\bigl( \lambda_{1},...,\lambda_{n}\bigr){\mathcal A} \Bigr)^{-1}C\cdot \bigl({\mathcal L}f\bigr)\bigl( \lambda_{1},...,\lambda_{n}\bigr)\Biggr]=x.
$$

\subsection{Abstract higher-order differential equations with multiple variables}\label{skockao}

In this subsection, we investigate the abstract higher-order differential equations with multiple variables. Of concern is the following problem
\begin{align}\label{jena}
\sum_{\alpha \in I}A_{\alpha}u^{(\alpha)}\bigl( t_{1},...,t_{n}\bigr)=f\bigl( t_{1},...,t_{n}\bigr),\quad \bigl( t_{1},...,t_{n}\bigr)\in [0,+\infty)^{n},
\end{align}
where $I$ is a finite non-empty subset of ${\mathbb N}_{0}^{n}$, $A_{\alpha}$ is a closed linear operator on $X$ for all $\alpha \in I$ and the function $f : [0,+\infty)^{n} \rightarrow X$ is continuous. By a solution of this problem, we mean any function $u: [0,+\infty)^{n} \rightarrow X$ such that the mapping $u^{(\alpha)}: [0,\infty)^{n} \rightarrow X $ is continuous and satisfies that $u^{(\alpha)} ( t_{1},...,t_{n} )\in D(A_{\alpha})$ for all $\alpha \in I$, $ ( t_{1},...,t_{n} )\in [0,+\infty)^{n} $ and \eqref{jena} identically holds for all $ ( t_{1},...,t_{n} )\in [0,+\infty)^{n} .$

We will first examine the question what initial conditions should be addressed to \eqref{jena} in order that the corresponding boundary value problem has at most one solution. To solve this problem, we will use the multidimensional vector-valued Laplace transform. First of all, we would like to recall the assertion of \cite[Lemma 7.3.1]{funkcionalne}, which is also valid in general SCLCSs (cf. \cite[Theorem 1.1.4(iii)]{FKP}):

\begin{lem}\label{profop}
Suppose that $m\in {\mathbb N},$ $f\in C^{m-1}([0,\infty): X)$, the function $f^{(m-1)}(\cdot)$ is locally absolutely continuous on $[0,\infty)$ and the functions $f(\cdot),...,f^{(m-1)}(\cdot)$ are exponentially bounded, with the meaning clear. Then the function $f^{(m)}(\cdot)$ is Laplace transformable and we have
$$
\int^{+\infty}_{0}e^{-\lambda t}f^{(m)}(t)\, dt=\lambda^{m}({\mathcal L }f)(\lambda)-\lambda^{m-1}f(0)-\lambda^{m-2}f^{\prime}(0)-...-f^{(m-1)}(0),\quad 
$$
for $\Re \lambda>0\mbox{ suff. large.}$
\end{lem}

In the sequel, we will use the standard multi-index notation; especially, if $\lambda=(\lambda_{1},...,\lambda_{n}) \in {\mathbb C}^{n}$ and $\alpha=(\alpha_{1},...,\alpha_{n}) \in {\mathbb N}_{0}^{n}$, then we set $\lambda^{\alpha}:=\lambda_{1}^{\alpha_{1}}\cdot ... \cdot \lambda_{n}^{\alpha_{n}}$ and $|\alpha|:=\alpha_{1}+...+\alpha_{n}.$ 
Now we will compute the value of $({\mathcal L}u^{(\alpha)})(\lambda)$ with the help of Lemma \ref{profop}. Assume that the mapping $u^{(\alpha)}(\cdot)$ is exponentially bounded in the sense that there exists real numbers $\omega_{j}\geq 0$ ($1\leq j\leq n$) such that for each seminorm $p\in \circledast$ there exists $M_{p}>0$ such that $p(u^{(\alpha)}(t_{1},...,t_{n}))\leq M_{p}\exp(\omega_{1}t_{1}+...+\omega_{n}t_{n})$ for all $(t_{1},...,t_{n})\in [0,+\infty)^{n}.$ 

In the first step, we recognize two possibilities: $\alpha_{n}=0$ and $\alpha_{n}>0.$ If $\alpha_{n}=0,$ then we do not do anything; if $\alpha_{n}>0,$ then we also assume
that for a.e. $(t_{1},....,t_{n-1})\in [0,+\infty)^{n-1}$ the mapping $u^{(\alpha_{1},....,\alpha_{n-1},\alpha_{n}-1)}(t_{1},....,t_{n-1},\cdot)$ is locally absolutely continuous as well as that the mappings $u^{(\alpha_{1},....,\alpha_{n-1},j_{n})}(t_{1},....,t_{n-1},\cdot)$ are exponentially bounded for $0\leq j_{n}\leq \alpha_{n}-1.$ Applying the Fubini theorem and Lemma \ref{profop}, we get that{\small
\begin{align*} &
\Bigl({\mathcal L}u^{(\alpha)}\Bigr)(\lambda)
\\&=\int^{+\infty}_{0}... \int^{+\infty}_{0}e^{-\lambda_{1}t_{1}-....\lambda_{n-1}t_{n-1}}\Biggl[\int^{+\infty}_{0}e^{-\lambda_{n}t_{n}}u^{(\alpha_{1},....,\alpha_{n})}\bigl( t_{1},...,t_{n}\bigr) \Biggr]\, dt_{1}...\, dt_{n-1}
\\& =\lambda_{n}^{\alpha_{n}}\int^{+\infty}_{0}... \int^{+\infty}_{0}e^{-\lambda_{1}t_{1}-....\lambda_{n-1}t_{n-1}-\lambda_{n}t_{n}}u^{(\alpha_{1},....,\alpha_{n-1},0)}\bigl( t_{1},...,t_{n}\bigr) \, dt_{1}...\, dt_{n}
\\& -\sum_{j_{n}=0}^{\alpha_{n}-1}\lambda_{n}^{\alpha_{n}-1-j}\int^{+\infty}_{0}... \int^{+\infty}_{0}e^{-\lambda_{1}t_{1}-....\lambda_{n-1}t_{n-1}}u^{(\alpha_{1},....,\alpha_{n-1},j_{n})}\bigl( t_{1},...,t_{n-1},0\bigr)\, dt_{1}...\, dt_{n-1},
\end{align*}}
for $\Re \lambda_{j}>\omega_{j}$ suff. large ($1\leq j\leq n$). For every $\alpha \in I,$ we endow the problem \eqref{jena} with the initial conditions
$$
u_{j_{n}}\bigl( t_{1},...,t_{n-1},0\bigr) := u^{(\alpha_{1},....,\alpha_{n-1},j_{n})}\bigl( t_{1},...,t_{n-1},0\bigr),\quad t_{1}\geq 0,\ ...,\ t_{n-1}\geq 0,
$$
for any $0\leq j_{n}\leq \alpha_{n}-1,$ so that we have exactly $\alpha_{n}$ initial values associated to the problem \eqref{jena} after the first step; here, we assume that all these initial values are exponentially bounded with respect to the variables $t_{1},...,\ t_{n-1}$ as well as that the function $u^{(\alpha_{1},....,\alpha_{n-1},0)}( t_{1},...,t_{n}) $ is exponentially bounded. In such a way, we get
\begin{align*} &
\Bigl({\mathcal L}u^{(\alpha)}\Bigr)(\lambda)
\\&=\lambda_{n}^{\alpha_{n}}\int^{+\infty}_{0}... \int^{+\infty}_{0}e^{-\lambda_{1}t_{1}-....\lambda_{n-1}t_{n-1}-\lambda_{n}t_{n}}u^{(\alpha_{1},....,\alpha_{n-1},0)}\bigl( t_{1},...,t_{n}\bigr) \, dt_{1}...\, dt_{n}
\\&+G_{1}\bigl( \lambda_{1},...,\lambda_{n}\bigr)
\end{align*}
for $\Re \lambda_{j}>\omega_{j}$ suff. large ($1\leq j\leq n$), where the form of function $ G_{1}( \lambda_{1},...,\lambda_{n})$
can be simply find.

In the second step, we recognize two possibilities: $\alpha_{n-1}=0$ and $\alpha_{n-1}>0.$ If $\alpha_{n-1}=0,$ then we do not do anything; if $\alpha_{n-1}>0,$ then we also assume
that for a.e. $(t_{1},....,t_{n-2},t_{n})\in [0,+\infty)^{n-1}$ the mapping $u^{(\alpha_{1},....,\alpha_{n-2},\alpha_{n-1}-1,0)}(t_{1},....,t_{n-2},\cdot,t_{n})$ is locally absolutely continuous as well as that the mappings\\ $u^{(\alpha_{1},....,\alpha_{n-1},j_{n-1},j_{n})}(t_{1},....,t_{n-2},\cdot,t_{n})$ are exponentially bounded for $0\leq j_{n-1}\leq \alpha_{n-1}-1$ and $0\leq j_{n}\leq \alpha_{n}-1.$ For every $\alpha \in I,$ we endow the problem \eqref{jena} with the initial conditions
$$
u_{j_{n-1}}\bigl( t_{1},...,t_{n-2},0,t_{n}\bigr) := u^{(\alpha_{1},....,\alpha_{n-2},j_{n-1},0)}\bigl( t_{1},...,t_{n-2},0,t_{n}\bigr), 
$$
for any $ t_{1}\geq 0,\ ...,\ t_{n-2}\geq 0,\ t_{n}\geq 0 $ and $0\leq j_{n-1}\leq \alpha_{n-1}-1,$ so that we have exactly $\alpha_{n-1}+\alpha_{n}$ initial values associated to the problem \eqref{jena} after the first two steps; here, we assume that all  initial values added in the second step are exponentially bounded with respect to the variables $t_{1},...,\ t_{n-2},\ t_{n} $ and that the mapping $u^{(\alpha_{1},....,\alpha_{n-2},0,0)} ( t_{1},...,t_{n} )$ is exponentially bounded. Applying the Fubini theorem and Lemma \ref{profop}, we get that
\begin{align*} &
\Bigl({\mathcal L}u^{(\alpha)}\Bigr)(\lambda)
\\&=\lambda_{n}^{\alpha_{n}}\lambda_{n-1}^{\alpha_{n-1}}\int^{+\infty}_{0}... \int^{+\infty}_{0}e^{-\lambda_{1}t_{1}-....\lambda_{n-1}t_{n-1}-\lambda_{n}t_{n}}u^{(\alpha_{1},....,\alpha_{n-2},0,0)}\bigl( t_{1},...,t_{n}\bigr) \, dt_{1}...\, dt_{n}
\\&+G_{2}\bigl( \lambda_{1},...,\lambda_{n}\bigr)
\end{align*}
for $\Re \lambda_{j}>\omega_{j}$ suff. large ($1\leq j\leq n$), where the form of function $ G_{2}( \lambda_{1},...,\lambda_{n})$
can be simply find.

...

In the $n$-th step, we recognize two possibilities: $\alpha_{1}=0$ and $\alpha_{1}>0.$ If $\alpha_{1}=0,$ then we do not do anything; if $\alpha_{ 1}>0,$ then we also assume
that for a.e. $(t_{2},...., t_{n})\in [0,+\infty)^{n-1}$ the mapping $u^{(\alpha_{1}-1,0,....,0)}(\cdot,t_{2},...,t_{n})$ is locally absolutely continuous as well as that the mappings $u^{(j_{1},j_{2},...,j_{n})}(\cdot,t_{2},....,t_{n})$ are exponentially bounded for $0\leq j_{n-1}\leq \alpha_{n-1}-1,$ $0\leq j_{n}\leq \alpha_{n}-1$ ..., $0\leq j_{2}\leq \alpha_{2}-1$ and $0\leq j_{1}\leq \alpha_{1}-1$. For every $\alpha \in I,$ we endow the problem \eqref{jena} with the initial conditions 
$$
u_{j_{1}}\bigl( 0,t_{2},...,t_{n}\bigr) := u^{(j_{1},0,...,0)}\bigl( 0,t_{2},...,t_{n}\bigr), 
$$
for any $ t_{2}\geq 0,\ ...,\  t_{n}\geq 0 $ and $0\leq j_{1}\leq \alpha_{1}-1,$ so that we have exactly $|\alpha|=\alpha_{1}+...+\alpha_{n}$ initial values associated to the problem \eqref{jena} after the whole procedure; in the last step, we assume that all added initial values are exponentially bounded with respect to the variables $t_{2},...,\ t_{n }$ and that the mapping $u ( t_{1},...,t_{n} )$ is exponentially bounded.
Applying the Fubini theorem and Lemma \ref{profop}, we get that
\begin{align*} &
\Bigl({\mathcal L}u^{(\alpha)}\Bigr)(\lambda)
\\&=\lambda^{\alpha}\int^{+\infty}_{0}... \int^{+\infty}_{0}e^{-\lambda_{1}t_{1}-....\lambda_{n-1}t_{n-1}-\lambda_{n}t_{n}}u\bigl( t_{1},...,t_{n}\bigr) \, dt_{1}...\, dt_{n}
\\&+G_{n}\bigl( \lambda_{1},...,\lambda_{n}\bigr)
\end{align*}
for $\Re \lambda_{j}>\omega_{j}$ suff. large ($1\leq j\leq n$), where the form of function $ G_{n}( \lambda_{1},...,\lambda_{n})$
can be simply find. Finally, if the above assumptions are satisfied for every $\alpha \in I$ and the operator
$$
\sum_{\alpha \in I}\lambda^{\alpha}A_{\alpha}
$$
is injective on the direct product of some right half planes, then the problem \eqref{sol} endowed with the above $|\alpha|$ initial values for each $\alpha \in I$ can have at most one solution. Sometimes for different values of $\alpha_{1}\in I$ and $\alpha_{2}\in I$ we can associate the same initial value (for example, this happens if $(\alpha_{1},\alpha_{2})=(3,2)$ and $(\alpha_{1},\alpha_{2})=(3,3)$), so that the final number of initial values is less or equal than $\sum_{\alpha \in I}|\alpha|$.
Observe also that the problem of asssociation initial values to \eqref{sol} is not so simple because we can permute the variables for the applications of Fubini theorem and obtain, in such a way, different sets of initial conditions for which problem \eqref{sol} has at most one Laplace transformable solution:

\begin{example}\label{ill}
Let us consider the problem
$$
\frac{\partial^{5}u}{\partial t_{1}^{3}\partial t_{2}^{2}}=f\bigl(t_{1},t_{2},t_{3}\bigr),\quad \bigl(t_{1},t_{2},t_{3}\bigr) \in [0,+\infty)^{3}. 
$$
The the above procedure gives the following set of initial conditions: $u^{(3,0,0)}(t_{1},0,t_{3})$, $u^{(3,1,0)}(t_{1},0,t_{3}),$ $u(0,t_{2},t_{3}),$ $u^{(1,0,0)}(0,t_{2},t_{3})$ and $u^{(2,0,0)}(0,t_{2},t_{3}).$ On the other hand, we can write
$$
\Bigl({\mathcal L}u^{(3,2,0)}\Bigr)(\lambda)=\int^{+\infty}_{0}\int^{+\infty}_{0}e^{-\lambda_{2}t_{2}-\lambda_{3}t_{3}}\Biggl[ e^{-\lambda_{1}t_{1}}u^{(3,2,0)}\bigl(t_{1},t_{2},t_{3}\bigr)\, dt_{1}\Biggr]\, dt_{2}\, dt_{3};
$$
in such a way, the above procedure gives the following set of initial conditions: $u^{(2,2,0)}(0,t_{2},t_{3}),$ $u^{(1,2,0)}(0,t_{2},t_{3}),$ $u^{(0,2,0)}(0,t_{2},t_{3}),$ $u (t_{1},0,t_{3})$ and $u^{(0,1,0)} (t_{1},0,t_{3}).$
\end{example}

Let us finally consider the following abstract complete second-order Cauchy problem, with $t_{1}=x$ and $t_{2}=y$:
\begin{align}\notag
Au_{xx}(x,y)& +Bu_{xy}(x,y)+Cu_{yy}(x,y)
\\\label{so}& +Du_{x}(x,y)+Eu_{y}(x,y)+Fu(x,y)=f(x,y),\quad x\geq 0,\ y\geq0,
\end{align}
where $A,\ B,\ C,\ D,\ E$ and $F$ are closed linear operators on $X.$ In accordance with our previous analysis, we endow the problem \eqref{so} with the following initial conditions:
\begin{align}\label{ic}
u(x,0)=f_{1}(x),\   u_{x}(x,0)=f_{2}(x),\ u_{y}(x,0)=f_{3}(x),\ u(0,y)=g_{1}(y),\  u_{x}(0,y)=g_{2}(y);
\end{align}
if we compute the Laplace transform of term $u_{xy}(x,y)$ by
$$
\Bigl({\mathcal L}u_{xy}(x,y)\Bigr)\bigl( \lambda_{1},\lambda_{2} \bigr)=\int^{+\infty}_{0}e^{-\lambda_{2}y}\Biggl[ \int^{+\infty}_{0}e^{-\lambda_{1}x}u_{xy}(x,y)\, dx \Biggr]\, dy 
$$
in place of
$$
\int^{+\infty}_{0}e^{-\lambda_{1}x}\Biggl[ \int^{+\infty}_{0}e^{-\lambda_{2}y}u_{xy}(x,y)\, dy \Biggr]\, dx ,
$$
then the equivalent set of initial conditions will be:
\begin{align*} 
u(x,0)=f_{1}(x),\  u_{y}(x,0)=f_{2}(x),\ u(0,y)=g_{1}(y),\  u_{y}(0,y)=g_{2}(y),\, u_{x}(0,y)=g_{3}(y).
\end{align*}
In the sequel, we will use the initial conditions \eqref{ic}. 

We have:
\begin{align*} 
\Bigl({\mathcal L}u_{xx}(x,y)\Bigr)\bigl( \lambda_{1},\lambda_{2} \bigr)= \lambda_{1}^{2}\Bigl({\mathcal L}u(x,y)\Bigr)\bigl( \lambda_{1},\lambda_{2} \bigr)- \int^{+\infty}_{0}e^{-\lambda_{2}y}\bigl[ \lambda_{1}g_{1}(y)+g_{2}(y)\bigr] \, dy,
\end{align*}
provided that the following conditions hold:
\begin{itemize}
\item[(xx1)] The functions $u(\cdot, \cdot)$ and $u_{xx}(\cdot, \cdot)$ are exponentially bounded; 
\item[(xx2)] For a.e. $y\geq 0$ the functions $u(\cdot,y)$ and $u_{x}(\cdot,y)$ are exponentially bounded and the function $u_{x}(\cdot,y)$ is locally absolutely continuous;
\item[(xx3)] The functions $g_{1}( \cdot)$ and $g_{2}(\cdot)$ are exponentially bounded;
\end{itemize}
\begin{align*} 
\Bigl({\mathcal L}u_{xy}(x,y)\Bigr)\bigl( \lambda_{1},\lambda_{2} \bigr)&= \lambda_{1}\lambda_{2} \Bigl({\mathcal L}u(x,y)\Bigr)\bigl( \lambda_{1},\lambda_{2} \bigr) 
\\& -\lambda_{2}\int^{+\infty}_{0}e^{-\lambda_{2}y}g_{1}(y)\, dy-\int^{+\infty}_{0}e^{-\lambda_{1}x}f_{2}(x)\, dx,
\end{align*}
provided that the following conditions hold:
\begin{itemize}
\item[(xy1)] The functions $u(\cdot, \cdot),$ $u_{x}(\cdot, \cdot)$ and $u_{xy}(\cdot, \cdot)$ are exponentially bounded; 
\item[(xy2)] For a.e. $x\geq 0$ the function $u_{x}(x,\cdot)$ is locally absolutely continuous and for a.e. $y\geq 0$ the function $u(\cdot ,y)$ is locally absolutely continuous;
\item[(xy3)] The function $f_{2}( \cdot)$ is exponentially bounded;
\end{itemize}
\begin{align*} 
\Bigl({\mathcal L}u_{yy}(x,y)\Bigr)\bigl( \lambda_{1},\lambda_{2} \bigr)=\lambda_{2}^{2}\Bigl({\mathcal L}u(x,y)\Bigr)\bigl( \lambda_{1},\lambda_{2} \bigr)- \int^{+\infty}_{0}e^{-\lambda_{1}x}\bigl[ \lambda_{2}f_{1}(x)+f_{3}(x)\bigr] \, dx,
\end{align*}
provided that the following conditions hold:
\begin{itemize}
\item[(yy1)] The functions $u(\cdot, \cdot)$ and $u_{yy}(\cdot, \cdot)$ are exponentially bounded; 
\item[(yy2)] For a.e. $x\geq 0$ the functions $u(x,\cdot)$ and $u_{y}(x,\cdot)$ are exponentially bounded and the function $u_{y}(x,\cdot)$ is locally absolutely continuous;
\item[(yy3)] The functions $f_{1}( \cdot)$ and $f_{3}(\cdot)$ are exponentially bounded;
\end{itemize}
\begin{align*} 
\Bigl({\mathcal L}u_{x}(x,y)\Bigr)\bigl( \lambda_{1},\lambda_{2} \bigr)=  \lambda_{1}\Bigl({\mathcal L}u(x,y)\Bigr)\bigl( \lambda_{1},\lambda_{2} \bigr)- \int^{+\infty}_{0}e^{-\lambda_{2}y} g_{1}(y)  \, dy,
\end{align*}
provided that the following conditions hold:
\begin{itemize}
\item[(x1)] The functions $u(\cdot, \cdot)$ and $u_{x}(\cdot, \cdot)$ are exponentially bounded; 
\item[(x2)] For a.e. $y\geq 0$ the function $u(\cdot,y)$ is locally absolutely continuous;
\item[(x3)] The function $g_{1}( \cdot)$ is exponentially bounded;
\end{itemize}
and
\begin{align*} 
\Bigl({\mathcal L}u_{y}(x,y)\Bigr)\bigl( \lambda_{1},\lambda_{2} \bigr)= \lambda_{2}\Bigl({\mathcal L}u(x,y)\Bigr)\bigl( \lambda_{1},\lambda_{2} \bigr)- \int^{+\infty}_{0}e^{-\lambda_{1}x} f_{1}(x)  \, dx,
\end{align*}
provided that the following conditions hold:
\begin{itemize}
\item[(y1)] The functions $u(\cdot, \cdot)$ and $u_{y}(\cdot, \cdot)$ are exponentially bounded; 
\item[(y2)] For a.e. $x\geq 0$ the function $u(x,\cdot )$ is locally absolutely continuous;
\item[(y3)] The function $f_{1}( \cdot)$ is exponentially bounded.
\end{itemize}
Inserting these results in \eqref{so} and keeping in mind the initial conditions \eqref{ic}, it follows that we should have
\begin{align}\notag &
\Bigl({\mathcal L}u(x,y) \Bigr)\bigl( \lambda_{1},\lambda_{2} \bigr)= \Biggl( \lambda_{1}^{2}A+\lambda_{1}\lambda_{2}B+\lambda_{2}^{2}C+ \lambda_{1}D+\lambda_{2}E+F\Biggr)^{-1}\times
\\\notag & \times \Biggl\{ \bigl({\mathcal L}f \bigr)\bigl( \lambda_{1},\lambda_{2} \bigr)+A\Bigl[ \lambda_{1}\bigl({\mathcal L} g_{1}\bigr)\bigl( \lambda_{2} \bigr)+\bigl({\mathcal L} g_{2}\bigr)\bigl( \lambda_{2} \bigr)\Bigr]+B\Bigl[\lambda_{2}\bigl({\mathcal L} g_{1}\bigr)\bigl( \lambda_{2} \bigr)+\bigl({\mathcal L}f_{2} \bigr)\bigl( \lambda_{1} \bigr) \Bigr]
\\\label{smiljana}& +C\Bigl[ \lambda_{2}\bigl({\mathcal L} f_{1}\bigr)\bigl( \lambda_{1} \bigr)+\bigl({\mathcal L} f_{3}\bigr)\bigl( \lambda_{1}\bigr) \Bigr]+D\bigl({\mathcal L}g_{1} \bigr)\bigl( \lambda_{2} \bigr)+E\bigl({\mathcal L}f_{1} \bigr)\bigl( \lambda_{1}\bigr)\Biggr\}.
\end{align}
 
Now we will formalize the things and prove the following result (cf. also \cite[Chapter 4]{thesis0}, \cite[Section 3]{baba}, \cite[Section 3]{dahija} and \cite[Section V(B)]{thesis1} for some more concrete applications of double Laplace transform); first of all, let us denote the function on the right hand side of \eqref{smiljana} by $G(\lambda_{1},\lambda_{2})$:

\begin{thm}\label{nm}
Suppose that there exist real numbers $\omega_{1}\geq 0,$ $\omega_{2}\geq 0$, $\epsilon_{1}>0$ and $\epsilon_{2}>0$ such that the following conditions hold:
\begin{itemize}
\item[(i)] For each seminorm $p\in \circledast$ there exists a real number $M_{p}>0$ such that 
\begin{align}\label{ok}
p\Bigl( G\bigl(\lambda_{1},\lambda_{2}\bigr)\Bigr) \leq M_{p}  \bigl| \lambda_{1}\bigr|^{-1-\epsilon_{1}} \cdot \bigl| \lambda_{2}\bigr|^{-1-\epsilon_{2}} , \quad \Re \lambda_{1} >\omega_{1},\ \Re \lambda_{2}>\omega_{2},
\end{align}
\begin{align}\notag 
p\Bigl( \lambda_{1}^{2} G& \bigl(\lambda_{1},\lambda_{2}\bigr)-\lambda_{1}\bigl({\mathcal L} g_{1}\bigr)\bigl( \lambda_{2} \bigr)-\bigl({\mathcal L} g_{2}\bigr)\bigl( \lambda_{2} \bigr)\Bigr) 
\\\label{ok1}& \leq M_{p}  \bigl| \lambda_{1}\bigr|^{-1-\epsilon_{1}} \cdot \bigl| \lambda_{2}\bigr|^{-1-\epsilon_{2}} , \quad \Re \lambda_{1} >\omega_{1},\ \Re \lambda_{2}>\omega_{2},
\end{align}
\begin{align}\notag 
p\Biggl( A\Bigl[\lambda_{1}^{2} G&\bigl(\lambda_{1},\lambda_{2}\bigr)-\lambda_{1}\bigl({\mathcal L} g_{1}\bigr)\bigl( \lambda_{2} \bigr)-\bigl({\mathcal L} g_{2}\bigr)\bigl( \lambda_{2} \bigr)\Bigr]\Biggr) 
\\\label{ok11}& \leq M_{p}  \bigl| \lambda_{1}\bigr|^{-1-\epsilon_{1}} \cdot \bigl| \lambda_{2}\bigr|^{-1-\epsilon_{2}} , \quad \Re \lambda_{1} >\omega_{1},\ \Re \lambda_{2}>\omega_{2},
\end{align}
\begin{align}\notag 
p\Bigl( \lambda_{1}\lambda_{2} G& \bigl(\lambda_{1},\lambda_{2}\bigr)-\lambda_{2}\bigl({\mathcal L} g_{1}\bigr)\bigl( \lambda_{2} \bigr)-\bigl({\mathcal L} f_{2}\bigr)\bigl( \lambda_{1} \bigr)\Bigr) 
\\\label{ok2}& \leq M_{p}  \bigl| \lambda_{1}\bigr|^{-1-\epsilon_{1}} \cdot \bigl| \lambda_{2}\bigr|^{-1-\epsilon_{2}} , \quad \Re \lambda_{1} >\omega_{1},\ \Re \lambda_{2}>\omega_{2},
\end{align}
\begin{align}\notag 
p\Biggl( B\Bigl[ \lambda_{1}\lambda_{2} G& \bigl(\lambda_{1},\lambda_{2}\bigr)-\lambda_{2}\bigl({\mathcal L} g_{1}\bigr)\bigl( \lambda_{2} \bigr)-\bigl({\mathcal L} f_{2}\bigr)\bigl( \lambda_{1} \bigr)\Bigr]\Biggr) 
\\\label{ok21}& \leq M_{p}  \bigl| \lambda_{1}\bigr|^{-1-\epsilon_{1}} \cdot \bigl| \lambda_{2}\bigr|^{-1-\epsilon_{2}} , \quad \Re \lambda_{1} >\omega_{1},\ \Re \lambda_{2}>\omega_{2},
\end{align}
\begin{align}\notag 
p\Bigl( \lambda_{2}^{2} G& \bigl(\lambda_{1},\lambda_{2}\bigr)-\lambda_{2}\bigl({\mathcal L} f_{1}\bigr)\bigl( \lambda_{1} \bigr)-\bigl({\mathcal L} f_{3}\bigr)\bigl( \lambda_{1} \bigr)\Bigr) 
\\\label{ok3}& \leq M_{p}  \bigl| \lambda_{1}\bigr|^{-1-\epsilon_{1}} \cdot \bigl| \lambda_{2}\bigr|^{-1-\epsilon_{2}} , \quad \Re \lambda_{1} >\omega_{1},\ \Re \lambda_{2}>\omega_{2},
\end{align}
\begin{align}\notag 
p\Biggl( C\Bigl[\lambda_{2}^{2} G& \bigl(\lambda_{1},\lambda_{2}\bigr)-\lambda_{2}\bigl({\mathcal L} f_{1}\bigr)\bigl( \lambda_{1} \bigr)-\bigl({\mathcal L} f_{3}\bigr)\bigl( \lambda_{1} \bigr)\Bigr]\Biggr) 
\\\label{ok31}& \leq M_{p}  \bigl| \lambda_{1}\bigr|^{-1-\epsilon_{1}} \cdot \bigl| \lambda_{2}\bigr|^{-1-\epsilon_{2}} , \quad \Re \lambda_{1} >\omega_{1},\ \Re \lambda_{2}>\omega_{2},
\end{align}
\begin{align}\notag 
p\Bigl( \lambda_{1}G& \bigl(\lambda_{1},\lambda_{2}\bigr)-\bigl({\mathcal L}g_{1} \bigr)\bigl( \lambda_{2} \bigr)  \Bigr)
\\\label{ok4}&  \leq M_{p}  \bigl| \lambda_{1}\bigr|^{-1-\epsilon_{1}} \cdot \bigl| \lambda_{2}\bigr|^{-1-\epsilon_{2}} , \quad \Re \lambda_{1} >\omega_{1},\ \Re \lambda_{2}>\omega_{2},
\end{align}
\begin{align}
p\Biggl( D\Bigl[\lambda_{1}G&\bigl(\lambda_{1},\lambda_{2}\bigr)-\bigl({\mathcal L}g_{1} \bigr)\bigl( \lambda_{2} \bigr)  \Bigr]\Biggr) \\\label{ok41} & \leq M_{p}  \bigl| \lambda_{1}\bigr|^{-1-\epsilon_{1}} \cdot \bigl| \lambda_{2}\bigr|^{-1-\epsilon_{2}} , \quad \Re \lambda_{1} >\omega_{1},\ \Re \lambda_{2}>\omega_{2},
\end{align}
\begin{align}\notag 
p\Bigl( \lambda_{2}G& \bigl(\lambda_{1},\lambda_{2}\bigr)-\bigl({\mathcal L}f_{1} \bigr)\bigl( \lambda_{1} \bigr)  \Bigr)
\\\label{ok5}&  \leq M_{p}  \bigl| \lambda_{1}\bigr|^{-1-\epsilon_{1}} \cdot \bigl| \lambda_{2}\bigr|^{-1-\epsilon_{2}} , \quad \Re \lambda_{1} >\omega_{1},\ \Re \lambda_{2}>\omega_{2},
\end{align}
\begin{align}
p\Biggl( E\Bigl[\lambda_{2}G&\bigl(\lambda_{1},\lambda_{2}\bigr)-\bigl({\mathcal L}f_{1} \bigr)\bigl( \lambda_{1} \bigr)  \Bigr]\Biggr) \\\label{ok51} & \leq M_{p}  \bigl| \lambda_{1}\bigr|^{-1-\epsilon_{1}} \cdot \bigl| \lambda_{2}\bigr|^{-1-\epsilon_{2}} , \quad \Re \lambda_{1} >\omega_{1},\ \Re \lambda_{2}>\omega_{2} 
\end{align}
and
\begin{align}\label{ok6}
p\Bigl( FG\bigl(\lambda_{1},\lambda_{2}\bigr)\Bigr) \leq M_{p}  \bigl| \lambda_{1}\bigr|^{-1-\epsilon_{1}} \cdot \bigl| \lambda_{2}\bigr|^{-1-\epsilon_{2}} , \quad \Re \lambda_{1} >\omega_{1},\ \Re \lambda_{2}>\omega_{2}.
\end{align}
\item[(ii)] All functions whose $p$-values are considered in part \emph{(i)} are holomorphic for $\Re \lambda_{1}>\omega_{1}$ and $\Re \lambda_{2}>\omega_{2}.$
\item[(iii)] All initial values are continuous and there exists a real number $\omega \leq \min(\omega_{1},\omega_{2})$ such that for each seminorm $p\in \circledast$ there exists a real number $M_{p}>0$ such that $ p(f_{1}(t))+ p( f_{2}(t))+p(f_{3}(t))+ p(g_{1}(t))+p(g_{2}(t))\leq M_{p}\exp(\omega t),$ $t\geq 0.$
\end{itemize}
Then there exists an exponentially bounded solution $u(x,y)$ of problem \emph{[\eqref{so}-\eqref{ic}]} which satisfies that $u\in C^{2}([0,+\infty)^{2} :  X)$ and all partial derivatives $u_{xx}(x,y)$, $u_{xy}(x,y)$, $u_{yy}(x,y)$, $u_{x}(x,y)$ and $u_{y}(x,y)$ are exponentially bounded.
\end{thm}

\begin{proof}
By the complex inversion theorem and our assumptions given in \eqref{ok} and (ii), there exists an exponentially bounded, continuous function $u(x,y),$  $x\geq 0,$ $y\geq 0$ such that \eqref{smiljana} holds. Using the estimate \eqref{ok1} and (ii), it follows that there exists an exponentially bounded, continuous function $v(x,y)$, $x\geq 0,$ $y\geq 0$ such that  $$\bigl({\mathcal L v}\bigr)\bigl( \lambda_{1},\lambda_{2}\bigr)=\lambda_{1}^{2} G \bigl(\lambda_{1},\lambda_{2}\bigr)-\lambda_{1}\bigl({\mathcal L} g_{1}\bigr)\bigl( \lambda_{2} \bigr)-\bigl({\mathcal L} g_{2}\bigr)\bigl( \lambda_{2} \bigr)$$
for $\Re \lambda_{1} >\omega_{1}$ and $ \Re \lambda_{2}>\omega_{2}.$ Keeping in mind 
Proposition
\ref{lapkonv} and the uniqueness theorem for the double Laplace transform, we can prove that 
$$
\int^{x}_{0}\int^{y}_{0}(x-t)v(t,s)\, dt\, ds=u(x,y)-g_{1}(y)-xg_{2}(y),\quad x\geq 0, \ y\geq 0.
$$
This simply yields that $u_{xx}(x,y)=v(x,y)$ for $x\geq 0$ and $y\geq 0.$ By \eqref{ok11}, Lemma \ref{bnm} and Proposition \ref{piano-MLO}, it follows that $u_{xx}(x,y)\in D(A)$ for $x\geq 0$ and $y\geq 0,$ as well as that $({\mathcal L}Au_{xx}(x,y))(\lambda_{1},\lambda_{2})=A({\mathcal L}Au_{xx}(x,y))(\lambda_{1},\lambda_{2})$ for $ Re \lambda_{1} >\omega_{1}$ and $ \Re \lambda_{2}>\omega_{2} .$
Keeping in mind  the estimates \eqref{ok2}-\eqref{ok6} and (ii)-(iii), we similarly obtain that all partial derivatives $u_{xx}(x,y)$, $u_{xy}(x,y)$, $u_{yy}(x,y)$, $u_{x}(x,y)$ and $u_{y}(x,y)$ are well-defined, continuous for $x\geq 0,$ $y\geq 0$ and exponentially bounded as well as that we have the following equalities:\\ $({\mathcal L}Bu_{xy}(x,y))(\lambda_{1},\lambda_{2})=B({\mathcal L}u_{xy}(x,y))(\lambda_{1},\lambda_{2}),$\\ $({\mathcal L}Cu_{yy}(x,y))(\lambda_{1},\lambda_{2})=C({\mathcal L}u_{yy}(x,y))(\lambda_{1},\lambda_{2}),$\\
$({\mathcal L}Du_{x}(x,y))(\lambda_{1},\lambda_{2})=D({\mathcal L}u_{x }(x,y))(\lambda_{1},\lambda_{2}) ,$\\
$({\mathcal L}Eu_{y}(x,y))(\lambda_{1},\lambda_{2}) =E({\mathcal L}u_{y}(x,y))(\lambda_{1},\lambda_{2})$ and\\ $({\mathcal L}Fu(x,y))(\lambda_{1},\lambda_{2})=F({\mathcal L}u(x,y))(\lambda_{1},\lambda_{2})$ 
for $ \Re \lambda_{1} >\omega_{1}$ and $ \Re \lambda_{2}>\omega_{2} .$ Therefore, 
\begin{align*}
A\Bigl[ \lambda_{1}^{2} G&\bigl(\lambda_{1},\lambda_{2}\bigr)-\lambda_{1}\bigl({\mathcal L} g_{1}\bigr)\bigl( \lambda_{2} \bigr)-\bigl({\mathcal L} g_{2}\bigr)\bigl( \lambda_{2} \bigr)\Bigr]
\\&+B\Bigl[ \lambda_{1}\lambda_{2} G \bigl(\lambda_{1},\lambda_{2}\bigr)-\lambda_{2}\bigl({\mathcal L} g_{1}\bigr)\bigl( \lambda_{2} \bigr)-\bigl({\mathcal L} f_{2}\bigr)\bigl( \lambda_{1} \bigr)\Bigr]
\\& +C\Bigl[\lambda_{2}^{2} G \bigl(\lambda_{1},\lambda_{2}\bigr)-\lambda_{2}\bigl({\mathcal L} f_{1}\bigr)\bigl( \lambda_{1} \bigr)-\bigl({\mathcal L} f_{3}\bigr)\bigl( \lambda_{1} \bigr)\Bigr]
\\& +D\Bigl[\lambda_{1}G\bigl(\lambda_{1},\lambda_{2}\bigr)-\bigl({\mathcal L}g_{1} \bigr)\bigl( \lambda_{2} \bigr)  \Bigr]+E\Bigl[\lambda_{2}G\bigl(\lambda_{1},\lambda_{2}\bigr)-\bigl({\mathcal L}f_{1} \bigr)\bigl( \lambda_{1} \bigr)  \Bigr]\\&+FG\bigl(\lambda_{1},\lambda_{2}\bigr)=\bigl({\mathcal L} f\bigr)\bigl(\lambda_{1},\lambda_{2}\bigr)
\end{align*}
for $ \Re \lambda_{1} >\omega_{1}$ and $ \Re \lambda_{2}>\omega_{2} .$ Using the above-mentioned equalities and Proposition \ref{piano-MLO} it readily follows that \eqref{jena} holds true. 
\end{proof}

It is clear that our results can be applied in a great number of concrete situations in complex Banach spaces (concerning possible applications of Theorem \ref{nm}, we would like to stress that there exists a closed linear operaor $A$ on a Banach space of Gevrey ultradifferentiable typefunctions whose $C$-resolvent is bounded by $(1+|\lambda|)^{-1}$, if $\lambda \in {\mathbb C}$ and $|\lambda|\geq r;$ see  \cite[Section 3.3.2]{funkcionalne} for more details). We leave the reader as an interesting problem to further exemplify the results established in this section; for some applications in $E_{l}$-type spaces and their projective limits, we refer the reader to the operators and examples considered in \cite{FKP,chelj,x263}.

\section{Conclusions and final remarks}\label{finale}

In this paper, we have investigated the multidimensional vector-valued Laplace transform of functions with values in sequentially complete locally convex spaces. We have also provided some applications to the abstract Volterra integro-differential inclusions with multiple variables.

We have not been able to consider here many important questions concerning the multidimensional vector-valued Laplace transform and its applications. We close the paper with the observation that the multidimensional Widder-Arendt theorem, the uniqueness sequences for the multidimensional vector-valued Laplace transform and the approximations of multidimensional vector-valued Laplace transform will be considered somewhere else.

\end{document}